\pdfoutput=1                      

\documentclass[10 pt,a4paper]{article}

\usepackage[accsupp]{axessibility}    

\usepackage{mathrsfs,amsfonts,amsmath,comment,verbatim,amsthm,amssymb,graphicx,fancyhdr,makeidx,amscd} 
\usepackage{color}

\newtheorem{definition}{Definition}
\newtheorem{proposition}[definition]{Proposition}
\newtheorem{lemma}[definition]{Lemma}
\newtheorem{theorem}[definition]{Theorem}
\newtheorem{corollary}[definition]{Corollary}

\newtheorem{question}{Question}

\theoremstyle{definition}
\newtheorem{remark}[definition]{Remark}

\usepackage{Commands_Giada}

\begin{document}

\newcommand{\riem}{(M^n, \langle \, , \, \rangle)}
\newcommand{\Hess}{\mathrm{Hess}\, }
\newcommand{\hess}{\mathrm{hess}\, }
\newcommand{\cut}{\mathrm{cut}}
\newcommand{\ind}{\mathrm{ind}}

\newcommand{\ess}{\mathrm{ess}}

\newcommand{\longra}{\longrightarrow}

\newcommand{\eps}{\varepsilon}

\newcommand{\ra}{\rightarrow}

\newcommand{\vol}{\mathrm{vol}}

\newcommand{\di}{\mathrm{d}}

\newcommand{\R}{\mathbb R}

\newcommand{\C}{\mathbb C}

\newcommand{\Z}{\mathbb Z}

\newcommand{\N}{\mathbb N}

\newcommand{\HH}{\mathbb H}

\newcommand{\esse}{\mathbb S}

\newcommand{\bull}{\rule{2.5mm}{2.5mm}\vskip 0.5 truecm}

\newcommand{\binomio}[2]{\genfrac{}{}{0pt}{}{#1}{#2}} 

\newcommand{\metric}{\langle \, , \, \rangle}

\newcommand{\lip}{\mathrm{Lip}}

\newcommand{\loc}{\mathrm{loc}}

\newcommand{\diver}{\mathrm{div}}

\newcommand{\disp}{\displaystyle}

\newcommand{\rad}{\mathrm{rad}}

\newcommand{\inj}{\mathrm{inj}}

\newcommand{\Ric}{\mathrm{Ric}}
\newcommand{\Sec}{\mathrm{Sec}}

\newcommand{\LL}{\mathscr{L}}

\newcommand{\MAX}{\mathrm{MAX}} 

\newcommand{\mmetric}{\langle\langle \, , \, \rangle\rangle}

\newcommand{\hol}{\mathrm{H\ddot{o}l}}

\newcommand{\capac}{\mathrm{cap}}

\newcommand{\bmo}{\{b <0\}}

\newcommand{\bmuo}{\{b \le 0\}}

\newcommand{\dist}{\mathrm{dist}}

\newcommand{\MM}{\mathscr{M}}
\newcommand{\GG}{\mathscr{G}}

\newcommand{\vp}{\varphi}

\renewcommand{\div}[1]{{\mathop{\mathrm div}}\left(#1\right)}

\newcommand{\divphi}[1]{{\mathop{\mathrm div}}\bigl(\vert \nabla #1

\vert^{-1} \varphi(\vert \nabla #1 \vert)\nabla #1   \bigr)}

\newcommand{\nablaphi}[1]{\vert \nabla #1\vert^{-1}

\varphi(\vert \nabla #1 \vert)\nabla #1}

\newcommand{\modnabla}[1]{\vert \nabla #1\vert }

\newcommand{\modnablaphi}[1]{\varphi\bigl(\vert \nabla #1 \vert\bigr)

\vert \nabla #1\vert }

\newcommand{\ds}{\displaystyle}

\newcommand{\cL}{\mathcal{L}}

\newcommand{\essem}{\mathds{S}^n}

\newcommand{\erre}{\mathds{R}}

\newcommand{\errem}{\mathds{R}^n}

\newcommand{\enne}{\mathds{N}}

\newcommand{\acca}{\mathds{H}}

\newcommand{\cvett}{\Gamma(TM)}

\newcommand{\cinf}{C^{\infty}(M)}

\newcommand{\sptg}[1]{T_{#1}M}

\newcommand{\partder}[1]{\frac{\partial}{\partial {#1}}}

\newcommand{\partderf}[2]{\frac{\partial {#1}}{\partial {#2}}}

\newcommand{\ctloc}{(\mathcal{U}, \varphi)}

\newcommand{\fcoord}{x^1, \ldots, x^n}

\newcommand{\ddk}[2]{\delta_{#2}^{#1}}

\newcommand{\christ}{\Gamma_{ij}^k}

\newcommand{\supp}{\operatorname{supp}}

\newcommand{\sgn}{\operatorname{sgn}}

\newcommand{\rg}{\operatorname{rg}}

\newcommand{\inv}[1]{{#1}^{-1}}

\newcommand{\id}{\operatorname{id}}

\newcommand{\jacobi}[3]{\sq{\sq{#1,#2},#3}+\sq{\sq{#2,#3},#1}+\sq{\sq{#3,#1},#2}=0}

\newcommand{\lie}{\mathfrak{g}}

\newcommand{\wedgedot}{\wedge\cdots\wedge}

\newcommand{\rp}{\erre\mathds{P}}

\newcommand{\II}{\operatorname{II}}

\newcommand{\gradh}[1]{\nabla_{H^n}{#1}}

\newcommand{\absh}[1]{{\left|#1\right|_{H^n}}}

\newcommand{\mob}{\mathrm{M\ddot{o}b}}

\newcommand{\mab}{\mathfrak{m\ddot{o}b}}

\newcommand{\foc}{\mathrm{foc}}

\newcommand{\F}{\mathcal{F}}

\newcommand{\Cf}{\mathcal{C}_f}

\newcommand{\cutf}{\mathrm{cut}_{f}}

\newcommand{\Cn}{\mathcal{C}_n}

\newcommand{\cutn}{\mathrm{cut}_{n}}

\newcommand{\Ca}{\mathcal{C}_a}

\newcommand{\cuta}{\mathrm{cut}_{a}}

\newcommand{\cutc}{\mathrm{cut}_c}

\newcommand{\cutcf}{\mathrm{cut}_{cf}}

\newcommand{\rk}{\mathrm{rk}}

\newcommand{\crit}{\mathrm{crit}}

\newcommand{\diam}{\mathrm{diam}}

\newcommand{\haus}{\mathscr{H}}

\newcommand{\po}{\mathrm{po}}

\newcommand{\gr}{\mathcal{G}}

\newcommand{\sn}{\mathrm{sn}_H}

\newcommand{\cn}{\mathrm{cn}_H}

\newcommand{\Tr}{\mathrm{Tr}}

\newcommand{\bh}{\mathbb{B}}
\newcommand{\Gr}{\mathrm{Gr}}
\newcommand{\Pa}{\mathcal{P}}

\newcommand{\aaa}{a(|\nabla u|)}
\newcommand{\aap}{a'(|\nabla u|)}
\newcommand{\inte}{\mathrm{Int}}

\newcommand{\tcr}{\textcolor{red}}
\newcommand{\tcb}{\textcolor{blue}}
\newcommand{\tco}{\textcolor{red}}
\newcommand{\critu}{{\rm crit}(u)}

\author{Giulio Colombo${}^{(1)}$ 
	\and 
	Alberto Farina${}^{(2)}$
\and Marco Magliaro${}^{(3)}$
\and Luciano Mari${}^{(1)}$
\and Marco Rigoli${}^{(1)}$ 
}
\title{\textbf{On the classification of capillary graphs in Euclidean and non-Euclidean spaces}}
\date{}
\maketitle
\scriptsize \begin{center} (1) -- Dipartimento di Matematica ``Federigo Enriques"\\
Universit\`a degli Studi di Milano\\ 
Via C. Saldini 50, 20133 Milano (Italy).
\end{center}    
\scriptsize \begin{center} (2) -- Laboratoire 
Ami\'enois de Math\'ematique Fondamentale et Appliqu\'ee\\
UMR CNRS 7352, Universit\'e Picardie ``Jules Verne'' \\33 Rue St Leu, 80039 Amiens (France).\\ 
\end{center}
\scriptsize \begin{center} (3) -- Dipartimento di Scienza e Alta Tecnologia\\
Universit\`a degli Studi dell'Insubria\\
Via Valleggio 11, 22100 Como (Italy)
\end{center}    

\bigskip

\noindent \textbf{E-mail addresses:}  giulio.colombo@unimi.it, alberto.farina@u-picardie.fr, marco.magliaro@uninsubria.it,\\
luciano.mari@unimi.it, marco.rigoli55@gmail.com

\vspace{0.2cm}


\vspace{0.2cm}


\normalsize

\begin{abstract}
	We prove some rigidity and classification results for graphs with prescribed mean curvature and locally constant Dirichlet and Neumann data, for instance as they appear in capillarity problems. We consider domains in Riemannian manifolds, with emphasis on $\R^2$ and $\R^3$. We classify both the underlying domain and the resulting solution, providing general splitting theorems in this setting. 	
\end{abstract}	

\tableofcontents

\section{Introduction}

If a cylindrical tube in vertical position $\Omega \times (a,\infty) \subseteq \R^{n+1} = \R^n \times \R$, with $a<<-1$, is dipped into a large reservoir filled with liquid modelled by the half-space $\R^n \times (-\infty,0]$, the combination of gravity and surface tension demands that the height function $u : \overline\Omega \to \R$ reached by the liquid in the tube satisfies the capillarity equation
\[
	\diver \left( \dfrac{\nabla u}{\sqrt{1+|\nabla u|^2}} \right) - \kappa u = 0 \qquad \text{in } \, \Omega,
\]
where $\kappa > 0$ is a physical constant, $\nabla, \diver$ are the gradient and divergence in $\R^n$ and 
\[
	\MM[u] : = \diver \left( \dfrac{\nabla u}{\sqrt{1+|\nabla u|^2}} \right)
\]
is the mean curvature of the graph of the liquid profile
\[
	\GG_u \ : \ \overline\Omega \to \R^{n+1}, \qquad \GG_u(x) \doteq \big(x, u(x)\big)
\]
with respect to the upward pointing unit normal. Furthermore, $\GG_u$ must intersect the boundary of the tube with constant angle. For a detailed account on the capillarity problem, we refer the interested reader to \cite{finn}.

A classical question, motivating the present paper, asks for which domains $\Omega$ the surface of the liquid also attains a (locally) constant height on $\partial \Omega$. In this case, the angle condition rewrites as constant Neumann data for $u$. Henceforth we will consider mean curvature prescriptions more general than $\kappa u$, so given $f \in C^1(\R)$ we call a domain $\Omega$ a \emph{capillary domain} if it supports a non-constant solution $u$ to 
\begin{equation}\label{eq_PMC}
	\left\{ \begin{array}{ll}
		\MM[u]  + f(u) = 0 & \text{in } \, \Omega, \\[0.2cm]
		u, \partial_\eta u \, \text{ locally constant} & \text{on $\partial \Omega$,} 
	\end{array}\right.
\end{equation}
where $\eta$ is the interior unit normal to $\partial \Omega$. The associated graph  $\GG_u(\Omega)$ will then be called a \emph{capillary graph}. If $f(u)$ is constant  the equation describes CMC graphs, in particular minimal ones if $f(u) \equiv 0$, while the capillarity equation corresponds to the choice $f(u) = -\kappa u$. The goal of this paper is to investigate the shape of capillary graphs over $\Omega \subseteq \R^n$ or, more generally, over domains in a complete Riemannian manifold.

Hereafter, 
\begin{itemize}
\item $\{\partial_j\Omega\}$, $j \le \infty$ denotes  the collection of connected components of $\partial \Omega$;
\item $\partial_\star \Omega = \big\{ x \in \partial \Omega \ : \ \partial_\eta u(x) \neq 0\big\}$ 
\end{itemize}

In his pioneering paper on the moving plane method, Serrin \cite[Theorem 2]{serrin} treated the case where $\Omega \subseteq \R^n$ is a bounded $C^2$ domain and $u \in C^2(\overline\Omega)$ solves
\begin{equation}\label{eq_PMC_serrin}
	\left\{ \begin{array}{ll}
		\MM[u]  + f(u) = 0 & \text{in } \, \Omega, \\[0.2cm]
		u > 0 & \text{in $\Omega$.} \\[0.2cm]
		u=0, \ \partial_\eta u = {\rm const} & \text{on $\partial \Omega$}
	\end{array}\right.
\end{equation}
for some $f \in C^1(\R)$, proving that $\Omega$ must be a ball and $u$ is radially symmetric about its center. The method was then refined by Reichel \cite{reichel_1} and Sirakov \cite{sirakov} to be applied to exterior domains. The most general result to our knowledge, \cite[Theorem 2]{sirakov}, states that a $C^2$ domain $\Omega$ with bounded complement supporting a solution $u \in C^2 (\overline{\Omega})$ to
\begin{equation}\label{eq_PMC_reichel}
	\left\{ \begin{array}{ll}
		\MM[u] + f(u) = 0 & \text{in } \, \Omega, \\[0.2cm]
		u \ge 0 & \text{in } \, \Omega, \\[0.2cm]
		u(x) \to 0 & \text{as $x \to \infty$} \\[0.2cm]
		u=b_j > 0 & \text{on $\partial_j \Omega$,} \\[0.2cm]
		\partial_\eta u = c_j \le 0 & \text{on $\partial_j \Omega$,}
	\end{array}\right.
\end{equation}
for some constants $b_j,c_j$ and for $f \in \lip_\loc(\R)$ non-increasing in a neighbourhood of zero, must be the exterior of a ball, and $u$ must be radially symmetric. The paper \cite{sirakov} also contains a similar characterization for $\Omega$ a $C^2$ bounded domain with holes.
 
If $\partial \Omega$ is unbounded, or if the word ``constant'' in \eqref{eq_PMC_serrin} is replaced by ``locally constant'', the problem is widely open even for CMC graphs. In fact, even though the results in \cite{serrin,reichel_1,sirakov} hold for a large class of quasilinear operators, some of which have been extensively studied in the literature, relatively few results have appeared so far for the mean curvature operator. 
In particular, an analogue of the famous Berestycki-Caffarelli-Nirenberg conjecture \cite{bcn} for capillary graphs is still completely open. More generally, one may ask:
 
\begin{question}\label{Q1}
	Can we classify all smooth enough domains $\Omega \subseteq \R^n$ supporting a solution to \eqref{eq_PMC}? 
\end{question}
 
Among the non-constant solutions to \eqref{eq_PMC} with  $\partial \Omega$ unbounded we single out those that are one-dimensional, namely, those for which 
\[
	\Omega = (0,T) \times \R^{n-1} \qquad \text{for some } \, T \le \infty 
\]
and $u$ is 1D. More precisely, there exists an affine hyperplane $\pi$ with unit normal $w$ such that $\Omega = \{ y + tw : y \in \pi, \ t \in (0,T)\}$, and in the coordinates $(t,y)$ the function $u$ only depends on $t$. For instance, when  $f(u)$ is constant, analyzing the resulting ODE for $u$ leads to the following classification of 1D solutions:

\begin{itemize}
	\item If $f(u) \equiv 0$, then either $T = \infty$ and $\GG_u$ is a half-hyperplane, or $T< \infty$ and $\GG_u$ is a piece of a half-hyperplane.
	\item If $f(u) \equiv H \neq 0$, then $T< \infty$ and $\GG_u$ is a piece of a cylinder with circular cross-section, i.e. the graph of $u(t)$ is a piece of $\mathbb{S}^1$ of suitable radius. Hereafter, we refer to this case as ``a piece of cylinder'', leaving implicit that the cross-section is circular. 	 
\end{itemize}

Question \ref{Q1} was first raised in a paper by three of the authors \cite{cmr}, where some classification results were obtained for domains in manifolds which are globally Lipschitz epigraphs or globally Lipschitz slabs, see below. Among them, we highlight \cite[Theorem 1.3]{cmr}, the first result on Question 1 when $\partial \Omega$ is unbounded: if $\Omega \subseteq \R^2$ is a globally Lipschitz epigraph supporting a solution to 
\[
	\left\{ \begin{array}{ll}
		\MM[u] + H = 0 & \text{in } \, \Omega, \\[0.2cm]
		\inf_\Omega u > -\infty, & \text{in } \, \Omega, \\[0.2cm]
		u=0, \, \partial_\eta u = c \neq 0 & \text{on $\partial \Omega$} \\[0.2cm]
	\end{array}\right.
\]
for some $H,c \in \R$, then $\Omega$ is a half-plane, $H=0$ and $\GG_u$ is a half-plane. Below, we shall significantly improve on this result. Very recently, Lian and Sicbaldi treated the case of any $C^1$ domain $\Omega \subseteq \R^2$ with $\partial \Omega$ unbounded and connected. By \cite[Theorem 1]{lian_sicbaldi}, if $u$ solves   
\begin{equation}\label{eq_PMC_sic}
	\left\{ \begin{array}{ll}
		\MM[u] + f(u) = 0 & \text{in } \, \Omega, \\[0.2cm]
		0 < u \le b & \text{in } \, \Omega, \\[0.2cm]
		u= 0, \, \partial_\eta u = c \neq 0 & \text{on $\partial \Omega$} \\[0.2cm]
		|\nabla u| \in L^\infty(\Omega) 		
	\end{array}\right.
\end{equation}
for some $b,c \in \R$, and if $f \in C^1(\R)$ admits a primitive $F$ satisfying 
\begin{equation}\label{eq_condiprimi}
	F \le 0 \quad \text{on } \, \R, \qquad F(0) \ge \frac{1}{\sqrt{1+c^2}} - 1, 
\end{equation}
then $\Omega$ is a half-plane and $u$ is 1D. Moreover, if $f' \le 0$ on $[0,b]$ the gradient request in \eqref{eq_PMC_sic} can be dropped. They also observed that condition \eqref{eq_condiprimi} is necessary for the existence of a 1D-solution in a half-plane.

We call a 1D solution \emph{monotone} if $u(t)$ is monotone in $t$. In particular, for a monotone solution $\partial_\eta u$ has different signs in the two components of $\partial \Omega$, allowing for $\partial_\eta u \equiv 0$ in one of them.
 
It is known that monotone solutions are a particular class of \emph{stable} solutions to the equation $\MM[u] + f(u) = 0$ (see \cite {FSV}), i.e. those for which the linearized operator is non-negative in the spectral sense on $C^\infty_c(\Omega)$. Variationally, for each $\Omega' \Subset \Omega$ they correspond to stationary points of the energy
\[
	\int_{\Omega'} \sqrt{1+|\nabla \psi|^2} - \int_{\Omega'} F(\psi), \qquad F(t) \doteq \int_0^t f(s) \di s 
\]
whose second variation is non-negative. 
 
\begin{remark}
	Since the linearized mean curvature operator is elliptic, if $f'(u) \le 0$ then every solution to $\MM[u] + f(u) = 0$ is automatically stable. This applies both to the capillary case $f(u) = -\kappa u$, $\kappa >0$ and to the CMC case. 
\end{remark}

 
In this paper, we will obtain some  characterization results for stable solutions to \eqref{eq_PMC} in the direction of Question 1, new even in the case of constant $f$. We first introduce the class of domains we will be interested in. 

\begin{definition}[\textbf{Mildly $\mathbf{v}$-transverse boundary}]\label{def_mildlytransv}
	Let $\Omega \subseteq \R^n$ be a $C^1$ domain, and let $\partial^\dagger\Omega \subseteq \partial \Omega$ be the union of some connected components of $\partial \Omega$ (possibly, a single component or the entire $\partial \Omega$). Then, $\partial^\dagger \Omega$ is said to be \emph{mildly $v$-transverse} for some $v \in \mathbb{S}^{n-1}$ if $\langle \eta,v \rangle$ does not change sign on any component of $\partial^\dagger\Omega$.
\end{definition}

Note that, if $\partial^\dagger \Omega$ is mildly $v$-transverse, the product $\langle \eta,v \rangle$ is allowed to take different signs on different components of $\partial^\dagger \Omega$. Also, a mildly $v$-transverse portion of $\partial\Omega$ may have any number of connected components, even countably many. Letting $e_j = \partial_{x^j}$, and writing points $x \in \R^n$ as $x = (x^1,x')$ with $x' \in \R^{n-1}$, the following two examples have mildly $e_1$-transverse boundaries:
 
\begin{itemize}
	\item An \emph{epigraph}: a set of the type $\{x^1 > \phi(x')\}$ for some function $\phi \in C^1(\R^n)$.
	\item A \emph{slab}: a set of the type $\{\phi_1(x') < x^1 < \phi_2(x')\}$ for some functions $\phi_j \in C^1(\R^{n-1})$ with $\phi_1(x') < \phi_2(x')$ for each $x' \in \R^{n-1}$.
\end{itemize}
In what follows, an epigraph (respectively, slab) is said to be bounded, or globally Lipschitz, if so is its defining function $\phi$ (resp. $\phi_1,\phi_2$). 

\vspace{0.3cm} 

Our first main result regards capillary CMC graphs in $\R^2$. Recently, Cui in \cite{cui} made a significant advance in the minimal case by proving the following theorem: if $\Omega$ is a $C^2$ domain with connected boundary, then the graph of a solution to 
\begin{equation}\label{eq_PMC_cui}
	\left\{ \begin{array}{ll}
		\MM[u] = 0 & \text{in } \, \Omega \subseteq \R^2, \\[0.2cm]
		u > 0 & \text{in $\Omega$} \\[0.2cm]
		u=0, \ \partial_\eta u = {\rm const} & \text{on $\partial \Omega$}
	\end{array}\right.
\end{equation}
must either be a half-plane or part of a vertical catenoid (thus, $\Omega$ is either a half-space or the complement of a ball). The argument relies on a reflection technique developed by Choe \cite{choe} and deep classification results for complete, boundaryless  minimal surfaces. As such, these techniques are specific to the minimal setting and to three-dimensional ambient spaces.

%


We obtain: 

\begin{theorem}\label{teo_CMC_intro_R2}
	Let $\Omega \subseteq \R^2$ be a $C^2$ domain supporting a non-constant solution $u \in C^2(\overline\Omega)$ to 
	\[
		\left\{ 
		\begin{array}{ll}
			\MM[u]  + H = 0 & \text{on } \Omega, \\[0.2cm]
			u = b_j, \ \partial_\eta u = c_j  & \text{on $\partial_j \Omega$} \\[0.2cm]
			\inf_\Omega u > -\infty,
		\end{array} \right.
	\]
	for some $H,b_j,c_j \in \R$.
	\begin{itemize}
		\item[(i)] If $\partial \Omega$ is connected and mildly $v$-transverse for some $v$, then $\Omega$ is a half-plane and $\GG_u$ is a half-plane;
		\item[(ii)] Assume that $\partial \Omega$ is disconnected, and  that $\partial_\star \Omega$ is connected and mildly $v$-transverse for some $v$. Then, $\Omega = (0,T) \times \R$ and $\GG_u$ is a monotone piece of a cylinder.
		\item[(iii)] Assume that $\partial_\star \Omega = \partial_1\Omega \cup \partial_2\Omega$ is disconnected and mildly $v$-transverse for some $v$, with $\langle \eta,v \rangle \ge 0$ on $\partial_1\Omega$ and $\langle \eta,v \rangle \le 0$ on $\partial_2\Omega$. If $c_1c_2<0$, then $\Omega = (0,T) \times \R$ and $u$ is a monotone piece of either a cylinder or a plane.
		\item[(iv)] Assume that $\partial_\star \Omega = \partial_1\Omega \cup \partial_2\Omega$ is disconnected and that $\partial_j \Omega$ is mildly $v_j$-transverse for $j \in \{1,2\}$ and some $v_j$ (possibly with $v_1 \neq v_2$). If $\inf_\Omega u$ is attained in a set $\Gamma$ disconnecting $\Omega$, which is the case for instance if  
		\[
		\inf_\Omega u < b_j \qquad \forall \, j \ge 3 
		\]
		and $\Gamma$ has an accumulation point in $\overline{\Omega}$, then $\Omega=(0,T) \times \R$ and $\GG_u$ is a piece of cylinder.
	\end{itemize}
\end{theorem}

\noindent The figures below show potentially capillary domains to which Theorem \ref{teo_CMC_intro_R2} applies, thereby excluding them. 

\begin{figure}[ht]\label{fig_1}
\caption{Cases $(i)$ and $(ii)$}
\begin{minipage}[b]{0.45\linewidth}
\centering
\includegraphics[scale=0.052]{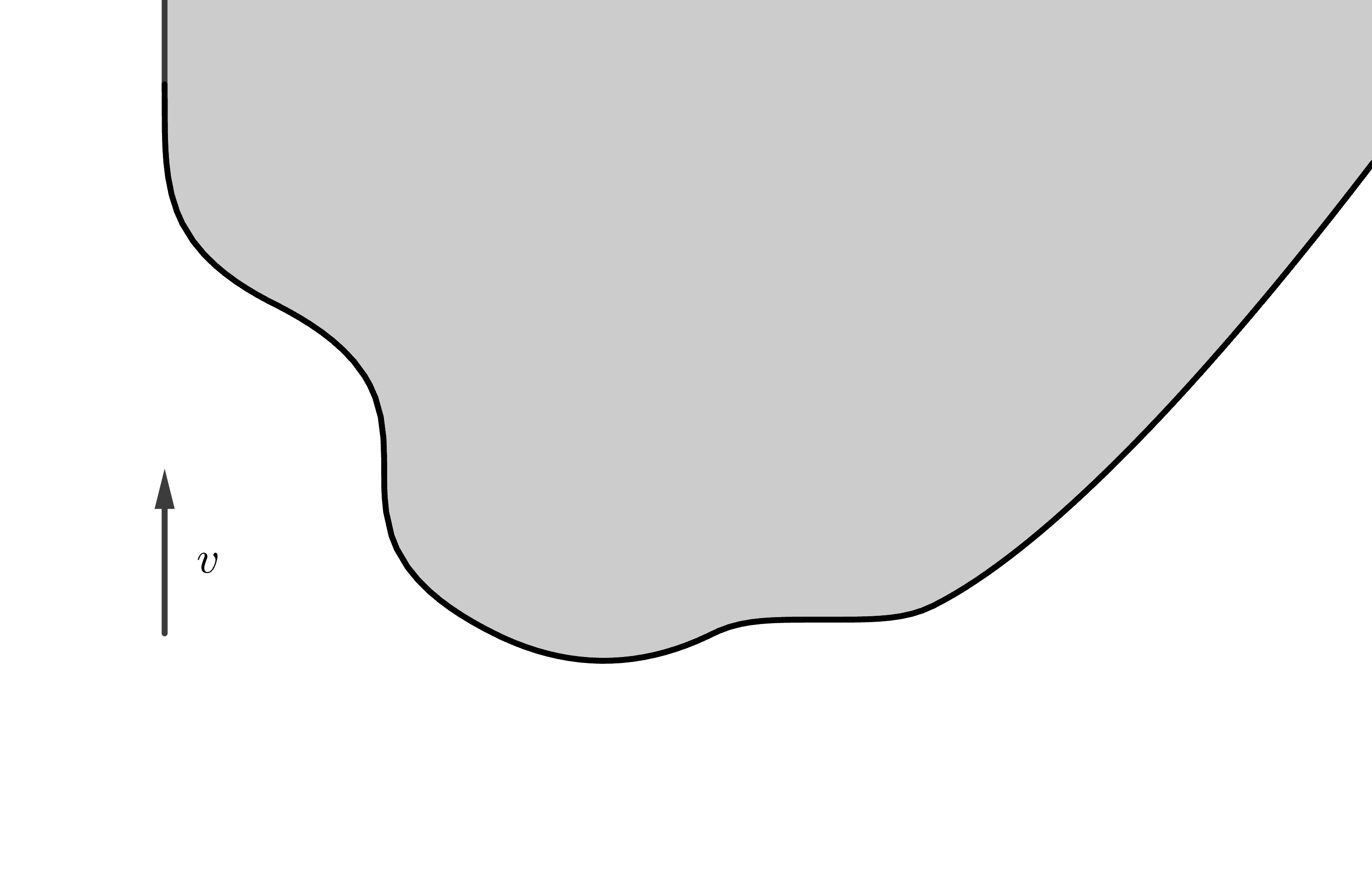}
\end{minipage}
\hspace{0.1cm}
\begin{minipage}[b]{0.45\linewidth}
\centering
\includegraphics[scale=0.051]{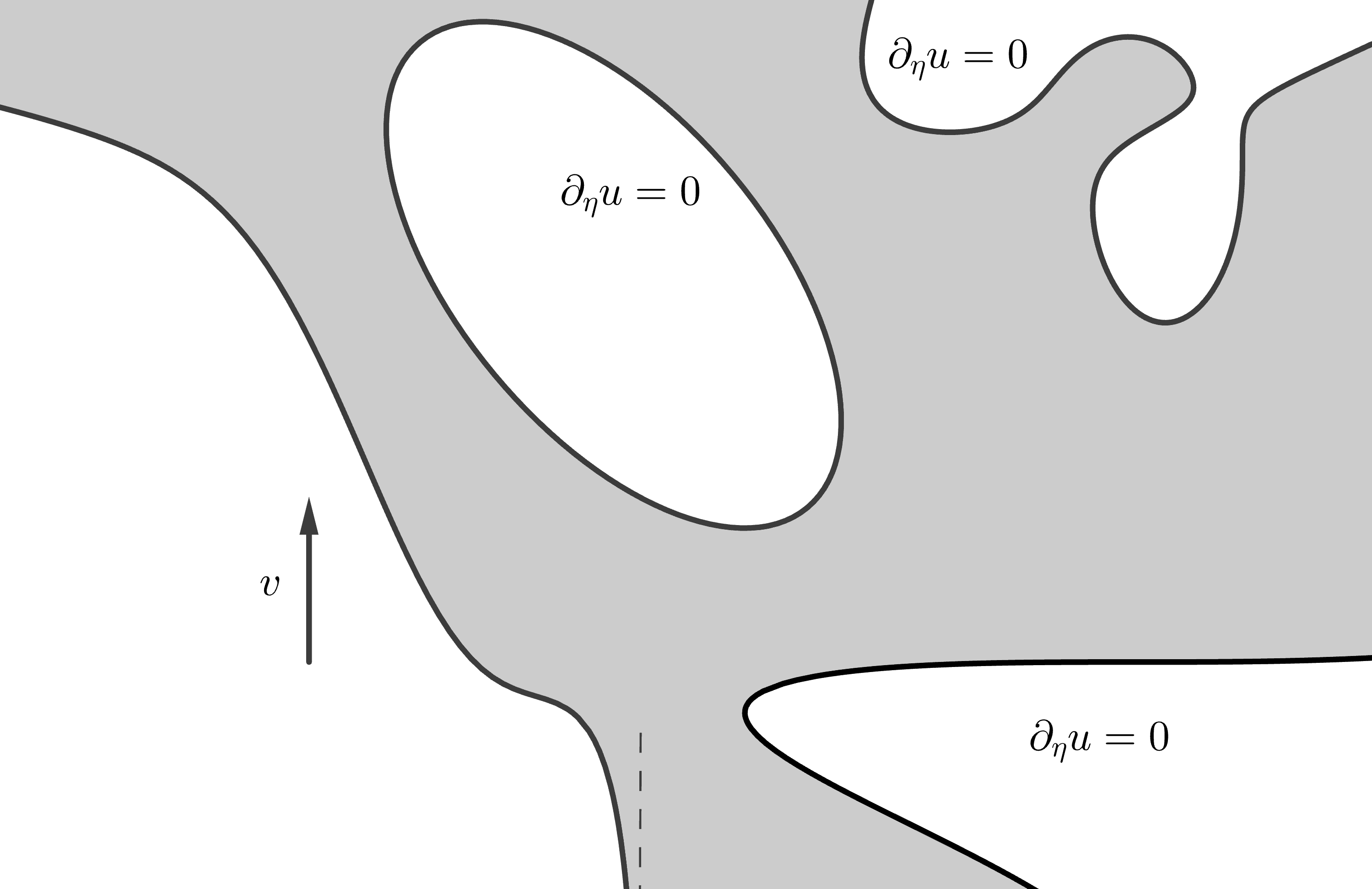}
\end{minipage}
\end{figure}

\begin{figure}[ht]\label{fig_2}
\caption{Cases $(iii)$ and $(iv)$}
\begin{minipage}[b]{0.45\linewidth}
\centering
\includegraphics[scale=0.05]{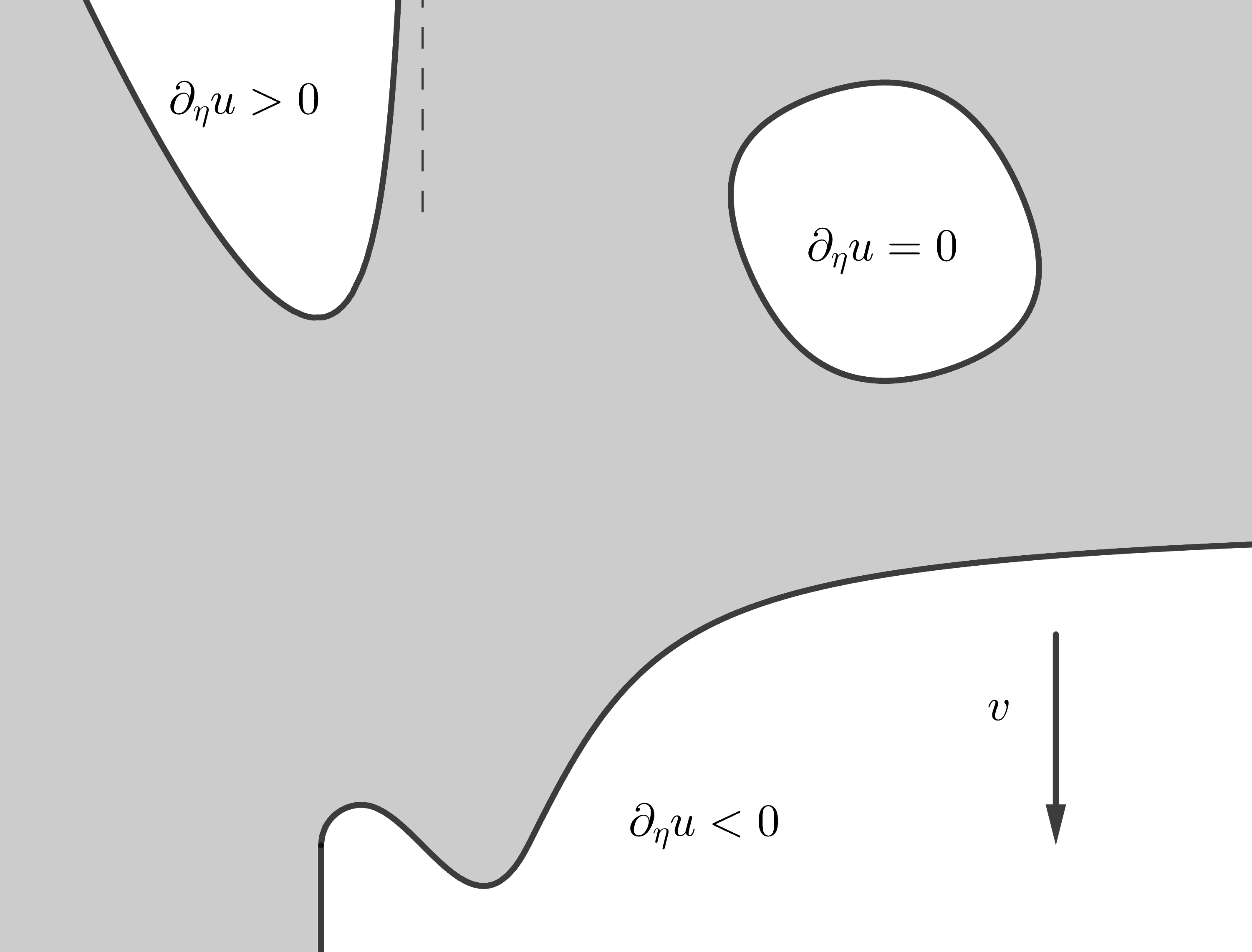}
\end{minipage}
\hspace{0.1cm}
\begin{minipage}[b]{0.45\linewidth}
\centering
\includegraphics[scale=0.054]{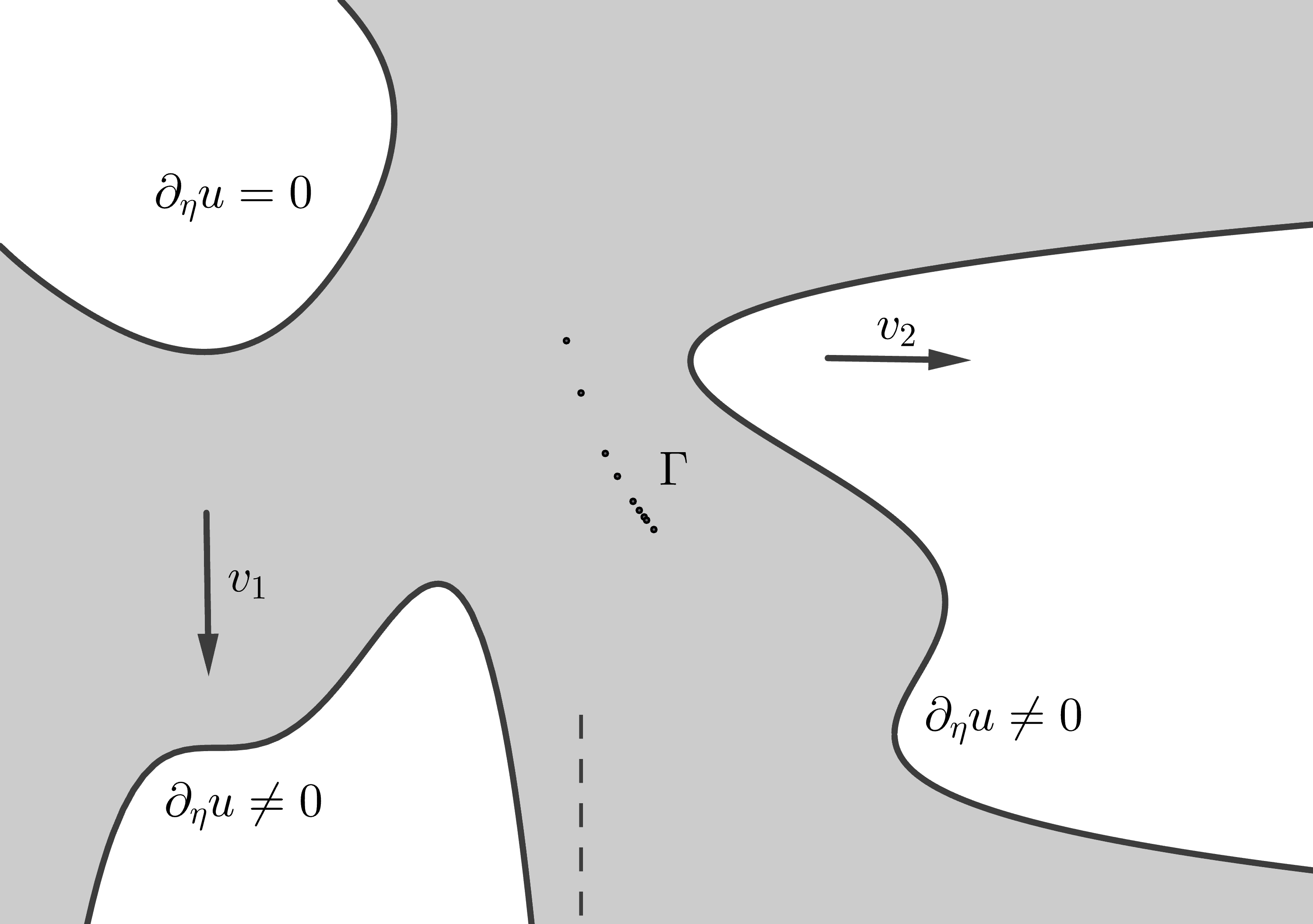}
\end{minipage}
\end{figure}

\begin{remark}
	Some remarks are in order:
	\begin{itemize}
		\item[-] Differently from most of the literature quoted above, we make no assumption on the sign of $u$ compared to that of its boundary values. In particular, if $H\equiv 0$ the conclusion in case $(i)$ does not follow from Cui's result \cite{cui}. Moreover, $u$ may possibly be unbounded above. 
		\item[-] In $(ii),(iii)$ and $(iv)$ no restriction is made a-priori on the number of connected components of $\partial \Omega$, which may possibly be infinite. 
		\item[-] Cases $(i)$ and $(iii)$ improve on \cite[Thm. 1.3]{cmr}, where the conclusion was obtained under the stronger assumptions that $\Omega$ is a globally Lipschitz epigraph (for $(i)$) or slab (for $(iii)$). 
		\item[-] Case $(iv)$ is related to the recent \cite{abm}. There, Agostiniani, Borghini and Mazzieri proved that a solution to 
		\begin{equation}\label{pb_abm}
			\left\{
			\begin{array}{ll}
				\Delta u - 1 = 0 & \text{in an annular domain }  \Omega \subseteq \R^2,  \\[0.2cm]
				u < 0 & \text{in } \Omega, \\[0.2cm]
				u=0, \ \partial_\eta u = c_1 \in \R & \text{on } \, \partial_1\Omega, \\[0.2cm]
				u=0, \ \partial_\eta u = c_2 \in \R  & \text{on } \, \partial_2\Omega
			\end{array}\right.
		\end{equation}
		must be radially symmetric provided that the set of minima of $u$ has an accumulation point. They also exhibit a non-radial annular domain $\Omega$ supporting a solution to \eqref{pb_abm} with a finite number of minimum points. This may suggest that our assumption on $\Gamma$ in $(iv)$ is necessary. We mention that an analogue of the main theorem in \cite{abm} for solutions to $\MM[u] - 1 = 0$ is yet to be established. The higher dimensional case of \eqref{pb_abm} was recently tackled in \cite{abbm}. 		
	\end{itemize}
\end{remark}

For domains $\Omega \subseteq \R^3$, we have a similar statement. 

\begin{theorem}\label{teo_CMC_intro_R3}
	Let $\Omega \subseteq \R^3$ be a $C^2$ domain satisfying
	\begin{equation}\label{eq_quadraticgrowth}
		|\Omega \cap B_R| = o(R^2 \log R) \qquad \text{as } \, R \to \infty,
	\end{equation}
	and supporting a non-constant solution $u \in C^2(\overline\Omega)$ to 
	\[
		\left\{ 
		\begin{array}{ll}
			\MM[u]  + H = 0 & \text{in } \Omega, \\[0.2cm]
			u = b_j, \ \partial_\eta u = c_j  & \text{on $\partial_j \Omega$} \\[0.2cm]
			\inf_\Omega u > -\infty,
		\end{array} \right.
	\]
	for some $H,b_j,c_j \in \R$. Then,   
	\begin{itemize}
		\item[(i)] If $\partial_1 \Omega$ is mildly $v$-transverse for some $v$ and $\langle v,\eta \rangle \not\equiv 0$ on $\partial_1 \Omega$, then $\partial \Omega$ must be disconnected;
		\item[(ii)] Assume that $\partial \Omega$ is disconnected, that $\partial_\star \Omega$ is connected and mildly $v$-transverse for some $v$, and that $\langle v,\eta \rangle \not\equiv 0$ on $\partial_\star \Omega$. Then, $\Omega = (0,T) \times \R^2$ and $\GG_u$ is a monotone piece of a cylinder.
		\item[(iii)] Assume that $\partial_\star \Omega = \partial_1\Omega \cup \partial_2\Omega$ is disconnected and mildly $v$-transverse for some $v$, with $\langle \eta,v \rangle \ge 0$ on $\partial_1\Omega$, $\langle \eta,v \rangle \le 0$ on $\partial_2\Omega$ and $\langle \eta, v \rangle \not\equiv 0$ on $\partial_\star \Omega$. If $c_1c_2<0$, then $\Omega = (0,T) \times \R^2$ and $u$ is a monotone piece of either a cylinder or a hyperplane.
		\item[(iv)] Assume that $\partial_\star \Omega = \partial_1\Omega \cup \partial_2\Omega$ is disconnected and that $\partial_j \Omega$ is mildly $v_j$-transverse for $j \in \{1,2\}$ and some $v_j$ (possibly with $v_1 \neq v_2$), with $\langle \eta, v_j \rangle \not \equiv 0$ on $\partial_j\Omega$. If $\inf_\Omega u$ is attained in a set $\Gamma$ disconnecting $\Omega$, which is the case for instance if  
		\[
		\inf_\Omega u < b_j \qquad \forall \, j \ge 3 
		\]
		and $\Gamma$ has Hausdorff measure $\haus^{2}(\Gamma)>0$, then $\Omega=(0,T) \times \R^2$ and $\GG_u$ is a piece of cylinder.
	\end{itemize}
\end{theorem}


\begin{remark}
	While not including the example of half-spaces in $\R^3$, condition \eqref{eq_quadraticgrowth} is still quite general and satisfied, for instance, if $\Omega$ is a slab  whose defining functions $\phi_1,\phi_2$ grow as follows:
	\[
		|\phi_j(x')| = o\Big( \log |x'|\Big) \qquad \text{as } \, |x'| \to \infty
	\]
	(note that no $L^\infty$-bound is imposed on the gradient of $\phi_j$). Also, mildly $e_1$-transverse regions like, for instance,
	\[
		\left\{ (x^1,x')\in \R^3 \ : \ |x'|\le 1 \ \text{or } \, |x^1| \le \frac{|x'|^2 + 1}{|x'|^2 -1} + \sqrt{\log|x'|} \right\}
	\]
	satisfy \eqref{eq_quadraticgrowth}. 
\end{remark}

\begin{remark}
	Case $(ii)$ improves on \cite[Thm 1.3]{cmr}, where the authors obtained the same conclusion by assuming that $\Omega$ is a globally Lipschitz, globally bounded slab.
\end{remark}

Theorems \ref{teo_CMC_intro_R2} and \ref{teo_CMC_intro_R3} follow from a general splitting result, which holds in a manifold setting without dimensional restriction and for general sources $f$. Hereafter, $(M^n,\metric)$ will be a complete Riemannian manifold of dimension $n \ge 2$ with Ricci curvature $\Ric$ and sectional curvature $\Sec$. Having fixed an origin $o \in M$, we denote by $B_r$ the geodesic ball of radius $r$ centered at $o$. Hereafter, a vector field $X$ of class $C^1$ in $\overline{\Omega}$ will be said to be Killing in $\overline{\Omega}$ if the Killing condition $\mathscr{L}_X g = 0$ is satisfied pointwise there.

\begin{theorem}\label{teo_main_MC_manifolds}
	Let $\Omega \subseteq M^n$ be a $C^2$ domain satisfying
	\[
		\Ric \ge 0 \quad \text{in } \, \Omega, \qquad |\Omega \cap B_R| = o(R^2 \log R) \qquad \text{as } \, R \to \infty
	\]
	and supporting a non-constant, stable solution $u \in C^3(\Omega) \cap C^2(\overline\Omega)$ of 
	\begin{equation}\label{eq_PMC_cmc2}
		\left\{ 
		\begin{array}{ll}
			\MM[u]  + f(u) = 0 & \quad\text{in } \Omega, \\[0.2cm]
			u = b_j, \ \partial_\eta u = c_j  & \quad \text{on $\partial_j \Omega$} \\[0.2cm]
			\inf_\Omega u > -\infty
		\end{array} \right.
	\end{equation}
	for some $b_j,c_j \in \R$ and $f \in C^1(\R)$. Assume that one of the following sets of assumptions is met for some constant $C_0$:
	\begin{itemize}
		\item[(A)] $- C_0 \le f(u) \le 0$ in $\Omega$ and $u$ is constant on $\partial_\star \Omega$.
		\item[(B)] $- C_0 \le f(u) \le 0$ in $\Omega$, $u$ is bounded on $\partial_\star \Omega$ and 
		\[
		\haus^{n-1}(\partial_\star^{b}\Omega \cap B_R)= o \big(R^2 \log R\big) \qquad \text{as } \, R \to \infty
		\]
		for some $b \in \R$, where $\partial_\star^b\Omega=\{ x \in \partial_\star \Omega : u(x) \neq b\}$.
		\item[(C)] $u$ is constant on $\partial_\star \Omega$ and bounded in $\Omega$.
		\item[(D)] $|f(u)| \leq C_0$ and $f'(u) \le 0$ in $\Omega$, $\sup_j |c_j| < \infty$ and $\Sec \ge - \kappa$ in $M$ for some $\kappa \in \R^+$. 
	\end{itemize}			
%
%
%
%
%
	If there exists a Killing field $X$ on $\overline\Omega$ with $|X| \in L^\infty(\Omega)$ and  
	\begin{equation}\label{condi_Hcj}
		\begin{array}{l}
			\disp c_j \langle \eta, X \rangle \ge 0 \qquad \text{for each $j$,} \\[0.2cm]
			\disp c_j \langle \eta, X \rangle \not \equiv 0 \qquad \text{for some $j$,}
		\end{array} 
	\end{equation}
	Then:
	\begin{itemize}
		\item[i)] $\Omega = (0,T) \times N$ with the product metric, for some $T \le \infty$ and some complete, totally geodesic hypersurface $N \subseteq M$. 
		\item[ii)] $u= u(t)$ is 1D and strictly monotone in the arclength $t \in (0,T)$, 
		\item[iii)] $\langle \partial_t, X \rangle$ is constant in $\Omega$.
	\end{itemize}
\end{theorem}

\begin{remark}
	\item[-] Again, no growth from above is imposed on $u$, and no restriction is made on the number of components of $\partial \Omega$. 

%
	%
	\item[-] The $t$-direction in the splitting $\Omega = (0,T) \times \R^{n-1}$ may not coincide with $X$. 
	\item[-] Condition \eqref{condi_Hcj}, automatically satisfied by monotone solutions with $X = \pm \partial_t$, implies the request that $\partial_\star \Omega$ be mildly $X$-transverse. We stress that \eqref{condi_Hcj} only depends on $X$ and on the boundary conditions imposed on $u$, and is therefore easily checkable in explicit examples.
	\item[-] The actual existence of monotone solutions depends on the properties of $f$, and should be examined case by case (for solutions on $\R$, see \cite{FSV}). Since our methods do not require the a priori knowledge of such solutions, we have not performed the 1D analysis here. Observe that, in case no monotone 1D solutions exist, Theorems \ref{teo_main_MC_manifolds} and \ref{teo_splitting_intro} can be seen as non-existence results. 	
	%
\end{remark}

\begin{remark}
	Conditions $(A)$ to $(D)$ are instrumental to prove that $u$ has \emph{moderate energy growth}, a crucial property to obtain Theorem \ref{teo_main_MC_manifolds}, see Section \ref{sec_energy} for more details. Indeed, further conditions are shown to imply moderate energy growth, and we refer to Theorem \ref{teo_goodcutoff_MC} for a more extensive list. We have  selected $(A)-(D)$ for the sake of simplicity. 	
\end{remark}

\begin{remark}
	A closely related result was obtained in \cite[Theorem 1.6]{cmr} for constant $f$. In fact, the assumption that $f$ be constant is essential for the methods therein to work. However, even in this case Theorem \ref{teo_main_MC_manifolds} improves on \cite{cmr} in some aspects, including the fact that the set $\{c_j\}$ is not a priori required to be bounded in cases $(A)$, $(B)$ and $(C)$.
\end{remark}

In view of Theorem \ref{teo_main_MC_manifolds}, Theorems \ref{teo_CMC_intro_R2} and \ref{teo_CMC_intro_R3} allow for extension to  stable solutions to $\MM[u] + f(u) = 0$. We only state the result in $\R^2$.

\begin{theorem}\label{teo_CMC_intro_R2_gen}
	Let $\Omega \subseteq \R^2$ be a $C^2$ domain supporting a non-constant stable solution $u \in C^3(\Omega) \cap C^2(\overline\Omega)$ to 
	\[
		\left\{ 
		\begin{array}{ll}
			\MM[u]  + f(u) = 0 & \text{on } \Omega, \\[0.2cm]
			u = b_j, \ \partial_\eta u = c_j  & \text{on $\partial_j \Omega$} \\[0.2cm]
			\inf_\Omega u > -\infty,
		\end{array} \right.
	\]
	for some $b_j,c_j \in \R$ and $f \in C^1(\R)$. Consider the following assumptions for some constant $C_0$:
	\begin{itemize}
		\item[(A)] $- C_0 \le f(u) \le 0$ in $\Omega$; 
		\item[(B)] $- C_0 \le f(u) \le 0$ in $\Omega$ and $\haus^{1}(\partial_1\Omega \cap B_R)= o \big(R^2 \log R\big)$ as $R \to \infty$;
		\item[(C)] $u$ is bounded in $\Omega$; 
		\item[(D)] $|f(u)| \leq C_0$ and $f'(u) \le 0$ in $\Omega$.
	\end{itemize}			
	Then, 
	\begin{itemize}
		\item[(i)] Assume one among $(A),(C)$ and $(D)$.\\
		If $\partial \Omega$ is connected and mildly $v$-transverse for some $v$, then $\Omega$ is a half-plane and $u$ is a $1D$, monotone solution;
		\item[(ii)] Assume one among $(A),(C)$ and $(D)$.\\
		If $\partial \Omega$ is disconnected, and $\partial_\star \Omega$ is connected and mildly $v$-transverse for some $v$, then $\Omega = (0,T) \times \R$ and $u$ is a $1D$, monotone solution.
		\item[(iii)] Assume one between $(B)$ and $(D)$.\\
		Assume that $\partial_\star \Omega = \partial_1\Omega \cup \partial_2\Omega$ is disconnected and mildly $v$-transverse for some $v$, with $\langle \eta,v \rangle \ge 0$ on $\partial_1\Omega$ and $\langle \eta,v \rangle \le 0$ on $\partial_2\Omega$. If $c_1c_2<0$, then $\Omega = (0,T) \times \R$ and $u$ is a $1D$, monotone solution.	
		\item[(iv)] Assume one among $(A),(C)$ and $(D)$.\\
		Assume that $\partial_\star \Omega = \partial_1\Omega \cup \partial_2\Omega$ is disconnected, and that $\partial_j \Omega$ is mildly $v_j$-transverse for $j \in \{1,2\}$ and some $v_j$ (possibly with $v_1 \neq v_2$). If $\inf_\Omega u$ is attained in a set $\Gamma$ disconnecting $\Omega$, which is the case for instance if all of the following are satisfied: 
		\begin{itemize}
			\item $\inf_\Omega u < b_j$ for each $j \ge 3$;
			\item $f$ is analytic in a neighbourhood of $\inf_\Omega u$ and $f(\inf_\Omega u) \neq 0$, 
			\item $\Gamma$ has an accumulation point in $\overline\Omega$,
		\end{itemize}	
		then $\Omega=(0,T) \times \R$ and $u$ is a $1D$ solution.
	\end{itemize}
\end{theorem}

The domains depicted in Figures \ref{fig_1} and 2 above provide examples to which Theorem \ref{teo_CMC_intro_R2_gen} applies.

\begin{remark}
	Observe that the capillarity case $f(u) = - \kappa u$ is included by $(D)$. Moreover, because of $(C)$ Theorem \ref{teo_CMC_intro_R2_gen} allows to treat the pairs $(u,f)$ considered by Lian and Sicbaldi \cite{lian_sicbaldi}, providing complementary classification results when $\partial \Omega$ is disconnected.
\end{remark}

We conclude this introduction by describing the main novelties of our work. In the generality of Theorem \ref{teo_main_MC_manifolds} the manifold does not possess enough symmetries to exploit the moving plane method as in \cite{serrin,reichel_1,sirakov,lian_sicbaldi}; indeed, even in our applications to $\R^2$ and $\R^3$ the method is not used. The strategy is instead based on some integral formulas relating $u$ to the function $\langle\nabla u,X\rangle$, which is a particular solution of the linearized problem. This approach follows \cite{cmr}, itself inspired by the integral inequalities obtained in \cite{Arma, Arma2} for the semilinear equation $\Delta u+f(u)=0$ in $\R^n$ and extended to manifolds in \cite{FarMarVal}. In the latter, the authors proved that if $u$ is monotone in the direction of a bounded Killing field $X$ in $\Omega$, then
\begin{equation}\label{eq_poica_lapla}
\int_\Omega\left\{|\nabla u|^2 |\II|^2 +\left|\nabla^\top|\nabla u|\right|^2+\Ric(\nabla u,\nabla u)\right\}\varphi^2\le\int_\Omega|\nabla\varphi|^2|\nabla u|^2 
\end{equation}
holds for each $\varphi\in\lip_c(\overline\Omega)$, where $\nabla^\top$ and $\II$ are, respectively, the tangential gradient and the second fundamental form of the level sets of $u$. In particular, test functions $\varphi$ whose support intersects $\partial\Omega$ are allowed. In fact, by a remarkable cancellation, the overdetermined boundary conditions force the boundary term appearing in the computations leading to \eqref{eq_poica_lapla} to vanish identically \emph{provided that} $\langle \nabla u, X \rangle$ has a sign in $\Omega$, say it is positive there. Testing \eqref{eq_poica_lapla} with suitable cut-off functions $\varphi$ and assuming $\Ric \ge 0$, one is able to deduce that 
\begin{equation}\label{eq_split}
|\nabla u|^2 |\II|^2 +\left|\nabla^\top|\nabla u|\right|^2 \equiv 0 \qquad \text{in } \, \Omega
\end{equation}
whenever the energy of $u$ satisfies 
\begin{equation}\label{eq_bounden_u}
\int_{B_R}|\nabla u|^2 = O \left( R^2 \log R \right) \qquad \text{as } \, R \to \infty 
\end{equation}
for balls $B_R$ centered at a fixed origin. Property \eqref{eq_split} implies that $\Omega$ splits and $u$ is 1D. In the present paper, we first obtain a new integral inequality for monotone solutions in the direction of $X$, extending \eqref{eq_poica_lapla} and valid for  all quasilinear equations of the form
\begin{equation}\label{eq_alapla_gen}
\diver \big( a(|\nabla u|)\nabla u \big) +f(u) = 0
\end{equation}
under mild regularity assumptions on $a$ (essentially, those guaranteeing locally uniform ellipticity even at critical points of $u$), see Theorem \ref{teo_Poincare_loc}. Clearly, to get rigidity, condition \eqref{eq_bounden_u} shall be replaced by a growth tailored to $a$: for operators of mean curvature type and under suitable conditions on $f$ and $u_{|\partial_\star \Omega}$, by a calibration argument,  such a growth is granted by a mild volume growth condition on $\Omega$ and (in some cases) $\partial_\star \Omega$, see Theorem \ref{teo_goodcutoff_MC}.

A key point is therefore to show that $w \doteq \langle \nabla u, X \rangle > 0$ in $\Omega$. To the best of our knowledge, proofs of the monotonicity of $u$ in the direction of a Killing field (usually, a parallel field in $\R^n$) are currently available only for certain classes of nonlinearities $f$ and non-compact domains $\Omega$, and typically rely on the moving plane and/or the sliding method (see for instance \cite{bcn,FarMarVal} and the references therein). These techniques often require  \emph{global} gradient or H\"older estimates on $u$ to be applied. For the mean curvature operator such estimates turn out to be subtle, especially when $f$ is not monotone or under just lower bounds on the Ricci curvature of $M$, see \cite{cmmr,cmr}. In \cite{cmr}, we proposed a new approach to get monotonicity for CMC capillary graphs, by combining pointwise gradient estimates and parabolicity arguments. We here complement \cite{cmr} with a new technique depending on energy estimates.
The method is quite versatile, and allows to consider much more general classes of nonlinearities $f$. In particular, we do not  need $f' \le 0$. Observe that request \eqref{condi_Hcj} rephrases as $w \ge 0$ on $\partial \Omega$, a necessary condition for monotonicity, and $w \not\equiv 0$ on $\partial \Omega$, the latter just to avoid a pathological case (which can be handled in the case of $\R^2$, see Section \ref{sec_critical}).

The solutions in Theorem \ref{teo_main_MC_manifolds} are monotone, and so are those in all items in Theorems \ref{teo_CMC_intro_R2}, \ref{teo_CMC_intro_R3} and \ref{teo_CMC_intro_R2_gen} but $(iv)$. To obtain $(iv)$, where 1D solutions are not monotone, one needs to separately treat each connected component $\Omega_0$ of $\Omega \backslash \Gamma$, prove rigidity for $\Omega_0$ and then glue the resulting 1D solutions. This requires a careful analysis, carried out in Proposition \ref{prop_splitting_alt}, Lemma \ref{lem_Du0} and Theorem \ref{teo_twopieces}, of independent interest.

A further feature to point out is that our techniques are not specific to the mean curvature operator, but allow to treat a large class of quasilinear equations under mild structural conditions, as described in the next section. Besides allowing for a greater generality, such an extension simplifies the resulting integral formulas making the role of the linearized operator of \eqref{eq_alapla_gen} and its eigenvalues more transparent.
%
%
%
%
%

\subsection{The general setting}

Let $(M^n,\metric)$ be a complete Riemannian manifold of dimension $n \ge 2$, $f \in C^1(\R)$ and consider the quasilinear equation 
\begin{equation}\label{equazu}
	\Delta_a u + f(u)=0 \qquad \text{on } \Omega \subseteq M,
\end{equation}
where
\begin{equation}\label{def_T}
	\Delta_a u = \diver\big(a(|\nabla u|)\nabla u\big) 
\end{equation}
and $a$ satisfies
\begin{equation}\label{assu_A}
	\begin{array}{rcl}
		 & \quad a(t) \in C^1(\R^+) \cap C(\R^+_0). \\[0.1cm]
		& \quad a(t)>0 \ \  \text{ for } t \in \R^+, \qquad \big(ta(t) \big)' >0 \ \ \text{ for } t \in \R^+,
	\end{array}
\end{equation}
For instance: 
\begin{itemize}
	\item[(a)] the mean curvature operator is obtained for the choice $a(t) = \frac{1}{\sqrt{1+t^2}}$, 
	\item[(b)] the $p$-Laplacian $\Delta_p$ for $p \in [2,\infty)$ is obtained with $a(t) = t^{p-2}$. More generally, for $2 \le p < q < \infty$ the $(p,q)$-Laplacian, introduced in  \cite{zhikov,marcellini}, is obtained with the choice $a(t) = t^{p-2} + t^{q-2}$. 
	\item[(c)] the operator of exponentially harmonic functions, first introduced in \cite{duc_eels, eels_lemaire}, is obtained for $a(t)= e^{t^2}$.
\end{itemize}

More examples can be found in \cite{serrin2}. In most of our results we shall restrict to operators satisfying
\begin{equation}\label{assu_A_strong}
a \in C^1(\R^+_0), \qquad a(t), \big( ta(t)\big)' > 0 \quad \text{for } \, t \in \R^+_0
\end{equation}
or the stronger
\begin{equation}\label{assu_A_superstrong}
	a\in C^{1,1}_\loc(\R_0^+), \qquad a(t), \big( ta(t)\big)' > 0 \quad \text{for } \, t \in \R^+_0.
\end{equation}
In particular, referring to the above examples, (a) and (c) satisfy \eqref{assu_A_superstrong}, as well as the standard $2$-Laplacian, while $\Delta_{2,q}$ in (b) satisfies \eqref{assu_A} for each $q>2$, \eqref{assu_A_strong} for $q \ge 3$ and \eqref{assu_A_superstrong} for $q \in \{3\} \cup [4,\infty)$. Further interesting operators treatable with the techniques below include that of the polytropic flow from gas dynamics, see \cite{sibner_sibner}, for which   
\[
a(t) = \left( 1 - \frac{\gamma-1}{2}t^2 \right)^\frac{1}{\gamma-1}, \qquad \gamma >1.
\]
Here, $a$ represents the density of a fluid whose velocity is $\nabla u$, and \eqref{assu_A_strong} is met if $t$ restricts to the interval 
\[
t \in (0, \upsilon^*), \qquad \upsilon^* \doteq \sqrt{ \frac{2}{\gamma+1}},
\]
which identifies the subsonic regime. Hence, our results are applicable provided that $\|\nabla u\|_\infty < \upsilon^*$.

The formal linearization of $\Delta_a$ at $u$ is
\begin{equation}\label{def_linearized}
\left.\frac{\di}{\di s} \right|_{s=0} \Delta_a(u+s\varphi) =  \diver\big(A(\nabla u) \nabla \varphi\big),
\end{equation}
where $A : \Gamma(TM\backslash \{0\}) \ra \Gamma(S^{(1,1)}(M))$ is the field of endomorphisms of $TM$ given by 
\begin{equation}\label{def_A}
	A(X) = \frac{a'(|X|)}{|X|} X_\flat \otimes X + a(|X|) {\rm Id},
\end{equation}
and $\flat$ is the musical isomorphism induced by $\metric$. Note that the eigenvalues of $A(X)$ are $\{\lambda_j(|X|)\}_{j=1}^n$ with
\begin{equation}\label{eigenvalues}
	\lambda_1(t) = a(t) + ta'(t) = \big(t a(t)\big)' \qquad \lambda_2(t) = \ldots = \lambda_n(t) = a(t),
\end{equation}
hence the second in \eqref{assu_A_strong} rephrases as $\lambda_1(t),\lambda_2(t) > 0$ on $\R^+_0$, equivalently $A(X)$ is elliptic, nonsingular and non-degenerate. Henceforth, we write
\[
\lambda_{\max}(t) = \max \big\{ \lambda_1(t), \lambda_2(t) \big\}.
\]
For instance, for the mean curvature operator, 
\[
\lambda_1(t) = (1+t^2)^{-3/2}, \qquad \lambda_2(t) = (1+t^2)^{-1/2} = \lambda_{\max}(t).
\]

\begin{definition}
	Let $a$ satisfy \eqref{assu_A_strong}. A function $u \in C^1(\Omega)$ weakly solving \eqref{equazu} is said to be \emph{stable} if
	\begin{equation}\label{hardy}
		\int_{M} f'(u) \varphi^2\di x \le \int_M \langle A(\nabla u)\nabla\varphi, \nabla \varphi\rangle \di x \qquad \forall \, \varphi \in \lip_c(\Omega).
	\end{equation}
\end{definition}

Here is our first main theorem for general operators: 

\begin{theorem} \label{teo_splitting_intro}
	Let $(M^n, \metric)$ be a complete Riemannian manifold. Let $\Omega$ be a $C^2$ domain with interior normal $\eta$, and assume that $\Ric\geq 0$ in $\Omega$. Let $a$ satisfy \eqref{assu_A_superstrong}, $f \in C^1(\R)$ and let $u \in C^3(\Omega)\cap C^2(\overline{\Omega})$ be a non-constant, stable solution to
	\begin{equation}\label{equazu_conbordo}
	\left\{
	\begin{array}{l@{\qquad}l}
		\Delta_a u + f(u) = 0 & \text{in } \Omega \\[0.3cm]
		u = b_j, \ \partial_\eta u = c_j  & \text{on $\partial_j \Omega$}. 
	\end{array}
	\right.
	\end{equation}
	If there exists a Killing field $X$ on $\overline\Omega$ with $|X| \in L^\infty(\Omega)$ and  
	\begin{equation}\label{condi_Hcj_con_a}
		\begin{array}{l}
			\disp c_j \langle \eta, X \rangle \ge 0 \qquad \text{for each $j$,} \\[0.2cm]
			\disp c_j \langle \eta, X \rangle \not \equiv 0 \qquad \text{for some $j$,}
		\end{array} 
	\end{equation}
	and 
	\begin{equation}\label{eq_lowenergy}
			\int_{\Omega \cap B_R} |\nabla u|^2 \lambda_{\max}(|\nabla u|) = O(R^2 \log R) \qquad \text{as } \, R \to \infty \, ,
	\end{equation}
	then: 
	\begin{itemize} 
		\item[i)] $\Omega = (0,T) \times N$ with the product metric, for some $T \le \infty$ and some complete, totally geodesic hypersurface $N \subseteq M$. 
		\item[ii)] $u= u(t)$ is 1D and strictly monotone in the arclength $t \in (0,T)$, 
		\item[iii)] $\langle \partial_t, X \rangle$ is constant in $\Omega$.
	\end{itemize}
	%
	%
\end{theorem}

For operators of mean curvature type, that is, satisfying
\[
t^2 \lambda_2(t) \le C_1 \lambda_1(t) \le \frac{C_2}{t} \qquad \forall \, t \in \R^+
\]
for some positive constants $C_1,C_2$, condition \eqref{eq_lowenergy} can be replaced by 
\[
|\Omega \cap B_R| = o\big( R^2 \log R \big) \qquad \text{as } \, R \to \infty, 
\]
provided that any of assumptions $(i)$ to $(vi)$ in Theorem \ref{teo_goodcutoff_MC} are satisfied in $\Omega_0 = \Omega$.

\vspace{0.5cm}

The paper is structured as follows: in Section \ref{sec_energy} we introduce the moderate energy growth assumption needed for our main theorems, and obtain some sufficient conditions for its validity. Among them, we stress the bound on the growth of $|\Omega \cap B_R|$ for operators of mean curvature type in Theorem \ref{teo_goodcutoff_MC}. In Section \ref{sec_stability} we study the stability condition and its consequences for solutions with moderate energy growth. The main result is Theorem \ref{teo_monotonicity}, which guarantees the monotonicity of $u$ in the direction of a Killing field $X$ under mild conditions. The goal of Section \ref{sec_poinc} is to prove the key integral formula for monotone solutions to \eqref{equazu_conbordo}, see Theorem \ref{teo_Poincare_loc} and the related splitting results including Theorems \ref{teo_main_MC_manifolds} and \ref{teo_splitting_intro}. In Section \ref{sec_critical} we will focus on Euclidean space, analysing the set of critical points of $u$, and finally prove Theorems \ref{teo_CMC_intro_R2}, \ref{teo_CMC_intro_R3} and \ref{teo_CMC_intro_R2_gen}.
 
\vspace{0.5cm}

\noindent \textbf{Acknowledgements:} A.F. was partially supported by ANR EINSTEIN CONSTRAINTS:
past, present, and future. Research project ANR-23-CE40-0010-02. L.M. and M.R. were supported by the PRIN project no. 20225J97H5 ``Differential-geometric aspects of manifolds via Global Analysis''. 

\vspace{0.2cm}

\noindent \textbf{Conflict of Interests.} The authors have no conflict of interest.
\vspace{0.2cm}

\noindent \textbf{Data availability statement.} No data was generated by this research.

\section{Moderate energy growth}\label{sec_energy}

Let $\Omega \subseteq M$ be an open subset with $C^1$ boundary, and $u \in C^1(\overline\Omega)$ be a weak solution to 
\[
\Delta_a u + f(u) = 0,
\]
where $a$ satisfies \eqref{assu_A} and $f \in C(\R)$. 
One of the core steps to achieve our classification results is to prove that $u$ has moderate energy growth, according to the next

\begin{definition}
	Let $\Omega_0\subseteq\Omega$ be an open set. We say that $u$ has \emph{moderate energy growth in $\Omega_0$} if there exists a sequence of functions $\{\varphi_j\} \subseteq \lip_c(\overline{\Omega}_0)$ such that, as $j \to \infty$, 
	\[
		0 \le \varphi_j \to 1 \quad \text{in } \, C_\loc(\overline{\Omega}_0), \qquad \int_{\Omega_0} |\nabla u|^2 \langle A(\nabla u) \nabla \varphi_j, \nabla \varphi_j \rangle \to 0 .
	\]
	If the condition is satisfied for $\Omega_0 = \Omega$ we simply say that $u$ has moderate energy growth.
\end{definition}

In this section, we describe some sufficient conditions for this property to hold. Hereafter, we consider balls $B_R$ centered at a fixed origin $o \in M$ and we define $r(x) = {\rm dist}(x,o)$. We begin with the following extension of the well-known ``logarithmic cut-off trick'', which is related to an argument by Karp \cite{karp}.

\begin{lemma} \label{lem_phij}
	Let $(M^n,\metric)$ be a complete Riemannian manifold and $0\leq f\in L^1_\loc(M)$. If
	\[
	\int_{B_R} f \leq R^2 \, \theta(R) \qquad \text{for all } \, R > 0
	\]
	for some non-decreasing, $C^1$ function $\theta : [R_1,\infty) \to \R^+$ such that
	\[
	\int_{R_1}^\infty \frac{\di s}{s \theta(s)} = \infty
	\]
	then there exists a sequence $\{\varphi_j\} \subseteq \lip_c(M)$ satisfying
	\[
	\varphi_j \to 1 \quad \text{in} \quad W^{1,\infty}_\loc(M) , \qquad \int_M f |\nabla\varphi_j|^2 \to 0
	\]
	as $j\to\infty$.
\end{lemma}

\begin{remark}
	A typical example of function $\theta(R)$ satisfying the above assumptions, that we will hereafter use in the paper, is $\theta(R) = \log R$.	
\end{remark}

\begin{proof}
	We define $\Theta(s) = s\theta(s)$ and for each $r\geq 1$ we set
	\[
	\psi(t) = \int_{R_1}^t \frac{\di s}{\Theta(s)} \, .
	\]
	The function $\psi : [R_1,\infty) \to \R^+_0$ is $C^2$, strictly increasing, strictly positive on $(R_1,\infty)$ and such that $\psi(t) \to \infty$ as $t\to\infty$. Let $\{R_j\} \subseteq [R_1,\infty)$ be a diverging sequence. For each $j\in\N$ let $r_j > R_j$ be the real number such that $\psi(r_j) = 2\psi(R_j)$ and define
	\[
	\psi_j(t) = \begin{cases}
		1 & \text{if } \, 0 < t \leq R_j \\
		2\left(1-\dfrac{\psi(t)}{\psi(r_j)}\right) & \text{if } \, R_j < t \leq r_j \\
		0 & \text{if } \, t > r_j \, .
	\end{cases}
	\]
	Then for $R_j < t < r_j$ we have
	\[
	\psi_j'(t) = -\frac{2}{\psi(r_j)} \frac{1}{\Theta(t)} \, , \qquad \psi_j''(t) = \frac{2}{\psi(r_j)} \frac{\Theta'(t)}{\Theta(t)^2} \, .
	\]
	In particular, by our assumption on $\theta$ we have $\psi_j'\leq 0$ and $\psi_j''\geq 0$ on $(R_j,r_j)$. 
	For each $j\in\N$ define $\varphi_j = \psi_j\circ r \in \lip_c(M)$. 
	Since $\{\varphi_j = 1\} = \overline{B}_{R_j}$ and $R_j\to\infty$ as $j\to\infty$, we also have $\varphi_j \to 1$ in $W^{1,\infty}_\loc(M)$. We now prove that $f|\nabla\varphi_j|^2 \to 0$ in $L^1(M)$. The function $H : \R^+ \to \R^+_0$ defined by
	\[
	H(t) = \int_{B_t} f
	\]
	is absolutely continuous on compact intervals contained in $\R^+$. By the co-area formula
	\begin{align*}
		\int_M f|\nabla\varphi_j|^2 & = \int_0^\infty \left(\int_{\partial B_t} f|\nabla\varphi_j|^2 \right) \, \di t \\
		& = \int_0^\infty H'(t) \psi_j'(t)^2 \, \di t \\
		& = \int_{R_j}^{r_j} H'(t) \psi_j'(t)^2 \, \di t \, .
	\end{align*}
	Integrating by parts, using that $0 \leq H(t)\leq t \Theta(t)$, $\psi_j' \leq 0$ and $\psi_j'' \geq 0$ and also using the explicit expression for $\psi_j'$ we obtain
	\begin{align*}
		\int_{R_j}^{r_j} H'(t) \psi_j'(t)^2 \, \di t & = H(t) \psi_j'(t)^2 \Big|_{t=R_j}^{t=r_j} - 2 \int_{R_j}^{r_j} H(t) \psi_j'(t) \psi_j''(t) \, \di t \\
		& \leq H(r_j) \psi_j'(r_j)^2 - 2 \int_{R_j}^{r_j} H(t) \psi_j'(t) \psi_j''(t) \, \di t \\
		& \leq r_j \Theta(r_j) \psi_j'(r_j)^2 - 2 \int_{R_j}^{r_j} t \Theta(t) \psi_j'(t) \psi_j''(t) \, \di t \\
		& = \frac{4}{\psi(r_j)^2} \frac{r_j}{\Theta(r_j)} + \frac{4}{\psi(r_j)} \int_{R_j}^{r_j} t \psi_j''(t) \, \di t \, .
	\end{align*}
	Another integration by parts yields
	\begin{align*}
		\int_{R_j}^{r_j} t \psi_j''(t) \, \di t & = t\psi_j'(t) \Big|_{t=R_j}^{t=r_j} - \int_{R_j}^{r_j} \psi_j'(t) \, \di t \\
		& = -\frac{2}{\psi(r_j)}\frac{t}{\Theta(t)} \Big|_{t=R_j}^{t=r_j} - \psi_j(t) \Big|_{t=R_j}^{t=r_j} \\
		& = \frac{2}{\psi(r_j)} \left(\frac{1}{\theta(R_j)} - \frac{1}{\theta(r_j)}\right) + 1
	\end{align*}
	so we get
	\[
	\int_M f|\nabla\varphi_j|^2 \leq \frac{8}{\psi(r_j)^2}\frac{1}{\theta(R_j)} + \frac{4}{\psi(r_j)} \to 0 \qquad \text{as } \, j \to \infty
	\]
	as desired.
\end{proof}

As a direct consequence, we have

\begin{corollary}\label{cor_lowenergy}
	If $a$ satisfies \eqref{assu_A}, $u$ has moderate energy growth in $\Omega_0$ whenever
	\[
	\int_{\Omega_0 \cap B_R} |\nabla u|^2\lambda_{\max}(|\nabla u|) = O(R^2 \log R) \qquad \text{as } \, R \to \infty. 
	\]
\end{corollary}

\begin{proof}
	It is enough to apply Lemma \ref{lem_phij} with the choices 
	\[
	f = |\nabla u|^2\lambda_{\max}(|\nabla u|) \mathbf{1}_{\Omega_0}, \qquad \theta(R) = \log R 
	\]
	and observe that 
	\[
	0 \le |\nabla u|^2 \langle A(\nabla u)\nabla \varphi_j, \nabla \varphi_j \rangle \le |\nabla u|^2 \lambda_{\max}(|\nabla u|)|\nabla \varphi_j|^2.   
	\]
\end{proof}

We next examine operators of mean curvature type, i.e. those for which 
\begin{equation}\label{eq_meancurv_type}
t^2 \lambda_1(t) \le C_1 \lambda_2(t) \le \frac{C_2}{t} \qquad \text{on } \, \R^+
\end{equation}
for some constants $C_1,C_2>0$. In this case, under suitable assumptions on $f$ and $u$ we can guarantee that $u$ has moderate energy growth by controlling the measures 
\[
|\Omega \cap B_R| \qquad \text{and possibly} \qquad \mathscr{H}^{n-1}(\partial \Omega \cap B_R).
\]

We first observe the following

\begin{proposition}\label{prop_moderate_boundgradient}
	Assume that $a$ satisfies \eqref{assu_A} and is of mean curvature type. If a solution $u$ to \eqref{equazu} satisfies
	\[
	|\nabla u| \in L^\infty(\Omega),
	\]
	and $|\Omega \cap B_R| = o(R^2 \log R)$ as $R \to \infty$, then $u$ has moderate energy growth.
\end{proposition}

\begin{proof}
	The mean curvature type condition implies that the eigenvalues of $A$ can be bounded as follows: 
	\[
	\lambda_{\max}(t) \le \max_{t \in [0,1]}(\lambda_{\max}) + \frac{2C_2}{1+t} \le \frac{C_3}{1+t}
	\]
	whence
	\begin{equation}\label{eq_uppermean}
	\int_{\Omega \cap B_R} |\nabla u|^2 \lambda_{\max} (|\nabla u|) \le C_3 \int_{\Omega \cap B_R} \frac{|\nabla u|^2}{1+|\nabla u|} \le C_3 \|\nabla u\|_\infty |\Omega \cap B_R|,  
	\end{equation}
	and we conclude by Corollary \ref{cor_lowenergy}.
\end{proof}

For the mean curvature operator, condition $|\nabla u| \in L^\infty(\Omega)$ is satisfied under mild assumptions on $M$, $u$ and $f$ including $f'(u) \le 0$. Indeed, we have the following theorem which improves on \cite[Theorem 1.5]{lian_sicbaldi}. We postpone its proof to the appendix, see Theorem \ref{thm_prop_grad}. The case where $f(u)$ is constant was considered in \cite{cmmr,cmr}.

\begin{theorem} \label{prop_grad}
	Let $(M^n,\metric)$ be a complete Riemannian manifold and $\Omega\subseteq M$ a domain. Let $f\in C^1(\R)$ and let $u\in C^3(\Omega) \cap C^1(\overline{\Omega})$ be a solution to
	\[
		\MM[u] + f(u) = 0 \qquad \text{in } \, \Omega \, .
	\]
	Suppose that
	\[
		\sup_\Omega |f(u)| < \infty \, , \qquad f'(u) \leq 0 \, , \qquad \inf_\Omega u > -\infty
	\]
	and that
	\[
		\Ric \geq 0 \quad \text{in } \, \Omega , \qquad \Sec \geq - \kappa \quad \text{in } \, M
	\]
	for some $\kappa\in\R^+$. Then
	\[
		\sup_\Omega |\nabla u| \leq \max\left\{ \sqrt{n/2}, \, \sup_{\partial\Omega} |\nabla u|\right\} .
	\]
\end{theorem}

%

If we do not assume $f'(u) \le 0$ in $\Omega$, global gradient estimates are, to our knowledge, currently not available, not even for the mean curvature operator. However, if $f(u)$ satisfies suitable bounds in $\Omega$ and $u$ suitable conditions on $\partial_\star \Omega = \partial\Omega \cap \{\partial_\eta u \neq 0\}$ we can still guarantee that $u$ has moderate energy growth by controlling the measures 
\[
|\Omega \cap B_R| \qquad \text{and} \qquad \mathscr{H}^{n-1}(\partial_\star \Omega \cap B_R),
\]
without needing to control $|\nabla u|$. Moreover, in certain cases all requests on $\partial_\star \Omega \cap B_R$ can be dropped. The calibration argument below is inspired by \cite[p. 403]{gilbargtrudinger}, \cite[p. 24]{simon} and \cite{FSV}. We localize our estimates to subdomains $\Omega_0 \subseteq \Omega$ whose boundary satisfies
\[
\partial \Omega_0 \subseteq \partial \Omega \cup \critu,
\]
where
\[
\critu = \big\{ x \in \Omega \ : \ |\nabla u(x)| = 0 \big\}
\]
is the \emph{interior critical set of $u$}. 

Define
\[
	\rho = \sqrt{ r^2 + u^2}
\]
that is, $\rho$ is the restriction to the graph of $u$ of the distance from $(o,0)$ in $M \times \R$. Hereafter, positive constants will be denoted by $C,C_0,C_1$ and so forth.

\begin{proposition} \label{prop_1_calib}
	Let $u \in C^2(\Omega) \cap C^1(\overline\Omega)$ be a weak solution to
	\[
		\Delta_a u + f(u) = 0 \qquad \text{in } \, \Omega,
	\]
	where $0 \le a \in C(\R^+_0)$ satisfies $ta(t) \in L^\infty(\R^+) \cap \lip_\loc(\R^+_0)$ and $f \in C(\R)$. Let $\Omega_0\subseteq\Omega$ be an open set with locally finite perimeter in $\Omega$ such that $\partial\Omega_0 \subseteq \partial\Omega \cup \critu$.
	\begin{itemize}
		\item [$(i)$] Suppose that
		\[
			u_{|\overline{\Omega}_0 \cap \partial_\star\Omega} \, \text{ is bounded,} \qquad -C_0 \leq f(u) \leq 0 \quad \text{in } \, \Omega_0, \qquad \inf_{\Omega_0} u > -\infty \, .
		\]
		and let $b\in\R$. Then there exists $C>0$ such that
		\begin{equation}
			\int_{\Omega_0\cap\{\rho\leq R\}} \aaa |\nabla u|^2 \leq C \left(|\Omega_0\cap B_{4R}| + \haus^{n-1}(\overline{\Omega}_0 \cap \partial^b_\star\Omega \cap B_{4R}) \right)
		\end{equation}
		for all sufficiently large $R$, where $\partial_\star^b\Omega=\{ x \in \partial_\star \Omega : u(x) \neq b\}$. 
		\item[$(ii)$] Suppose that
		\[
		u_{|\overline{\Omega}_0 \cap \partial_\star\Omega} \, \text{ is constant,} \qquad -C_0 \leq f(u) \leq 0 \quad \text{in } \, \Omega_0, \qquad \inf_{\Omega_0} u > -\infty \, ,
		\]
		then there exists $C>0$ such that
		\begin{equation}
			\int_{\Omega_0\cap\{\rho\leq R\}} \aaa |\nabla u|^2 \leq C |\Omega_0\cap B_{4R}|
		\end{equation}
		for all sufficiently large $R$. 
		\item[$(iii)$] Suppose that 
		\[
		u_{|\overline{\Omega}_0 \cap \partial_\star\Omega} = b \, \text{ is constant,} \qquad f(u) \leq 0 \quad \text{in } \, \Omega_0, \qquad \inf_{\Omega_0} u = b ,
		\]
		then there exists $C>0$ such that
		\begin{equation}
			\int_{\Omega_0\cap\{\rho\leq R\}} \aaa |\nabla u|^2 \leq C |\Omega_0\cap B_{4R}|
		\end{equation}
		for all sufficiently large $R$.
		\item [$(iv)$] Suppose that
		\[
			u_{|\overline{\Omega}_0 \cap\partial_\star\Omega} \, \text{ is constant,} \qquad |f(u)| \leq C_0 \quad \text{in } \, \Omega_0 \, ,
		\]
		then there exists $C>0$ such that
		\begin{equation}
			\int_{\Omega_0\cap\{\rho\leq R\}} \aaa |\nabla u|^2 \leq (C + C_0 R) |\Omega_0\cap B_{8R}|
		\end{equation}
		for all sufficiently large $R$. 
		\item [$(v)$] Suppose that 
		\[
		u_{|\overline{\Omega}_0 \cap\partial_\star\Omega} \, \text{ is constant,}  \qquad u \in L^\infty(\Omega_0), 
		\]
		then there exists $C'>0$ such that
		\begin{equation}
			\int_{\Omega_0\cap\{\rho\leq R\}} \aaa |\nabla u|^2 \leq C' |\Omega_0\cap B_{8R}|
		\end{equation}
		for all sufficiently large $R$. 
		\item [$(vi)$] Suppose that
		\[
			f(u) \leq 0 \qquad \text{in } \, \Omega_0
		\]
		with no further assumptions on $u$. Then there exists $C>0$ such that
		\[
			\int_{\Omega_0\cap\{\rho\leq R\}} \aaa |\nabla u|^2 \leq C \left(|\Omega_0\cap B_{2R}| + R\haus^{n-1}(\overline{\Omega}_0 \cap \partial_\star\Omega \cap B_{2R}) \right)
		\]
		for all $R>1$.		
	\end{itemize}
\end{proposition}

\begin{remark}
	The constant $C$ depends on $C_0, \|a\|_{L^\infty(\R^+)}, \|t a\|_{L^\infty(\R^+)}$ and, according to the case, $b$ (in $(i)$), $\|u\|_{L^\infty(\overline{\Omega}_0 \cap \partial_\star \Omega)}$ (in $(i)-(v)$) and $\|u\|_{L^\infty(\Omega_0)}$ (in $(v)$). Likewise, the lower bounds for $R$ guaranteeing the inequalities in $(i)-(vi)$ depend on the aforementioned quantities. 
\end{remark}

\begin{proof}
	For fixed $R>0$, define $\psi\in\lip_c(M)$ by
	\[
	\psi(x) = \left\{
	\begin{array}{l@{\quad}l}
		1 & \text{if } \, r(x) \leq R \\[0.2cm]
		2 - \dfrac{r(x)}{R} & \text{if } \, R < r(x) < 2R \\[0.3cm]
		0 & \text{if } \, r(x) \geq 2R
	\end{array}
	\right.
	\]
	Let $b_0,b\in\R$ be given satisfying $0 \leq b - b_0 \leq R$ and define $\gamma\in\lip(\R)$ by
	\[
		\gamma(t) = \left\{
			\begin{array}{l@{\quad}l}
				\dfrac{b_0-b}{R} & \text{if } \, t \leq b_0 \\[0.3cm]
				\dfrac{t-b}{R} & \text{if } \, b_0 < t < b_0 + R \\[0.3cm]
				\dfrac{b_0-b}{R}+1 & \text{if } \, t \geq b_0 + R
			\end{array}
		\right.
	\]
	Note that by construction
	\[
		\gamma(b) = 0, \qquad 0\leq\gamma' = \frac{1}{R}\mathbf{1}_{(b_0,b_0+R)}, \qquad |\gamma| \leq 1 \quad \text{on } \, \R .
	\]
	Consider the vector field
	\[
		W = \psi\gamma(u)a(|\nabla u|)\nabla u
	\]
	which is continuous and compactly supported in $\overline{\Omega}$, locally Lipschitz\footnote{For $|X|,|Y|\leq R$ we have $|a(|X|)X-a(|Y|)Y| \leq C_R|X-Y|$, with $C_R = 2\|a\|_{L^\infty} + [ta]_{\lip([0,R])}$, since
			\begin{align*}
			a(|X|)X - a(|Y|)Y & = [a(|X|) - a(|Y|)]X + a(|Y|)(X-Y), \\
			[a(|X|) - a(|Y|)]|X| & = |X|a(|X|) - |Y|a(|Y|) + (|Y|-|X|)a(|Y|)
			\end{align*}
	so $a(|\nabla u|)\nabla u$ is locally Lipschitz in $\Omega$ as $u \in C^2(\Omega)$.} in $\Omega$ and satisfies
	\[
		\diver \, W = - \psi\gamma(u)f(u) + \gamma(u) a(|\nabla u|) \langle\nabla u,\nabla\psi\rangle + \psi \gamma'(u) a(|\nabla u|) |\nabla u|^2
	\]
	weakly in $\Omega$. We apply the divergence theorem (see Lemma \ref{lem_divergence} in the appendix) to the vector field $W$ on the domain $\Omega_0$, observing that $W = 0$ on $\{\nabla u = 0\}$ since $ta(t) \to 0$ as $t\to0$. We obtain
	\begin{align*}
		\int_{\Omega_0} \psi \gamma'(u) \aaa |\nabla u|^2 & =  
		\int_{\Omega_0} \psi \gamma(u)f(u) - \int_{\Omega_0} \gamma(u) \aaa \langle \nabla u, \nabla \psi \rangle \\
		& \phantom{=\;} - \int_{\overline{\Omega}_0\cap\partial \Omega} a(|\nabla u|) \gamma(u)\psi \partial_\eta u 
	\end{align*}
	where, we recall, $\eta$ is the inward normal. We next  estimate the various terms. First, by our choice of $\gamma$ and $\psi$ we have
	\[
		\int_{\Omega_0} \psi \gamma'(u) \aaa |\nabla u|^2 \geq \frac{1}{R} \int_{\Omega_0 \cap B_R \cap \{b_0 < u < b_0+R\}} \aaa |\nabla u|^2
	\]
	and, since $ta(t) \in L^\infty(\R^+)$,
	\[
		- \int_{\Omega_0} \gamma(u) \aaa \langle \nabla u, \nabla \psi \rangle \leq \frac{1}{R} \int_{\Omega_0 \cap B_{2R}} \aaa |\nabla u| \leq \frac{C}{R} |\Omega_0 \cap B_{2R}|
	\]
	Noting that $\gamma(u)\partial_\eta u = 0$ on $\partial\Omega \backslash \partial^b_\star\Omega$ we have
	\[
		- \int_{\overline{\Omega}_0\cap\partial \Omega} a(|\nabla u|) \gamma(u) \psi \partial_\eta u = - \int_{\overline{\Omega}_0\cap\partial^b_\star \Omega} a(|\nabla u|) \gamma(u) \psi \partial_\eta u \, .
	\]
	For notational convenience, let us set $F^b_\star = \overline{\Omega}_0\cap\partial^b_\star \Omega$. Then, by also using the expression of $\gamma$ we have
	\[
		- \int_{\overline{\Omega}_0\cap\partial \Omega} a(|\nabla u|) \gamma(u) \psi \partial_\eta u \leq C \cdot \min\left\{1,\dfrac{\sup_{F^b_\star \cap B_{2R}}|u-b|}{R}\right\} \cdot \haus^{n-1}(F^b_\star\cap B_{2R})
	\]
	Finally, splitting $f(u) = [f(u)]_+ - [f(u)]_-$ into its positive and negative part and using that $\gamma(u) < 0$ on $\{u<b\}$ and $\gamma(u)>0$ on $\{u>b\}$, we have
	\begin{align*}
		\int_{\Omega_0} \psi \gamma(u)f(u) & = \int_{\Omega_0 \cap \{u<b\}} \psi \gamma(u)f(u) + \int_{\Omega_0 \cap \{u>b\}} \psi \gamma(u)f(u) \\
		& \leq \int_{\Omega_0 \cap \{u<b\}} \psi |\gamma(u)| [f(u)]_- + \int_{\Omega_0 \cap \{u>b\}} \psi \gamma(u) [f(u)]_+ .
	\end{align*}
	By using the expression of $\gamma$ we further estimate
	\[
		\int_{\Omega_0 \cap \{u>b\}} \psi \gamma(u) [f(u)]_+ \leq \min\left\{\frac{b_0-b+R}{R},\frac{\sup_{\Omega_0\cap B_{2R}} (u-b)_+}{R} \right\} \cdot \sup_{\Omega_0\cap B_{2R}} [f(u)]_+ \cdot |\Omega_0 \cap B_{2R}|
	\]
	and similarly
	\[
		\int_{\Omega_0 \cap \{u<b\}} \psi |\gamma(u)| [f(u)]_- \leq \min\left\{\frac{b-b_0}{R},\frac{\sup_{\Omega_0\cap B_{2R}} (u-b)_-}{R} \right\} \cdot \sup_{\Omega_0\cap B_{2R}} [f(u)]_- \cdot |\Omega_0 \cap B_{2R}|
	\]
	Putting everything together, we obtain
	\begin{align*}
		\int_{\Omega_0 \cap B_R \cap \{b_0 < u < b_0+R\}} \aaa |\nabla u|^2 & \leq C |\Omega_0 \cap B_{2R}| \\
		& \phantom{=\;} + C \cdot \min\left\{R,\sup_{F^b_\star\cap B_{2R}}|u-b|\right\} \cdot \haus^{n-1}(F^b_\star\cap B_{2R}) \\
		& \phantom{=\;} + \min\left\{b-b_0+R,\sup_{\Omega_0\cap B_{2R}} (u-b)_+ \right\} \cdot \sup_{\Omega_0\cap B_{2R}} [f(u)]_+ \cdot |\Omega_0 \cap B_{2R}| \\
		& \phantom{=\;} + \min\left\{b-b_0,\sup_{\Omega_0\cap B_{2R}} (u-b)_- \right\} \cdot \sup_{\Omega_0\cap B_{2R}} [f(u)]_- \cdot |\Omega_0 \cap B_{2R}|
	\end{align*}
	
	Now, we consider several cases. \\[0.2cm]
	\textbf{Cases $(i)$ and $(ii)$.} Suppose that the assumptions in $(i)$ hold:
	\[
		u_{|\overline{\Omega}_0 \cap \partial_\star\Omega} \, \text{ is bounded,} \qquad -C_0 \leq f(u) \leq 0 \quad \text{in } \, \Omega_0, \qquad \inf_{\Omega_0} u > -\infty \, .
	\]
	Then we can choose $b_0 = \inf_{\Omega_0} u$ and fixing any $b\geq b_0$ we obtain
	\begin{align*}
		\int_{\Omega_0 \cap B_R \cap \{u < \inf_{\Omega_0} u+R\}} \aaa |\nabla u|^2 & \leq (C + C_0(b-\inf_{\Omega_0} u)) \cdot |\Omega_0 \cap B_{2R}| \\
		& \phantom{=\;} + C \haus^{n-1}(F^b_\star \cap B_{2R}) \qquad \qquad \text{for all } \, R > b-\inf_{\Omega_0} u
	\end{align*}
	where we use that $[f(u)]_+ \equiv 0$, $[f(u)]_- \leq C_0$ and $(u-b)_- \leq b-b_0$. 
	
	If we further assume that
	\[
		u \text{ is constant on } \overline{\Omega}_0 \cap \partial_\star\Omega
	\]
	i.e. that the assumptions in $(ii)$ are met, then choosing $b$ as the constant value of $u$ on that set we have $F^b_\ast = \emptyset$ and
	\[
		\int_{\Omega_0 \cap B_R \cap \{u < \inf_{\Omega_0} u+R\}} \aaa |\nabla u|^2 \leq (C + C_0(b-\inf_{\Omega_0} u)) \cdot |\Omega_0 \cap B_{2R}| \qquad \text{for all } \, R > b-\inf_{\Omega_0} u
	\]
	Observe that if $\inf_{\Omega_0}u > -R/2$, then
	\[
		\Omega_0 \cap \{ \rho<R/2 \} \subseteq \Omega_0 \cap B_{R/2} \cap \{u<R/2\} \subseteq \Omega_0 \cap B_R \cap \{u<\inf_{\Omega_0} u + R\}
	\]
	hence
	\[
		\int_{\Omega_0 \cap \{\rho\leq R/2\}} \aaa |\nabla u|^2 \leq \int_{\Omega_0 \cap B_R \cap \{u < \inf_{\Omega_0} u+R\}} \aaa |\nabla u|^2
	\]
	and we therefore obtain the desired conclusions in $(i)$ and $(ii)$ up to replacing $R$ by $2R$. \\[0.2cm]
	\textbf{Case $(iii)$.} Suppose that
	\[
		u_{|\overline{\Omega}_0 \cap\partial_\star\Omega} \, \text{ is constant,} \qquad f(u) \leq 0 \quad \text{in } \, \Omega_0, \qquad \text{and} \qquad u_{|\overline{\Omega}_0 \cap\partial_\star\Omega} \equiv \inf_{\Omega_0} u > -\infty \, .
	\]
	Then we can choose $b = b_0 = \inf_{\Omega_0} u$ to obtain
	\[
		\int_{\Omega_0 \cap B_R \cap \{u < \inf_{\Omega_0} u+R\}} \aaa |\nabla u|^2 \leq C |\Omega_0 \cap B_{2R}| \qquad \text{for all } \, R>0 \, 
	\]
	and we conclude as in the previous step.\\[0.2cm]
	\textbf{Cases $(iv)$ and $(v)$.} Suppose that the assumptions in $(iv)$ are met:
	\[
		u_{|\overline{\Omega}_0 \cap \partial_\star\Omega} \, \text{ is constant,} \qquad |f(u)| \leq C_0 \quad \text{in } \, \Omega_0 \, .
	\]
	Then we can choose $b$ the constant value of $u$ on $\overline{\Omega}_0 \cap \partial_\star\Omega$ and $b_0 = b - R/2$ to obtain
	\[
		\int_{\Omega_0 \cap B_R \cap\left\{b - \frac{R}{2} < u < b + \frac{R}{2}\right\}} \aaa |\nabla u|^2 \leq \left( C + 3 C_0 \min\left\{ \frac{R}{2}, \sup_{\Omega_0\cap B_{2R}} |u-b| \right\}\right) |\Omega_0 \cap B_{2R}| .
	\]
	In particular, for some $C'>0$ we have
	\[
		\int_{\Omega_0 \cap B_R \cap \left\{b - \frac{R}{2} < u < b + \frac{R}{2}\right\}} \aaa |\nabla u|^2 \leq C' R |\Omega_0 \cap B_{2R}| \qquad \text{for all } \, R\geq 1.
	\]
	If we assume the requests in $(v)$, condition $|f(u)| \le C_0$ is then automatic for suitable $C_0$, and for some $C''>0$ we have
	\[
		\int_{\Omega_0 \cap B_R \cap \left\{b - \frac{R}{2} < u < b + \frac{R}{2}\right\}} \aaa |\nabla u|^2 \leq C'' |\Omega_0 \cap B_{2R}| \qquad \text{for all } \, R\ge 1.
	\]
	The conclusions in $(iv)$ and $(v)$ follow by observing that 
	\[
	\{\rho \le R/4\} \subset B_R \cap \{|u| \le R/4\} \subset B_R \cap \left\{b - \frac{R}{2} < u < b + \frac{R}{2} \right\}
	\]
	for large enough $R$. \\[0.2cm]
	\textbf{Case $(vi)$.} Suppose that
	\[
		f(u) \leq 0 \qquad \text{in } \, \Omega_0
	\]
	with no further assumptions on $u$. Then choosing $b = b_0 = 0$ we get
	\begin{align*}
		\int_{\Omega_0 \cap B_R \cap \{0 < u < R\}} \aaa |\nabla u|^2 & \leq C|\Omega_0\cap B_{2R}| \\
		& \phantom{=\;} + C\cdot\min\{R,\sup_{F^0_\star \cap B_{2R}} |u|\} \cdot \haus^{n-1}(F^0_\star \cap B_{2R})
	\end{align*}
	while choosing $b = b_0 = -R$ we have
	\begin{align*}
		\int_{\Omega_0 \cap B_R \cap \{-R < u < 0\}} \aaa |\nabla u|^2 & \leq C|\Omega_0\cap B_{2R}| \\
		& \phantom{=\;} + C\cdot\min\{R,\sup_{F^{-R}_\star \cap B_{2R}} |u+R|\} \cdot \haus^{n-1}(F^{-R}_\star \cap B_{2R}) \, .
	\end{align*}
	Since $F^0_\star$ and $F^{-R}_\star$ are contained in $\overline{\Omega}_0 \cap \partial_\star\Omega$, and using that $\nabla u = 0$ almost everywhere on $\{u=0\}$, we get
	\begin{align*}
		\int_{\Omega_0 \cap B_R \cap \{|u| < R\}} \aaa |\nabla u|^2 & \leq C\left(|\Omega_0\cap B_{2R}| + R \cdot \haus^{n-1}(\overline{\Omega}_0 \cap \partial_\star\Omega \cap B_{2R})\right)
	\end{align*}
	for all $R\geq 1$. The desired conclusion follows. 
	\end{proof}

\begin{proposition}\label{prop_inteannuli}
	Let $0 \le a \in C(\R^+_0)$ and $u$ satisfy
	\begin{equation}\label{primointe_abstract}
		\int_{\Omega_0 \cap \{\rho \le R\}} \aaa |\nabla u|^2 \le V(R)
	\end{equation}
	for some $0 \le V\in C(\R^+)$. Then, there exists $C_3>1$ such that, for each $R>0$, 
	\[
		\int_{\Omega_0 \cap \{\sqrt{R} \le \rho \le R\}} \frac{\aaa |\nabla u|^2}{\rho^2} \le C_3\left( \frac{V(R)}{R^2} + \int_{\sqrt{R}}^R \frac{V(\sigma)}{\sigma^3} \di \sigma\right).
	\]
\end{proposition}
\begin{proof}
	Writing $\rho^{-2}-R^{-2}$ as an integral and using Fubini's theorem and \eqref{primointe_abstract}, we get
	\[
	\begin{array}{lcl}
		\disp \int_{\Omega_0 \cap \{\sqrt{R} \le \rho \le R\}} \aaa |\nabla u|^2 \big[\rho^{-2}-R^{-2}\big] & = & \disp 2\int_{\Omega_0 \cap \{\sqrt{R} \le \rho \le R\}} \aaa |\nabla u|^2 \left[\int^{R}_{\rho} \sigma^{-3}\di \sigma\right]\di x, \\[0.5cm]
		& = & \disp \disp 2\int_{\sqrt{R}}^R \sigma^{-3} \left[\int_{\Omega_0 \cap \{\sqrt{R} \le \rho \le \sigma\}} \aaa |\nabla u|^2 \di x \right] \di \sigma \\[0.5cm]
		& \le & \disp \disp 2\int_{\sqrt{R}}^R \sigma^{-3} V(\sigma) \di \sigma \\[0.5cm]
	\end{array}
	\]
	Again by \eqref{primointe_abstract},
	\[
	\begin{array}{lcl}
		\disp \int_{\Omega_0 \cap \{\sqrt{R} \le \rho \le R\}} \frac{\aaa |\nabla u|^2}{\rho^2} & = & \disp \frac{1}{R^2}\int_{\Omega_0 \cap \{\sqrt{R} \le \rho \le R\}} \aaa |\nabla u|^2 \\[0.5cm]
		& & \disp + \int_{\Omega_0 \cap \{\sqrt{R} \le \rho \le R\}} \aaa |\nabla u|^2 \big[\rho^{-2}-R^{-2}\big] \\[0.5cm] 
		& \le & \disp C\left( \frac{V(R)}{R^2} + \int_{\sqrt{R}}^R \frac{V(\sigma)}{\sigma^3} \di \sigma\right).
	\end{array}
	\]
\end{proof}

\begin{proposition}\label{prop_volume_MC}    
	Let $a$ satisfy \eqref{assu_A} and 
	\begin{equation}\label{eq_ass_lambda}
	\begin{array}{ll}
	t^2 \lambda_{\max}(t) \to 0 & \quad \text{as } \, t \to 0, \\[0.3cm]
	t^2 \lambda_1(t) \le C \lambda_2(t) & \quad \text{on } \, (0, \infty).
	\end{array}
	\end{equation}
	Assume that $u$ satisfies \eqref{primointe_abstract} for some $0 \le V \in C(\R^+)$, and that
	\begin{equation}\label{crescivol}
		\lim_{R \ra \infty} \frac{V(R)}{R^2 \log R}=0. 
	\end{equation}
	Then, $u$ has moderate energy growth in $\Omega_0$.
\end{proposition}

\begin{proof}
	Choose 
	$$
	\varphi_R = \left\{ \begin{array}{ll}
		1 & \quad \text{if } \, \rho \le \sqrt{R}, \\[0.1cm]
		\disp \frac{2\log (R/\rho)}{\log R} & \quad \text{if } \, \rho \in [\sqrt{R},R] \\[0.3cm]
		0 & \quad \text{if } \, \rho \ge R.
	\end{array}\right.
	$$
	Then, 
	$$
	\nabla \varphi_R = -2\frac{r\nabla r+ u\nabla u}{\rho^2 \log R} \qquad \text{for } \, \rho \in [\sqrt{R},R].
	$$
	We thus have on $\{|\nabla u|>0\}$
	$$
	\begin{array}{lcl}
		\disp \langle A(\nabla u) \nabla \varphi_R, \nabla \varphi_R \rangle & = & \disp \frac{4r^2}{\rho^4 \log^2 R}\langle A(\nabla u)\nabla r, \nabla r\rangle +  \frac{4u^2}{\rho^4 \log^2 R} \langle A(\nabla u) \nabla u, \nabla u\rangle \\[0.4cm]
		& & \disp + \frac{8r u}{\rho^4 \log^2R} \langle A(\nabla u)\nabla u, \nabla r\rangle.
	\end{array}
	$$
	Hence, by Young's inequality and using $\rho^2 = r^2 + u^2$, 
	$$
	\begin{array}{lcl}
		\disp \langle A(\nabla u) \nabla \varphi_R, \nabla \varphi_R \rangle & \le & \disp C \left(\frac{r^2}{\rho^4 \log^2 R}\langle A(\nabla u)\nabla r, \nabla r\rangle +  \frac{u^2}{\rho^4 \log^2 R} \langle A(\nabla u) \nabla u, \nabla u\rangle\right) \\[0.5cm]
		& \le & \disp \frac{C}{\rho^2 \log^2 R} \Big(\langle A(\nabla u)\nabla r, \nabla r\rangle + \langle A(\nabla u) \nabla u, \nabla u\rangle\Big).
	\end{array}
	$$
	In our assumptions, 
	$$
	\begin{array}{lcl}
		\disp \langle A(\nabla u)\nabla r, \nabla r\rangle \le \lambda_{\max}(|\nabla u|) \le C \aaa \\[0.2cm]
		\langle A(\nabla u) \nabla u, \nabla u\rangle = \lambda_1(|\nabla u|)|\nabla u|^2 \le C\aaa.
	\end{array}
	$$
	Hence 
	$$
	\begin{array}{lcl}
		\disp \langle A(\nabla u) \nabla \varphi_R, \nabla \varphi_R \rangle & \le & \disp C\frac{\aaa}{\rho^2 \log^2 R}.
	\end{array}
	$$
	holds in $\{|\nabla u|>0\}$. Integrating and using Proposition \ref{prop_inteannuli}, 
	$$
	\begin{array}{lcl}
		\disp \int_{\Omega_0 \cap \{|\nabla u|>0\}} |\nabla u|^2 \langle A(\nabla u) \nabla \varphi_R, \nabla \varphi_R \rangle & \le & \disp \frac{C}{\log^2 R} \int_{\Omega_0 \cap \{\sqrt{R} \le \rho \le R\}} \frac{\aaa |\nabla u|^2}{\rho^2} \\[0.5cm]
		& \le & \disp C\left( \frac{V(R)}{R^2\log^2R} + \frac{1}{\log^2 R}\int_{\sqrt{R}}^R \sigma^{-3} V(\sigma) \di \sigma\right).
	\end{array}
	$$
	By the first in \eqref{eq_ass_lambda}, $|\nabla u|^2 A(\nabla u) \to 0$ as a tensor if $|\nabla u| \to 0$, so the set $\{|\nabla u|=0\}$ can be added in the left-hand side keeping the validity of the inequality. Because of \eqref{crescivol}, for a suitable function $o_R(1)$ of $R$ vanishing as $R \to \infty$ we therefore have 
	$$
	\begin{array}{lcl}
		\disp \int_{\Omega_0} |\nabla u|^2 \langle A(\nabla u) \nabla \varphi_R, \nabla \varphi_R \rangle & \le & \disp o_R(1) \left( \frac{1}{\log R} + \frac{1}{\log^2 R}\int_{\sqrt{R}}^R \frac{\log \sigma}{\sigma}\di \sigma\right) \\[0.5cm]
		& \le & \disp o_R(1) \left( \frac{1}{\log R} + \frac{1}{\log R}\int_{\sqrt{R}}^R \frac{\di \sigma}{\sigma}\right) = o_R(1) \left(\frac{1}{\log R}+ 1\right)
	\end{array}
	$$
	and the thesis follows.
\end{proof}

Putting together Propositions \ref{prop_1_calib}, \ref{prop_inteannuli} and \ref{prop_volume_MC}, we readily deduce the following when $a$ satisfies the stronger \eqref{assu_A_strong}.

\begin{theorem}\label{teo_goodcutoff_MC}
	Assume that  $f \in C^1(\R)$, that $a$ satisfies \eqref{assu_A_strong} and that
	\[
			\disp t^2 \lambda_1(t) \le C_1 \lambda_2(t) \le \frac{C_2}{t} \qquad \text{on } \, (0, \infty).
	\]
	for some constants $C_1,C_2>0$. Let $u \in C^2(\Omega) \cap C^1(\overline\Omega)$ solve $\Delta_a u + f(u) = 0$. Let $\Omega_0\subseteq\Omega$ be an open set with locally finite perimeter in $\Omega$ such that $\partial\Omega_0 \subseteq \partial\Omega \cup \critu$ and
	\[
			|\Omega_0 \cap B_R| = o\big( R^2 \log R\big) \qquad \text{as } \, R \to \infty.
	\] 
	Then, $u$ has moderate energy growth in $\Omega_0$ in any of the following cases:
	\begin{itemize}
		\item[(i)] $- C_0 \le f(u) \le 0$ in $\Omega_0$, 
		\[
			u_{|\overline{\Omega}_0 \cap \partial_\star\Omega} \, \text{ is bounded,} \qquad \inf_{\Omega_0} u > -\infty \, .
		\]
		and for some $b \in \R$ 
		\[
			\haus^{n-1}(\overline{\Omega}_0 \cap \partial^b_\star\Omega \cap B_{R}) = o\big( R^2 \log R\big) \qquad \text{as } \, R \to \infty, 
		\]
		where $\partial_\star^b\Omega = \{ x \in \partial_\star \Omega : f(x) \neq b\}$.
		\item[(ii)] $- C_0 \le f(u) \le 0$ in $\Omega_0$, 
		\[
			u_{|\overline{\Omega}_0 \cap \partial_\star\Omega} \, \text{ is constant,} \qquad \inf_{\Omega_0} u > -\infty \, .
		\]
		\item[(iii)] $f(u) \le 0$ in $\Omega_0$, 
		\[
			u_{|\overline{\Omega}_0 \cap \partial_\star\Omega} \, \text{ is a constant $b$,} \qquad \inf_{\Omega_0} u = b\, 
		\]
		\item[(iv)] $|f(u)| \le C_0$ in $\Omega_0$, $u$ is constant on $\overline{\Omega}_0 \cap \partial_\star \Omega$ and 
		\[
			|\Omega_0 \cap B_R| = o\big(R \log R\big) \qquad \text{as } \, R \to \infty.
		\]
		\item[(v)] $u$ is constant on $\overline{\Omega}_0 \cap \partial_\star \Omega$ and bounded in $\Omega_0$.
		\item[(vi)] $f(u) \le 0$ in $\Omega_0$ and 
		\[
			\haus^{n-1}(\overline{\Omega}_0 \cap \partial_\star\Omega \cap B_{R}) = o\big(R \log R\big) \qquad \text{as } \, R \to \infty.
		\]
	\end{itemize}			
\end{theorem}


\section{Stability and monotonicity}\label{sec_stability}

Let $u \in C^2(\Omega)$ be a non-constant, stable solution to $\Delta_a u + f(u) =0$ with $f \in C^1(\R)$. In this section, we shall investigate some consequences of the stability property, so we assume that:

\begin{itemize}
	\item $a$ satisfies \eqref{assu_A_strong}. In particular, $A(\nabla u)$ is a $(1,1)$ tensor field on $\Omega$  whose eigenvalues are locally bounded away from $0$ and $\infty$ (we shortly say that $A(\nabla u)$ is locally uniformly elliptic);
	\item \eqref{hardy} holds for each test function $\varphi \in \lip_c(\Omega)$. 
\end{itemize}

Observe that 

\begin{itemize}
	\item[-] If $u \in C^2(\Omega)$ and $a$ satisfies \eqref{assu_A_strong}, then $A(\nabla u)$ has $L^\infty_\loc(\Omega)$ coefficients;
	\item[-] If $u \in C^2(\Omega)$ and $a$ satisfies \eqref{assu_A_superstrong}, then $A(\nabla u)$ has $\lip_\loc(\Omega)$ coefficients.
\end{itemize} 

If $A(\nabla u)$ has $\lip_\loc$ coefficients, it is well-known that stability is equivalent to the existence of $0 < v \in C^{1,\alpha}_\loc(\Omega)$ solving the linearized equation
\[
	\diver \left(A(\nabla u) \nabla v\right) + f'(u)v = 0,
\]
see \cite{mosspie} for a proof. An analogous result holds when $A(\nabla u)$ has only $L^\infty_\loc$ coefficients. This is probably well-known, but we include a brief proof for the sake of completeness.

\begin{lemma}\label{lem_mosspie}
	Let $A$ be a $(1,1)$ tensor field with $L^\infty_\loc(\Omega)$ coefficients which is locally uniformly elliptic, and let $V \in L^\infty_\loc(\Omega)$. Then, the following are equivalent:
	\begin{itemize}
		\item[(i)] For each $\varphi \in \lip_c(\Omega)$,
		\[
			I(\varphi) \doteq \int_M \langle A\nabla\varphi, \nabla \varphi\rangle \di x - \int_{M} V \varphi^2\di x \ge 0.
		\]
		\item[(ii)] There exists $0 < v \in C^{0,\alpha}_\loc(\Omega) \cap H^1_\loc(\Omega)$ solving $\diver(A \nabla v) + Vv = 0$ weakly in $\Omega$;
		\item[(iii)] There exists $0 < w \in H^1_\loc(\Omega)$ solving $\diver(A \nabla w) + Vw \le 0$ weakly in $\Omega$, where $w>0$ means that $w$ has locally positive essential infimum on $\Omega$.
	\end{itemize}
\end{lemma}

\begin{proof}
	$(ii) \Rightarrow (iii)$ is obvious, and $(iii) \Rightarrow (i)$ follows the standard path by integrating 
	\[
		\diver(A \nabla w) + Vw \le 0
	\]
	against $\varphi^2/w$ and using Young's inequality, see \cite[Lemma 3.10]{prs}. To prove $(i) \Rightarrow (ii)$, we follow the argument in \cite[Theorem 1]{fischercolbrieschoen}: we choose an increasing, smooth exhaustion $\Omega_j \uparrow M$ and associated solutions to
	\[
		\left\{ \begin{array}{ll}
			\diver \left( A\nabla v_j\right) + V v_j = 0	& \text{in } \Omega_j \\[0.2cm]
			v_j = 1 & \text{on } \, \partial \Omega_j
		\end{array} \right.
	\]
	The existence of $v_j$ follows from \cite[Theorems 8.6, 8.12, 8.29]{gilbargtrudinger} if we prove that 
	\[
		\lambda_1(\Omega_j) = \inf\Big\{ Q(\psi) \ : \ 0 \not \equiv \psi \in \lip_c(\Omega_j) \Big\} > 0, \qquad Q(\psi) = \frac{I(\psi)}{\int \psi^2}.
	\]
	Indeed, if $\lambda_1(\Omega_j) = 0$, pick a minimizer $\xi_j$ for $I$ in $H^1_0(\Omega_j)$. Since $I$ is homogeneous and $I(\psi) = I(|\psi|)$ we can suppose that $\xi_j \ge 0$, $\int \xi_j^2  = 1$. Extending $\xi_j$ with zero outside of $\Omega_j$ yields a function $\xi_j \in H^1_0(\Omega_{j+1})$. Now, from $0 = \lambda_1(\Omega_j) = I(\xi_j) \ge \lambda_1(\Omega_{j+1}) \ge 0$ we deduce that $\xi_j$  also minimizes $Q$ in $H^1_0(\Omega_{j+1})$, hence it is a non-negative solution to 
	\[
		\diver \left( A \nabla \xi_j\right) + V\xi_j = 0 \qquad \text{in } \, \Omega_{j+1}
	\]
	which vanishes in a neighbourhood of $\partial \Omega_{j+1}$, contradicting the Harnack inequality in \cite[Theorem 8.20 and Corollary 8.21]{gilbargtrudinger}. Once the sequence $\{v_j\}$ is produced, to obtain $v$ it is enough to rescale $v_j$ to $1$ at a fixed point $p \in \Omega_1$ and to pass to limits by using again the Harnack inequality together with the local uniform $C^{0,\alpha}$ and $H^1$ elliptic estimates.  
\end{proof}

\begin{remark}\label{rem_harnackHopf}
	If $a$ satisfies \eqref{assu_A_strong} and $u \in C^1(\Omega)$, non-negative weak solutions $0 \le w \in C^1(\Omega)$ of 
	\[
		\diver\big(A(\nabla u) \nabla w \big) + f'(u)w \le 0
	\]
	satisfy the (half)-Harnack inequality \cite[Thm. 7.1.2]{pucci_serrin}, so either  $w\equiv 0$ or $w>0$ everywhere in the domain $\Omega$. Moreover, if $a$ satisfies \eqref{assu_A_superstrong} and $u \in C^2(\overline\Omega)$, the Hopf Lemma holds for solutions $w \in C^2(\Omega) \cap C^1(\overline\Omega)$, see \cite[Theorem 2.8.3]{pucci_serrin}.
\end{remark}

As is well-known, positive solutions to the linearized equation of $\Delta_a u + f(u) = 0$ can be produced if $u$ is monotone in the direction of a Killing vector field. We include the proof of the next result in our needed generality.

\begin{lemma}\label{lem_linearKilling}
	Assume the validity of \eqref{assu_A_strong}, 
 	let $\Omega \subseteq M$ be a domain and suppose that $u \in C^2(\Omega)$ solves $\Delta_a u + f(u)=0$. Let $X$ be a Killing field in $\overline\Omega$. Then, $w \doteq \langle \nabla u, X\rangle$ solves
	\begin{equation}\label{linearized_w} 
		\diver\big( A(\nabla u) \nabla w\big) + f'(u) w = 0 
	\end{equation}
	weakly in $\Omega$.
\end{lemma}

\begin{proof}
Consider the flux $\Phi : \mathscr{D} \subseteq \Omega \times \R \to \Omega$ associated to $X$, defined on its maximal domain $\mathscr{D}$, set  $\Phi_t(x)=\Phi(t,x)$ and let $\mathscr{D}_t \subseteq \Omega$ be the maximal domain of $\Phi_t$. Set $u_t \doteq u \circ \Phi_t : \mathscr{D}_t \to \R$. The Killing property of $X$ implies the next identities for each $\psi \in C^2(\Omega)$ and $\psi_t = \psi \circ \Phi_t : \mathscr{D}_t \to \R$:
	\begin{equation}\label{propbase}
		\left\{\begin{array}{l}
			\partial_t \psi_t = \langle \nabla \psi, X\rangle \circ \Phi_t, \qquad \nabla \psi_t = \di \Phi_{-t}(\nabla \psi), \\[0.2cm]
			|\nabla \psi_t|^2 = 
			|\nabla \psi|^2 \circ \Phi_t, \\[0.2cm]
			\di (\partial_t \psi_t) = \di\big(\langle \nabla \psi, X\rangle \big) \circ \di \Phi_t = \nabla \di  \psi(X, \di \Phi_t) + \langle \nabla \psi, \nabla_{\di \Phi_t} X \rangle.
		\end{array}\right.
	\end{equation}
	%
	In particular, from the second equation in \eqref{propbase} we deduce that, fixing $z \in U$ and considering the vector field $\nabla u_t$ along the constant curve $\gamma(t)=z$, 
	\begin{equation}\label{derit}
		\left.\frac{\nabla}{\di t}\right|_{0}\big(\nabla u_t(z)\big) = \left.\frac{\nabla}{\di t}\right|_0 \Big(\di \Phi_{-t}\big[\nabla u(\Phi_t(z))\big]\Big) = (\mathscr{L}_X \nabla u)(z) = [X, \nabla u] = \nabla_X \nabla u - \nabla_{\nabla u}X.
	\end{equation}
	
	%
	Since $\Phi_t$ is an isometry, $u_t$ solves weakly $\Delta_a u_t + f(u_t) = 0$ on $\mathscr{D}_t$. Fix $\varphi \in C^\infty_c(\Omega)$, and choose $t_0$ small enough so that, for $|t| \le t_0$, $\varphi \in C^\infty_c(\mathscr{D}_{t})$. Then,  
	\[
	\int_M \Big\{ a(|\nabla u_t(z)|) \langle \nabla u_t, \nabla \varphi\rangle_z - f(u_t(z))\varphi(z) \Big\} \di z = 0.
	\]
	Next, by \eqref{assu_A_strong} and since $u \in C^2(\Omega)$ the integrand is a Lipschitz function of $(t, z) \in (-\eps, \eps) \times {\rm spt}(\varphi)$. Hence, by the differentiation theorem under the integral sign we can differentiate in $t$ to get 
	\begin{equation}\label{diffeinte}
		\begin{array}{lcl}         
			0 & = & \disp \disp \int \Big\{\frac{a'(|\nabla u_t|)}{|\nabla u_t|} \nabla \di  u_t(\nabla u_t,X) \langle \nabla u_t, \nabla \varphi\rangle + a(|\nabla u_t|) \frac{\di}{\di t}\langle \nabla u_t, \nabla \varphi \rangle - f'(u_t)(w \circ \Phi_t)\varphi\Big\}.
		\end{array}
	\end{equation}
	Evaluating at $t=0$ and noting that, by \eqref{derit} and since $X$ is Killing,
	$$
	\begin{array}{lcl}
		\disp \left.\frac{\di}{\di t}\right|_{t=0}\langle \nabla u_t, \nabla \varphi \rangle_z & = & \disp \langle \left.\frac{\nabla}{\di t}\right|_0(\nabla u_t), \nabla \varphi \rangle_z = \langle \nabla_X \nabla u, \nabla \varphi \rangle - \langle \nabla_{\nabla u} X, \nabla \varphi \rangle \\[0.3cm]
		& = & \disp \nabla \di  u(X, \nabla \varphi) - \langle \nabla_{\nabla u} X, \nabla \varphi \rangle = \disp \nabla \di  u(X, \nabla \varphi) + \langle \nabla_{\nabla \varphi} X, \nabla u \rangle
	\end{array}
	$$
	we obtain
	\begin{equation}\label{diffeinte_2}
		\begin{array}{lcl}
			0 & = & \disp \disp \disp \int \Big\{ \frac{a'(|\nabla u|)}{|\nabla u|} \nabla \di  u(\nabla u,X) \langle \nabla u, \nabla \varphi\rangle + a(|\nabla u|) \big[\nabla \di  u(X, \nabla \varphi) + \langle \nabla_{\nabla \varphi} X, \nabla u \rangle\big] - f'(u)w \varphi\Big\}.
		\end{array}
	\end{equation}
	Using now that 
	$$
	\langle \nabla w, \nabla \varphi \rangle = \nabla \varphi\big( \langle \nabla u, X \rangle \big) = \nabla \di  u (\nabla \varphi, X) + \langle \nabla u, \nabla_{\nabla \varphi}X\rangle
	$$
	we get
	\begin{equation}\label{final}
		\begin{array}{lcl}
			0 & = & \disp \int \Big\{\frac{a'(|\nabla u|)}{|\nabla u|} \langle \nabla w, \nabla u \rangle \langle \nabla u, \nabla \varphi \rangle + a(|\nabla u|)\langle \nabla w, \nabla \varphi \rangle - f'(u)w \varphi\Big\} \\[0.4cm]
			& = & \disp \int \Big\{ \langle A(\nabla u)\nabla w, \nabla \varphi \rangle - f'(u)w \varphi \Big\},
		\end{array}
	\end{equation}
	that concludes the proof. 
\end{proof}

\begin{lemma} \label{Lemma_w-}
	Let $(M^n,\metric)$ be a Riemannian manifold and $\Omega \subseteq M$ a domain. Let $a$ satisfy \eqref{assu_A_strong}, $f \in C^1(\R)$ 
	and let $w\in C^1(\Omega)\cap C(\overline{\Omega})$ be a solution to
	\begin{equation} \label{w-eq}
	\left\{
	\begin{array}{ll}
		\diver\big( A(\nabla u) \nabla w\big) + f'(u) w = 0 & \qquad \text{weakly in } \Omega \\[0.2cm]
		w \geq 0 & \qquad \text{on } \, \partial\Omega \\
	\end{array}
	\right.
	\end{equation}
	for some $u \in C^1(\overline\Omega)$. Suppose that there exists $\{\varphi_j\} \subseteq \lip_c(\overline\Omega)$ satisfying 
	\begin{equation} \label{Hgrowth}
	\varphi_j \to 1 \quad \text{in } \, C_\loc(\overline\Omega) \qquad \text{and} \qquad \int_\Omega w_-^2 \langle A(\nabla u) \nabla \varphi_j, \nabla \varphi_j \rangle \to 0
\end{equation}
as $j \to \infty$, where $w_- = \max\{-w,0\}$ denotes the negative part of $w$. Then either
	\begin{equation} \label{w-alt3}
		w\equiv 0 \text{ in } \Omega \qquad \text{or} \qquad w<0 \text{ in } \Omega \qquad \text{or} \qquad w>0 \text{ in } \Omega.
	\end{equation}
\end{lemma}

\begin{remark}
		We emphasize that in Lemma \ref{Lemma_w-} no assumptions are made on the regularity of $\partial\Omega$, as we only require $\Omega$ to be open and connected.
\end{remark}

\begin{proof}
By Lemma \ref{lem_mosspie}, since $u$ is stable there exists a weak solution $0 < v \in C^{0,\alpha}_\loc(\Omega) \cap H^1_\loc(\Omega)$ to
\[
		\diver \left( A(\nabla u)\nabla v\right) + f'(u)v = 0.
\]

Set $z = w_-/v$. We first show that
	\begin{equation} \label{A-L1}
		\int_\Omega \varphi^2 v^2 \langle A(\nabla u) \nabla z, \nabla z \rangle \leq 4 \int_\Omega w_-^2 \langle A(\nabla u)\nabla\varphi,\nabla\varphi \rangle 
		\qquad \text{for all } \, \varphi \in \lip_c(M) \, .
	\end{equation}
	To do so, for $\eps>0$ let us set $v_\eps \doteq v+\eps$ and $z_\eps = w_-/v_\eps$. The function $v_\eps$ solves
	\[
	\diver \left( A(\nabla u)\nabla v_\eps\right) + f'(u)v_\eps = \eps f'(u)
	\]
	weakly in  $\Omega$, so $z_\eps$ satisfies
	\begin{equation} \label{veps_weak}
		\diver \left( v_\eps^2 A(\nabla u)\nabla z_\eps \right) \ge -\eps w_-f'(u) \qquad \text{weakly in } \, \Omega.
	\end{equation}
	Let $\varphi \in \lip_c(\overline\Omega)$ be a cut-off function on $M$, $\delta>0$ and set $U_\delta \doteq \{z > \delta\}$ and $U_{\eps,\delta} \doteq \{z_\eps > \delta\}$.
	By construction
	\[
	\overline{U}_{\eps,\delta} \subseteq \overline{\{w_- > \eps\delta\}} \subseteq \Omega
	\]
	so $(z_\eps -\delta)_+$ vanishes in a neighbourhood of $\partial \Omega$ and thus, testing with $\varphi^2 (z_\eps -\delta)_+ \in \lip_c(\Omega)$ we get
	\begin{align*}
		\disp - \eps \int_{\Omega} \varphi^2 f'(u) & w_-(z_\eps -\delta)_+ \le - \disp \int_\Omega \langle A(\nabla u) \nabla z_\eps, \nabla(\varphi^2(z_\eps -\delta)_+) \rangle v_\eps^2\\[0.5cm]
		& = \disp - \int_\Omega \varphi^2 \langle A(\nabla u) \nabla z_\eps, \nabla z_\eps \rangle v_\eps^2 \mathbf{1}_{U_{\eps,\delta}} + 2 \int_\Omega \varphi(z_\eps - \delta)_+ \langle A(\nabla u) \nabla z_\eps, \nabla \varphi \rangle v_\eps^2 \\[0.5cm]
		& \leq \disp - \frac{1}{2} \int_\Omega \varphi^2 v_\eps^2 \langle A(\nabla u) \nabla z_\eps, \nabla z_\eps \rangle \mathbf{1}_{U_{\eps,\delta}} + 2 \int_\Omega (z_\eps - \delta)_+^2v_\eps^2 \langle A(\nabla u) \nabla \varphi, \nabla \varphi \rangle.	    
	\end{align*}
	Since $(z_\eps-\delta)^2_+ v_\eps^2 \le w_-^2$ on $\Omega$, we have	
	\[
	\frac{1}{2} \int_\Omega \varphi^2 v_\eps^2 \langle A(\nabla u) \nabla z_\eps, \nabla z_\eps \rangle \mathbf{1}_{U_{\eps,\delta}} \le 2 \int_\Omega w_-^2 \langle A(\nabla u) \nabla \varphi, \nabla \varphi \rangle + \eps \int_\Omega \varphi^2 f'(u)w_-(z_\eps -\delta)_+ .
	\]
	We keep $\delta>0$ fixed and we let $\eps\to0^+$ in this inequality. Since
	\[
	|\eps\varphi^2 f'(u)w_-(z_\eps-\delta)_+| \leq \eps \varphi^2 |f'(u)| w_- z_\eps = \varphi^2 |f'(u)| w_-^2 \frac{\eps}{v_\eps} \leq \varphi^2 |f'(u)| w_-^2 \in L^1(\Omega)
	\]
	by dominated convergence we have
	\[
	\lim_{\eps\to0^+} \eps \int_\Omega \varphi^2 f'(u)w_-(z_\eps -\delta)_+ = 0 \, .
	\]
	On the other hand, expanding $\nabla z_\eps$ and using Fatou's lemma we have
	\[
	\liminf_{\eps\to0^+} \int_\Omega \varphi^2 v_\eps^2 \langle A(\nabla u) \nabla z_\eps, \nabla z_\eps \rangle \mathbf{1}_{U_{\eps,\delta}} \geq \int_{\Omega} \varphi^2 v^2 \langle A(\nabla u) \nabla z, \nabla z \rangle \mathbf{1}_{U_\delta}
	\]
	so we get
	\[
	\int_\Omega \varphi^2 v^2 \langle A(\nabla u) \nabla z, \nabla z \rangle \mathbf{1}_{U_\delta} \le 4 \int_\Omega w_-^2 \langle A(\nabla u) \nabla \varphi, \nabla \varphi \rangle
	\]
	and then letting $\delta\to0^+$ we obtain, by monotone convergence, \eqref{A-L1}. In our assumptions, choosing $\varphi = \varphi_j$ in \eqref{A-L1} and letting $j \to \infty$ we deduce
%
	\[
	\int_\Omega v^2 \langle A(\nabla u)\nabla z,\nabla z\rangle = 0 \, .
	\]
	Since $v>0$ in $\Omega$ and $\langle A(\nabla u)\,\cdot\,,\,\cdot\,\rangle$ is positive definite we infer $\nabla z\equiv0$ almost everywhere in $\Omega$ and therefore $z$ is constant in $\Omega$, that is, there exists a constant $\lambda\geq0$ such that
	\[
		w_- = \lambda v \qquad \text{in } \, \Omega \, .
	\]
	Since $v>0$ in $\Omega$, this implies that either $w_->0$ everywhere in $\Omega$ or $w_- \equiv 0$ in $\Omega$, that is, either $w<0$ everywhere in $\Omega$ or $w\geq 0$ in $\Omega$. In the latter case, by the (half)-Harnack inequality (see Remark \ref{rem_harnackHopf}) we either have $w\equiv 0$ or $w>0$ everywhere in $\Omega$.
\end{proof}

\begin{remark} \label{rem-Lemma_w-}
	In the proof of Lemma \ref{Lemma_w-} we showed that if there exists a strictly positive solution $v\in C^{0,\alpha}_\loc(\Omega) \cap H^1_\loc(\Omega)$ to
	\[
		\diver(A(\nabla u)\nabla v) + f'(u)v = 0 \qquad \text{in } \, \Omega
	\]
	and $w\in C^1(\Omega)\cap C(\overline{\Omega})$ is a weak solution to
	\begin{equation} \label{weq-rem}
		\diver(A(\nabla u)\nabla w) + f'(u)w = 0 \qquad \text{in } \, \Omega
	\end{equation}
	satisfying $w\geq 0$ on $\partial\Omega$ and \eqref{Hgrowth}, then either $w > 0$ in $\Omega$ or $w = \lambda v$ in $\Omega$ for some constant $\lambda\leq 0$. 
	The same argument applied to both $w_+$ and $w_-$ shows that if $w\in C^1(\Omega)\cap C(\overline{\Omega})$ is a weak solution to \eqref{weq-rem} satisfying both $w=0$ on $\partial \Omega$ and \eqref{Hgrowth} with $w$ in place of $w_-$, then $w = \lambda v$ in $\Omega$ for some constant $\lambda\in\R$.
\end{remark}

\begin{theorem}\label{teo_monotonicity}
Let $(M, \metric)$ be a complete Riemannian manifold, let $\Omega$ be a $C^1$ domain and let $u \in C^2(\Omega) \cap C^1(\overline{\Omega})$ be a non-constant, stable solution of
\[
\Delta_a u + f(u) = 0 \qquad \text{in } \Omega, 
\]	
where $a$ satisfies \eqref{assu_A_strong} and $f \in C^1(\R)$. Let $\Omega_0 \subseteq \Omega$ be a subdomain such that $\partial \Omega_0 \subseteq \partial \Omega \cup \crit(u)$. Assume that $X$ is a bounded Killing vector field in $\Omega_0$, continuous in $\overline\Omega_0$ and satisfying 
\begin{equation}\label{eq_bound_Kil}
\langle \nabla u, X \rangle \ge 0, \quad \langle \nabla u, X \rangle \not \equiv 0 \qquad \text{on } \, \overline{\Omega}_0 \cap \partial \Omega. 
\end{equation}
If $u$ has moderate energy growth in $\Omega_0$, then $\langle \nabla u, X \rangle > 0$ in $\Omega_0$.
\end{theorem}

\begin{proof}
	By Lemma \ref{lem_linearKilling}, $w = \langle \nabla u,X \rangle$ solves the linearized equation 
	\[
	\diver \left( A(\nabla u)\nabla w\right) + f'(u)w = 0
	\]
	weakly in $\Omega_0$ and $w\geq 0$, $w \not \equiv 0$ on $\partial\Omega_0$. Since, in our assumptions, $w^2 \leq C|\nabla u|^2$ and $u$ has moderate energy growth in $\Omega_0$, we can apply Lemma \ref{Lemma_w-} in $\Omega_0$ to deduce that $w \geq 0$ there. Since $w\not\equiv 0$ on $\partial\Omega_0$, it must be $w>0$ somewhere in $\Omega_0$ and therefore, by the Harnack inequality in \cite[Thm. 7.1.2]{pucci_serrin}, $w>0$ everywhere in $\Omega_0$.
\end{proof}

\section{The geometric Poincar\'e formula and splitting}\label{sec_poinc}
For $x \not \in \critu$, fix a local orthonormal frame $\{e_i\} = \{\nu, e_\alpha\}$, where $\nu=\nabla u/|\nabla u|$ and $1 \le \alpha \le n-1$. Given $\phi \in C^\infty(M)$, we denote by  $\nabla^\top  \phi$ the component of $\nabla \phi$ orthogonal to $\nu$, and by $\II(x)$ the second fundamental form of the level set $\{u=u(x)\}$ at $x$. The subscript $\nu$ will mean derivation in the direction of $\nu$, so that, for instance, $u_\nu = |\nabla u|$. The coefficients of (the $(0,2)$-version of) $A(X)$ will be denoted by $A^{ij}$.

\begin{lemma}\label{lem_basicomp}
	Let $A$ satisfy \eqref{assu_A}. Near a point where $|\nabla u| \neq 0$, the following identities hold:
	\begin{equation}
		\begin{array}{ll}\label{identities}
			1) & \quad \langle \nabla |\nabla u|, e_i\rangle = \nabla \di  u(\nu, e_i) = \frac{1}{u_\nu} u_{ki}u^k, \qquad \big| \nabla |\nabla u|\big|^2 = u_{\nu\nu}^2 + \big| \nabla^\top  |\nabla u|\big|^2 \\[0.3cm]
			2) & \quad |\nabla \di  u|^2 - \big| \nabla |\nabla u|\big|^2 = \big| \nabla^\top   |\nabla u|\big|^2 + u_\nu^2|\II|^2 \\[0.3cm]
			3) & \quad \langle A(\nabla u)\nabla |\nabla u|, \nabla |\nabla u|\rangle 
			= \lambda_2|\nabla^\top  |\nabla u||^2 + \lambda_1 u_{\nu\nu}^2 \\[0.3cm]
			4) & \quad \nabla \di u^{(2)}(\nabla u, \nabla u) = u_\nu^2( u_{\nu\nu}^2 + |\nabla^\top  |\nabla u||^2) \\[0.3cm]
			5) & \quad A^{ij}u_{jk}u_{ir}A^{kr} = 2\lambda_1\lambda_2 \big|\nabla^\top  |\nabla u|\big|^2 + \lambda_1^2u_{\nu\nu}^2 + u_\nu^2\lambda_2^2|\II|^2 \\[0.3cm]
			6) & \quad \langle A(\nabla u) \nabla |\nabla u|, \nabla u \rangle 
			= u_\nu \lambda_1 u_{\nu\nu}. 
		\end{array}
	\end{equation}
	where $\nabla \di u^{(2)}$ is the tensor whose components are $u_{ik}u^k_j$ and $\lambda_j = \lambda_j(|\nabla u|)$.
\end{lemma}

\begin{remark}
	Item $2)$ was used by Sternberg and Zumbrun in \cite{sternbergzumbrun1}.
\end{remark}
\begin{proof}
	$1)$ and $2)$ are standard identities. As for $3)$, by $1)$
	$$
	A(\nabla u)\nabla\abs{\nabla u}=\nabla \di   u(\nu,\nu)\lambda_1\nu+\nabla \di   u(\nu,e_\alpha)\lambda_2e_\alpha=u_{\nu\nu}\lambda_1\nu+u_{\nu\alpha}\lambda_2e_\alpha,
	$$
	so
	$$
	\langle A(\nabla u)\nabla\abs{\nabla u},\nabla\abs{\nabla u}\rangle=\lambda_1u_{\nu\nu}^2+\lambda_2u_{\nu\alpha}u_{\nu\alpha}
	$$
	and, by $1)$, $u_{\nu\alpha}u_{\nu\alpha}=\abs{\nabla^\top  \abs{\nabla u}}^2$.\\
	To prove $4)$, observe that
	$$
	\nabla \di u^{(2)}(\nabla u,\nabla u)=u_{ik}u^k_ju^iu^j=u_\nu^2 u_{\nu k}u_{\nu k}=u_\nu^2\left(u_{\nu\nu}^2+\abs{\nabla^\top  \abs{\nabla u}}^2\right).
	$$
	For identity $5)$, we write
	\begin{align*}
		A^{ij}u_{jk}u_{ir}A^{kr}=&\frac{a'(\abs{\nabla u})^2}{\abs{\nabla u}^2}\left[\nabla \di  u(\nabla u,\nabla u)\right]^2+2\frac{a(\abs{\nabla u})a'(\abs{\nabla u})}{\abs{\nabla u}}\nabla \di u^{(2)}(\nabla u,\nabla u)+a(\abs{\nabla u})^2\abs{\nabla \di  u}^2\\
		=&\frac{a'(\abs{\nabla u})^2}{\abs{\nabla u}^2}u_{\nu\nu}^2u_\nu^4+2\frac{a(\abs{\nabla u})a'(\abs{\nabla u})}{\abs{\nabla u}}u_\nu^2\left(u_{\nu\nu}^2+\abs{\nabla^\top  \abs{\nabla u}}^2\right)+\\
		&+a(\abs{\nabla u})^2\left(u_{\nu\nu}^2+2\abs{\nabla^\top  \abs{\nabla u}}^2+u_\nu^2\abs{\operatorname{II}}^2\right)\\
		=&\left[\abs{\nabla u}a'(\abs{\nabla u})+a(\abs{\nabla u})\right]^2u_{\nu\nu}^2+\\
		&+2a(\abs{\nabla u})\left[a(\abs{\nabla u})+\abs{\nabla u}a'(\abs{\nabla u})\right]\abs{\nabla^\top  \abs{\nabla u}}^2+a(\abs{\nabla u})u_\nu^2\abs{\operatorname{II}}^2\\
		=&\lambda_1^2u_{\nu\nu}^2+2\lambda_1\lambda_2\abs{\nabla^\top  \abs{\nabla u}}^2+\lambda_2^2u_\nu^2\abs{\operatorname{II}}^2.
	\end{align*} 
	Finally, to prove $6)$, we just write
	$$
	\langle A(\nabla u)\nabla\abs{\nabla u},\nabla u\rangle=A^{ij}\frac1{u_\nu}u_{ik}u^ku_j=\lambda_1\frac{u^i}{u_\nu}u_{ik}u^k=\frac{\lambda_1}{u_\nu}\nabla \di  u(\nabla u,\nabla u)=\lambda_1u_{\nu\nu}u_\nu^2.
	$$
\end{proof}

We begin with the following Bochner type formulas:

\begin{proposition}
	Let $u \in C^3(\Omega)$, and let $a$ satisfy \eqref{assu_A_superstrong}. Then, at points of $\Omega \backslash \critu$ where $a(|\nabla u|)$ is twice differentiable, it holds
	\begin{equation}\label{bochner_better}
		\begin{array}{lcl}
			\disp \frac{1}{2} \diver\big( A(\nabla u) \nabla |\nabla u|^2 \big) & = & \lambda_1 u_{\nu\nu}^2 + (\lambda_1+\lambda_2) \big|\nabla^\top  |\nabla u|\big|^2 +\lambda_2u_\nu^2|\II|^2 \\[0.3cm]
			& & \disp + \lambda_2\Ric(\nabla u, \nabla u) + \langle \nabla \Delta_a u, \nabla u \rangle.
		\end{array} 
	\end{equation}
	and 
	\begin{equation}\label{bochner_betterfornablau}
		\begin{array}{lcl}
			\disp |\nabla u| \diver\big( A(\nabla u) \nabla |\nabla u| \big) & = & \lambda_1 \big|\nabla^\top  |\nabla u|\big|^2 +\lambda_2u_\nu^2|\II|^2 \\[0.3cm]
			& & \disp + \lambda_2\Ric(\nabla u, \nabla u) + \langle \nabla \Delta_a u, \nabla u \rangle,
		\end{array} 
	\end{equation}
	where $\lambda_j = \lambda_j(|\nabla u|)$. Moreover, if $\Delta_a u + f(u) = 0$ with $f \in C^1(\R)$ then \eqref{bochner_better} and \eqref{bochner_betterfornablau} hold, with the inequality sign $\geq$ in place of equality, weakly in the entire $\Omega$.
\end{proposition}

\begin{remark}\label{rem_stampa}
	The weak inequality requires to give a proper meaning to $\nabla^\top  |\nabla u|$, $u_{\nu\nu}$ and $u^2_\nu |\II|^2$ on $\critu$. By Stampacchia's inequality, $\nabla \di u \equiv 0$ a.e. on $\critu$. Since  
	\[
		u_{\nu\nu}^2 + 2|\nabla^\top |\nabla u||^2 + u_\nu^2|\II|^2 = |\nabla \di u|^2 \qquad \text{where } \, |\nabla u|> 0
	\]
	the terms $|\nabla^\top |\nabla u||^2$, $u_{\nu\nu}^2$ and $u_\nu^2|\II|^2$ extend continuously to the entire $\Omega$ by setting them to zero on $\critu$.
\end{remark}

\begin{proof}
	First, observe that $a(|\nabla u|) \in C^{1,1}_\loc(\{|\nabla u|>0\})$, hence the subset $E$ where it is twice differentiable has full measure in $\{|\nabla u|>0\}$. For $x \in E$ define
	\[
		(\star) = \diver\big( A(\nabla u) \nabla |\nabla u|^2 \big). 
	\]
	Using $1)$ in Lemma \ref{lem_basicomp}, we compute 
	$$
		\begin{array}{lcl}
			(\star) & = & \disp \diver \left( \frac{a'(|\nabla u|)}{|\nabla u|} \langle \nabla u , \nabla |\nabla u|^2 \rangle \nabla u + a(|\nabla u|) \nabla |\nabla u|^2\right) \\[0.4cm]
			& = & \disp \diver \left( 2\frac{a'(|\nabla u|)}{|\nabla u|} \nabla \di  u(\nabla u, \nabla u) \nabla u + a(|\nabla u|) \nabla |\nabla u|^2\right) \\[0.4cm]
			& = & \disp \left( 2\frac{a'(|\nabla u|)}{|\nabla u|}\right)' \frac{\big[\nabla \di  u(\nabla u, \nabla u)\big]^2}{|\nabla u|} + 2\frac{a'(|\nabla u|)}{|\nabla u|} \langle \nabla \big(\nabla \di  u(\nabla u, \nabla u)), \nabla u \rangle\\[0.4cm]
			& & \disp + 2\frac{a'(|\nabla u|)}{|\nabla u|} \nabla \di  u(\nabla u, \nabla u) \Delta u + \frac{a'(|\nabla u|)}{|\nabla u|} \nabla \di  u( \nabla u, \nabla |\nabla u|^2) + a(|\nabla u|) \Delta|\nabla u|^2\\[0.4cm]
			& = & \disp 2\left(\frac{a'(|\nabla u|)}{|\nabla u|}\right)' \frac{\big[\nabla \di  u(\nabla u, \nabla u)\big]^2}{|\nabla u|} + 2\frac{a'(|\nabla u|)}{|\nabla u|} \langle \nabla \big(\nabla \di  u(\nabla u, \nabla u)), \nabla u \rangle\\[0.4cm]
			& & \disp + 2\frac{a'(|\nabla u|)}{|\nabla u|} \nabla \di  u(\nabla u, \nabla u) \Delta u + 2\frac{a'(|\nabla u|)}{|\nabla u|} \nabla \di u^{(2)}(\nabla u, \nabla u) + \\[0.4cm]
			& & + \disp 2a(|\nabla u|) \big[ |\nabla \di  u|^2 + \langle \nabla \Delta u, \nabla u \rangle + \Ric(\nabla u, \nabla u)\big].
		\end{array}
	$$
	Next, 
	\begin{equation}\label{dos}
		\begin{array}{lcl}
			\disp a(|\nabla u|) \langle \nabla \Delta u, \nabla u \rangle & = & \disp \langle \nabla \Big(a(|\nabla u|)\Delta u\Big), \nabla u \rangle - \frac{a'(|\nabla u|)}{|\nabla u|}\nabla \di  u(\nabla u, \nabla u) \Delta u. \\[0.4cm]
		\end{array}
	\end{equation}
	Since 
	$$
		\Delta_a u = \frac{a'(|\nabla u|)}{|\nabla u|}\nabla \di  u(\nabla u, \nabla u) + a(|\nabla u|) \Delta u,
	$$
	inserting it into \eqref{dos} we get
	\begin{equation}\label{tres}
		\begin{array}{lcl}
			\disp a(|\nabla u|) \langle \nabla \Delta u, \nabla u \rangle & = & \disp \langle \nabla \left(-\frac{a'(|\nabla u|)}{|\nabla u|}\nabla \di  u(\nabla u, \nabla u) + \Delta_a u \right), \nabla u \rangle \\[0.4cm]
			& & \disp - \frac{a'(|\nabla u|)}{|\nabla u|}\nabla \di   u(\nabla u, \nabla u) \Delta u. \\[0.4cm]
			& = & \disp - \left(\frac{a'(|\nabla u|)}{|\nabla u|}\right)' \frac{\big[\nabla \di  u(\nabla u, \nabla u)\big]^2}{|\nabla u|} - \frac{a'(|\nabla u|)}{|\nabla u|} \langle \nabla \big( \nabla \di  u(\nabla u, \nabla u) \big), \nabla u \rangle \\[0.4cm]
			& &  \disp + \langle \nabla \left(\Delta_a u\right), \nabla u \rangle - \frac{a'(|\nabla u|)}{|\nabla u|}\nabla \di  u(\nabla u, \nabla u)\Delta u.
		\end{array}
	\end{equation}
	Inserting this in turn into $(\star)$ and simplifying, we obtain that
	$$
		\begin{array}{lcl}
			(\star) & = & \disp 2\frac{a'(|\nabla u|)}{|\nabla u|} \nabla \di u^{(2)}(\nabla u, \nabla u) + \disp 2a(|\nabla u|) \big[ |\nabla \di  u|^2 + \Ric(\nabla u, \nabla u)\big] \\[0.4cm]
			& & + 2\langle \nabla \Delta_a u, \nabla u \rangle.
		\end{array}
	$$
	Using then identities $2)$ and $4)$ in Lemma \ref{lem_basicomp},
	$$
		\begin{array}{lcl}
			(\star) & = & \disp \disp 2 a'(|\nabla u|)|\nabla u|\Big( u_{\nu\nu}^2 + |\nabla^\top   u_\nu|^2\Big) + 2\langle \nabla \Delta_a u, \nabla u \rangle \\[0.4cm]
			& & + \disp 2a(|\nabla u|) \big[u_{\nu\nu}^2 + 2|\nabla^\top   u_\nu|^2 + u_\nu^2|\II|^2 + \Ric(\nabla u, \nabla u)\big] \\[0.4cm]
			& = & \disp 2 \lambda_1 u_{\nu\nu}^2 + 2(\lambda_1+ \lambda_2) |\nabla^\top   u_\nu|^2 + 2\lambda_2 \Big[u_\nu^2|\II|^2 + \Ric(\nabla u, \nabla u)\Big] \\[0.4cm]
			& & \disp + 2\langle \nabla \Delta_a u, \nabla u \rangle,
		\end{array}
	$$
	proving \eqref{bochner_better}. To show \eqref{bochner_betterfornablau}, simply observe that expanding $(\star)$ and using $3)$ in Lemma \ref{lem_basicomp} we get
	$$
		(\star) = 2 |\nabla u| \diver \big( A(\nabla u) \nabla |\nabla u|\big) + 2\lambda_2 |\nabla^\top   u_\nu|^2 + 2 \lambda_1 u_{\nu\nu}^2.
	$$
	To prove that the inequalities hold weakly in $\Omega$, first observe that the vector fields 
	\[
		A(\nabla u)\nabla |\nabla u|^2, \qquad A(\nabla u) \nabla |\nabla u|
	\]
	are locally Lipschitz in $\{|\nabla u|>0\}$. Therefore, the pointwise equalities we just showed hold weakly against test functions supported there. Let now $0 \le \varphi \in \lip_c(\Omega)$. Fix a cutoff $0 \le \psi \in C^\infty_c(\R^+_0)$ supported in $[0,2)$, identically one in $[0,1]$ and with $\psi' \le 0$, and for $\eps > 0$ define the function
	\[
		\eta_\eps(t) = 1-\psi(t/\eps) \qquad \text{for } \, t \in [0,\infty).
	\]
	Let $(\star \star)$ be the right-hand side of \eqref{bochner_better}. We multiply \eqref{bochner_better} by $\eta_\eps(|\nabla u|^2)\varphi \in \lip_c(\{|\nabla u|>0\})$ and integrate by parts to get
	\[
		\begin{array}{lcl}
			\disp \int (\star\star)\eta_\eps(|\nabla u|^2)\varphi & = & \disp - \frac{1}{2} \int \langle A(\nabla u) \nabla |\nabla u|^2, \nabla \left( \eta_\eps(|\nabla u|^2)\varphi\right) \rangle \\[0.5cm]
			& \le & \disp - \frac{1}{2} \int \eta_\eps(|\nabla u|^2) \langle A(\nabla u) \nabla |\nabla u|^2, \nabla\varphi \rangle,  
		\end{array}
	\]
	where we used the ellipticity of $A$ and $\eta_\eps' \ge 0$. Letting $\eps \to 0$ we get
	\[
		\int_{\{|\nabla u|>0\}} (\star\star)\varphi \le \disp - \frac{1}{2} \int_{\{|\nabla u|>0\}} \langle A(\nabla u) \nabla |\nabla u|^2, \nabla\varphi \rangle.  
	\]
	We are left to prove that the domain of integration $\{|\nabla u|>0\}$ can be replaced by the entire $\Omega$. This is clear, in our assumptions on $a$, for the left-hand side. This is also the case of the right-hand side, since 
	\[
		\langle \nabla \Delta_a u, \nabla u \rangle = - f'(u)|\nabla u|^2 \in C(\Omega)
	\]
	and by Remark \ref{rem_stampa}. The case of \eqref{bochner_betterfornablau} is analogous.
\end{proof}

Next, we use the above Bochner's formula to get  a geometric Poincar\'e inequality. A related inequality was obtained in \cite{DPV}. The first step is the following proposition.

\begin{proposition} \label{prop_Poincare_loc}
	Let $\Omega \subseteq M$ be a $C^2$ domain with interior normal $\eta$, and let $u \in C^3(\Omega) \cap C^2(\overline\Omega)$ be a solution to 
	\[
		\Delta_a u + f(u) = 0,
	\]
	where $a$ satisfies \eqref{assu_A_superstrong} and $f \in C^1(\R)$. Let $\Omega_0\subseteq\Omega$ be a domain with locally finite perimeter in $\Omega$ such that 
	\begin{equation}\label{eq_patialomega0}
	\partial\Omega_0 \subseteq \partial\Omega \cup \critu. 
	\end{equation}
	If $w \in C^2(\Omega) \cap C^1(\overline{\Omega})$ solves
	\begin{equation} \label{poinloc0}
		w \ge 0 \ \ \text{in } \, \Omega_0, \qquad \diver \big( A(\nabla u) \nabla w\big) + f'(u)w \le 0 \qquad \text{weakly in } \, \Omega_0,
	\end{equation}
	then, setting $\lambda_j(t)$ as in \eqref{eigenvalues}, for each $\varphi \in \lip_c(M)$ and $\eps>0$ it holds
	\begin{equation} \label{poinloc}
		\begin{split}
			\int_{\Omega_0} |\nabla u|^2 \langle A(\nabla u)\nabla\varphi,\nabla\varphi\rangle & \geq \int_{\Omega_0} \varphi^2 \left\{ \lambda_1|\nabla^\top|\nabla u||^2 + \lambda_2 [|\nabla u|^2|\II|^2 + \Ric(\nabla u,\nabla u)] \right\} \\
			& \phantom{=\;} + \int_{\Omega_0} (w+\eps)^2 \langle A(\nabla u)\nabla\left(\frac{\varphi|\nabla u|}{w+\eps}\right),\nabla\left(\frac{\varphi|\nabla u|}{w+\eps}\right)\rangle \\
			& \phantom{=\;} - \int_{\Omega_0} \frac{\eps \varphi^2}{w+\eps} f'(u)|\nabla u|^2 \\
			& \phantom{=\;} + \int_{\overline\Omega_0\cap\partial\Omega} \varphi^2 \langle A(\nabla u)\nabla\frac{|\nabla u|^2}{2} - \frac{|\nabla u|^2}{w+\eps} A(\nabla u)\nabla w,\eta\rangle \,
		\end{split}
	\end{equation}
	where $\lambda_j = \lambda_j(|\nabla u|)$.
\end{proposition}

\begin{proof}
	Set $w_\eps = w+\eps$, $z = |\nabla u|/w_\eps$ and consider the vector field
	\[
		V = w_\eps |\nabla u| A(\nabla u) \nabla z \equiv \frac{1}{2} A(\nabla u)\nabla|\nabla u|^2 - \frac{|\nabla u|^2}{w_\eps} A(\nabla u)\nabla w
	\]
	which is continuous in $\overline{\Omega}$ and locally Lipschitz continuous in $\Omega$. From \eqref{bochner_betterfornablau} and \eqref{poinloc0} we have
	\[
		\diver\, V = \lambda_1|\nabla^\top|\nabla u||^2 + \lambda_2 [|\nabla u|^2|\II|^2 + \Ric(\nabla u,\nabla u)] - \frac{\eps}{w_\eps} f'(u) |\nabla u|^2 + w_\eps^2 \langle A(\nabla u)\nabla z,\nabla z\rangle \, 
	\]
	weakly in $\Omega$. Let $\varphi\in\lip_c(M)$ be given. We have
	\[
		\diver(\varphi^2 V) = \varphi^2 \diver\, V + 2 w_\eps \varphi |\nabla u| \langle A(\nabla u)\nabla z, \nabla \varphi\rangle
	\]
	and from the identity
	\[
	\begin{array}{lcl}	
	\disp \varphi^2 w_\eps^2 \langle A(\nabla u)\nabla z,\nabla z\rangle + 2 \varphi w_\eps |\nabla u| \langle A(\nabla u)\nabla z,\nabla \varphi\rangle & = & \disp w_\eps^2 \langle A(\nabla u)\nabla (\varphi z),\nabla(\varphi z)\rangle \\[0.3cm]
	& & \disp - |\nabla u|^2 \langle A(\nabla u)\nabla\varphi,\nabla\varphi\rangle
	\end{array}
	\]
	we get
	\begin{equation}
		\begin{split}
			\diver\,(\varphi^2 V) & \ge \varphi^2 \left\{\lambda_1|\nabla^\top|\nabla u||^2 + \lambda_2 [|\nabla u|^2|\II|^2 + \Ric(\nabla u,\nabla u)]\right\} - \frac{\eps \varphi^2}{w_\eps} f'(u)|\nabla u|^2 \\
			& \phantom{=\;} + w_\eps^2 \langle A(\nabla u)\nabla (\varphi z),\nabla(\varphi z)\rangle - |\nabla u|^2 \langle A(\nabla u)\nabla\varphi,\nabla\varphi\rangle \, .
		\end{split}
	\end{equation}
	The vector field $W = \varphi^2 V$ vanishes on $\partial\Omega_0 \cap \Omega$ and is compactly supported in $\overline{\Omega}$, so by the divergence theorem (see Lemma \ref{lem_divergence} in the appendix) we have
	\begin{align*}
		0 & \geq \int_{\Omega_0} \varphi^2 \left\{\lambda_1|\nabla^\top|\nabla u||^2 + \lambda_2 [|\nabla u|^2|\II|^2 + \Ric(\nabla u,\nabla u)]\right\} - \int_{\Omega_0} \frac{\eps \varphi^2}{w_\eps} f'(u)|\nabla u|^2 \\
		& \phantom{=\;} + \int_{\Omega_0} w_\eps^2 \langle A(\nabla u)\nabla (\varphi z),\nabla(\varphi z)\rangle - \int_{\Omega_0} |\nabla u|^2 \langle A(\nabla u)\nabla\varphi,\nabla\varphi\rangle \\
		& \phantom{=\;} + \int_{\overline{\Omega}_0\cap\partial\Omega} \varphi^2 \langle V,\eta\rangle
	\end{align*}
	which is \eqref{poinloc}.

\end{proof}

In the next Lemma, we deal with the boundary term and extend to a quasilinear setting a ``magical identity"  first discovered in \cite{Arma,Arma2}, see also \cite{FarMarVal,cmr}.

\begin{lemma}\label{lem_bordo}
	Let $\Omega\subseteq M$ be a $C^2$ domain and let $u \in C^2(\overline\Omega)$ be such that both $u$ and $c_j \doteq \partial_\eta u$ are constant on a connected component $\partial_j \Omega$ of $\partial\Omega$, and let $a$ satisfy \eqref{assu_A}. Then, for each vector field $X$ in a neighbourhood of $\partial_j \Omega$, the function $w= \langle \nabla u, X\rangle$ satisfies
	\begin{equation}\label{identity_quasilinear}
		\langle w A(\nabla u) \nabla \frac{|\nabla u|^2}{2} - |\nabla u|^2  A(\nabla u)\nabla w, \eta\rangle = - \lambda_1(|\nabla u|)|\nabla u|^2 c_j \langle \eta, \nabla_{\eta} X \rangle
	\end{equation}
	on $\partial_j \Omega$. In particular, if $X$ is Killing then the right hand side of \eqref{identity_quasilinear} vanishes.
\end{lemma}

\begin{proof}
	The identity is obvious if $\partial_\eta u = 0$, since in this case $\nabla u \equiv 0$ on $\partial_j \Omega$. Thus, assume $c_j \neq 0$, so that $c_j \eta = \nabla u$. Since $A(\nabla u)$ is symmetric and $A(\nabla u)\eta = \lambda_1(|\nabla u|) \eta$, we deduce
	\[
		\begin{array}{l}
			\disp \langle w A(\nabla u) \nabla \frac{|\nabla u|^2}{2} - |\nabla u|^2  A(\nabla u)\nabla w, \eta\rangle \\[0.3cm]
			= \ \disp \langle w \nabla \frac{|\nabla u|^2}{2} - |\nabla u|^2 \nabla w, \eta\rangle \lambda_1(|\nabla u|) \qquad \text{on } \, \partial_j \Omega.
	\end{array}
	\]
	The conclusion therefore follows from the identity 
	\begin{equation}\label{eq_FMV}
		\langle w \nabla \frac{|\nabla u|^2}{2} - |\nabla u|^2 \nabla w, \eta \rangle =  - |\nabla u|^2 c_j \langle \eta, \nabla_\eta X \rangle \qquad \text{on } \partial_j \Omega,
	\end{equation} 
	which was proved\footnote{There is an uninfluent sign mistake there, since the right-hand side of \eqref{eq_FMV} is written as $- |\nabla u|^3 \langle \eta, \nabla_\eta X \rangle$.} in \cite[Lemma 32]{FarMarVal}.
\end{proof}

We are ready to pass to limits as $\eps \to 0$ in Proposition \ref{prop_Poincare_loc}.

\begin{theorem} \label{teo_Poincare_loc}
	Let $\Omega \subseteq M$ be a $C^2$ domain and let $u \in C^3(\Omega) \cap C^2(\overline\Omega)$ be a stable solution to 
	\[
	\left\{ \begin{array}{ll}
		\Delta_a u + f(u)=0 & \quad \text{in } \, \Omega, \\[0.2cm]
		u, \partial_\eta u \, \text{ locally constant} & \quad \text{on } \, \partial \Omega, 
	\end{array}
	\right.
	\]
	where $f \in C^1(\R)$ and $a$ satisfies \eqref{assu_A_superstrong}. Let $\Omega_0\subseteq\Omega$ be a domain with locally finite perimeter in $\Omega$ such that $\partial\Omega_0 \subseteq \partial\Omega \cup \critu$, let $X$ be a Killing vector field in $\overline\Omega$ and suppose that $w = \langle \nabla u, X \rangle$ satisfies $w \ge 0$, $w \not\equiv 0$ in $\Omega_0$. Then for each $\varphi \in \lip_c(M)$ it holds
	\begin{equation} \label{piufinalissima_teo3}
		\begin{split}
			\int_{\Omega_0} |\nabla u|^2 \langle A(\nabla u)\nabla\varphi,\nabla\varphi\rangle & \geq \int_{\Omega_0} \varphi^2 \left\{ \lambda_1|\nabla^\top|\nabla u||^2 + \lambda_2 [|\nabla u|^2|\II|^2 + \Ric(\nabla u,\nabla u)] \right\} \\
			& \phantom{=\;} + \int_{\Omega_0} w^2 \langle A(\nabla u)\nabla\left(\frac{\varphi|\nabla u|}{w}\right),\nabla\left(\frac{\varphi|\nabla u|}{w}\right)\rangle \, .
		\end{split}
	\end{equation}
\end{theorem}

\begin{proof}
	Since $w$ satisfies $\diver(A(\nabla u)\nabla w) + f'(u) w = 0$ in $\Omega$ and $w\geq 0$, $w\not\equiv 0$ in $\Omega_0$, by the Harnack inequality (see Remark \ref{rem_harnackHopf}) we have $w>0$ in $\Omega_0$. From Lemma \ref{lem_bordo} (applied to any $C^1$ extension of $X$ in a neighbourhood of $\partial\Omega$) we know that
	\begin{equation} \label{w_bordo}
		\langle w A(\nabla u) \nabla \frac{|\nabla u|^2}{2},\eta\rangle = \langle |\nabla u|^2  A(\nabla u)\nabla w, \eta\rangle \qquad \text{on } \, \partial\Omega \, .
	\end{equation}
	We claim that
	\begin{equation} \label{w_bordo2}
		\{ x \in \overline{\Omega}_0 \cap \partial\Omega : w(x) = 0 \} \subseteq \{ x \in \overline{\Omega} : \nabla u(x) = 0 \} \, .
	\end{equation}
	Indeed, suppose by contradiction that there exists $x\in\overline{\Omega}_0 \cap \partial\Omega$ such that $w(x) = 0$ and $\nabla u(x)\neq 0$. Since $\partial\Omega_0 \subseteq \critu \cup \partial\Omega$ and by assumption we have $\nabla u\neq 0$ at $x$, there exists a neighbourhood $U$ of $x$ in $M$ such that $U\cap\critu = \emptyset$ and therefore $U\cap\partial\Omega_0 = U\cap\partial\Omega$. Since $\partial\Omega$ is $C^2$ regular, $\Omega_0$ satisfies an interior ball condition at $x$. Having $w>0$ in $\Omega_0$ and $w(x)=0$, by Hopf's lemma (see again Remark \ref{rem_harnackHopf}) we deduce $\partial_\eta w(x) > 0$. Then, we have
	\[
	\langle|\nabla u|^2 A(\nabla u)\nabla w,\eta\rangle = \lambda_1(|\nabla u|) |\nabla u|^2 \partial_\eta w \neq 0 \qquad \text{at } \, x
	\]
	where we used that $A(\nabla u)\eta = \lambda_1(|\nabla u|)\eta$ since $\eta$ is proportional to $\nabla u$ at $x$. But then the left-hand side of \eqref{w_bordo} must also be nonzero at $x$, contradicting the assumption $w(x) = 0$.
	
	Let $\varphi\in\lip_c(M)$ be given. From Proposition \ref{prop_Poincare_loc} and \eqref{w_bordo} we have
	\begin{equation} \label{poinalt}
		\begin{split}
			\int_{\Omega_0} |\nabla u|^2 \langle A(\nabla u)\nabla\varphi,\nabla\varphi\rangle & = \int_{\Omega_0} \varphi^2 \left\{ \lambda_1|\nabla^\top|\nabla u||^2 + \lambda_2 [|\nabla u|^2|\II|^2 + \Ric(\nabla u,\nabla u)] \right\} \\
			& \phantom{=\;} + \int_{\Omega_0} (w+\eps)^2 \langle A(\nabla u)\nabla\left(\frac{\varphi|\nabla u|}{w+\eps}\right),\nabla\left(\frac{\varphi|\nabla u|}{w+\eps}\right)\rangle \\
			& \phantom{=\;} - \int_{\Omega_0} \frac{\eps \varphi^2}{w+\eps} f'(u)|\nabla u|^2 \\
			& \phantom{=\;} + \int_{\overline\Omega_0\cap\partial\Omega} \frac{\eps\varphi^2}{w+\eps} \langle A(\nabla u)\frac{\nabla|\nabla u|^2}{2},\eta\rangle
		\end{split}
	\end{equation}
	for any $\eps>0$. We now let $\eps\to0$ to obtain \eqref{piufinalissima_teo3} from \eqref{poinalt}. Since $|\eps/(w+\eps)|<1$ in $\Omega_0$, by the Lebesgue convergence theorem we have
	\[
	\int_{\Omega_0} \frac{\eps \varphi^2}{w+\eps} f'(u)|\nabla u|^2 \to 0 \qquad \text{as } \, \eps \to 0 \, .
	\]
	Setting $E_0 = \{w=0\} \cap \overline\Omega_0\cap\partial\Omega$ and $E_1 = \{w>0\} \cap \overline\Omega_0\cap\partial\Omega$, we have
	\begin{align*}
		\int_{\overline\Omega_0\cap\partial\Omega} \frac{\eps\varphi^2}{w+\eps} \langle A(\nabla u)\frac{\nabla|\nabla u|^2}{2},\eta\rangle & = \int_{E_0} \varphi^2 \langle A(\nabla u)\frac{\nabla|\nabla u|^2}{2},\eta\rangle \\
		& \phantom{=\;} + \int_{E_1} \frac{\eps\varphi^2}{w+\eps} \langle A(\nabla u)\frac{\nabla|\nabla u|^2}{2},\eta\rangle \, .
	\end{align*}
	When letting $\eps\to0$, the integral over $E_1$ converges to $0$ by the Lebesgue dominated convergence theorem. On the other hand, by \eqref{w_bordo2} we have $\nabla u = 0$ and therefore $\nabla|\nabla u|^2 = 0$ on $E_0$, so the integral over $E_0$ also vanishes and we get
	\[
	\int_{\overline\Omega_0\cap\partial\Omega} \frac{\eps\varphi^2}{w+\eps} \langle A(\nabla u)\frac{\nabla|\nabla u|^2}{2},\eta\rangle \to 0 \qquad \text{as } \, \eps \to 0 \, .
	\]
	Finally, by explicit computation of $\nabla(\varphi|\nabla u|/(w+\eps))$ and Fatou's lemma we have
	\[
	\liminf_{\eps\to0} \int_{\Omega_0} (w+\eps)^2 \langle A(\nabla u)\nabla\left(\frac{\varphi|\nabla u|}{w+\eps}\right),\nabla\left(\frac{\varphi|\nabla u|}{w+\eps}\right)\rangle \geq \int_{\Omega_0} w^2 \langle A(\nabla u)\nabla\left(\frac{\varphi|\nabla u|}{w}\right),\nabla\left(\frac{\varphi|\nabla u|}{w}\right)\rangle
	\]
	and this concludes the proof.
\end{proof}

\begin{remark}
	In the above proof, since $w>0$ in $\Omega_0$ then $\critu\cap\Omega_0 = \emptyset$, whence $\Omega_0$ is a connected component of $\Omega\backslash\critu$.
\end{remark}

Theorem \ref{teo_Poincare_loc} will be used later to infer that solutions with moderate energy growth satisfy $\nabla^\top|\nabla u| \equiv 0$ and $\II \equiv 0$ in $\Omega_0$. Once this is achieved, the next proposition provides a general splitting result, of independent interest.

\begin{proposition}[\textbf{The splitting lemma}] \label{prop_splitting_alt}
	Let $(M,\metric)$ be a complete Riemannian manifold and $\Omega\subseteq M$ a $C^1$ domain. Let $a$ satisfy \eqref{assu_A}, $f \in C^1(\R)$ and let $u\in C^1(\overline{\Omega})$ be a non-constant weak solution to
	 
	\[
		\left\{
			\begin{array}{ll}
				\Delta_a u + f(u) = 0 & \qquad \text{in } \, \Omega \\[0.2cm]
				u \text{ locally constant} & \qquad \text{on } \, \partial\Omega \, .
			\end{array}
		\right.
	\]
	If $\Omega_0$ is a connected component of $\Omega\backslash\critu$ such that
	\begin{equation} \label{Ddu_trivial}
		u \in C^2(\Omega_0), \qquad \big|\nabla^\top  |\nabla u|\big|^2 \equiv 0 \qquad \text{and} \qquad |\II|^2 \equiv 0 \qquad \text{in } \, \Omega_0
	\end{equation}
	then:
	\begin{itemize}
		\item [i)] there exists an isometry $\Psi : I \times N \to \Omega_0$, where $I\subseteq\R$ is an open interval and $N$ is a complete, totally geodesic hypersurface of $M$;
		\item [ii)] $u(\Psi(t,x)) = u_0(t)$ for all $(t,x)\in I\times N$, for a suitable function $u_0 : I \to \R$ such that $u_0'>0$.
	\end{itemize}
\end{proposition}

\begin{remark}
	Observe that we assume no regularity at all on the the portion $\partial \Omega_0 \cap \critu$. The fact that the latter is totally geodesic, in particular smooth, is a consequence of the proof.	
\end{remark}

\begin{proof}
	By the identities 1) and 2) in \eqref{identities} and by \eqref{Ddu_trivial}, the only nontrivial component of $\nabla \di u$ in $\Omega_0$ is the one in the direction of $\nu = \nabla u/|\nabla u|$:
	\begin{equation}\label{eq_Ddu}
		\nabla \di u = \nabla \di u\left(\nu,\nu\right) \frac{\di u}{|\nabla u|} \otimes \frac{\di u}{|\nabla u|} \qquad \text{in } \, \Omega_0 .
	\end{equation}
	A straightforward computation allows us to deduce from \eqref{eq_Ddu} that $|\nabla u|$ is locally constant on level sets of $u$ intersected with $\Omega_0$ and that $\nu$ is a parallel vector field. Indeed, from the standard identity $\di|\nabla u|^2 = 2\nabla\di u(\nabla u,\,\cdot\,)$ we have $\di|\nabla u| = \nabla\di u(\nu,\,\cdot)$ in $\{\nabla u\neq 0\}$, so by \eqref{eq_Ddu} we get
	\[
		\nabla|\nabla u| = \nabla\di u(\nu,\nu) \nu \qquad \text{in } \, \Omega_0
	\]
	from which we have
	\begin{align*}
		\langle\nabla_X \nu,Y\rangle & = \frac{\langle\nabla_X \nabla u,Y\rangle}{|\nabla u|} - \frac{\langle X,\nabla|\nabla u|\rangle\langle\nabla u,Y\rangle}{|\nabla u|^2} \\
		& = \frac{\nabla\di u(X,Y)}{|\nabla u|} - \frac{\nabla\di u(\nu,\nu)\langle\nu, X\rangle\langle\nabla u,Y\rangle}{|\nabla u|^2} \\
		& = \frac{1}{|\nabla u|} \left(\nabla\di u(X,Y) - \nabla\di u(\nu,\nu)\langle\nu,X\rangle\langle\nu,Y\rangle\right) = 0
	\end{align*}
	for any pair of tangent vectors $X,Y\in T_x M$, $x\in\Omega_0$, that is, $\nabla\nu = 0$ in $\Omega_0$. In particular, integral curves of $\nu$ are geodesics in $M$. By the very definition of $\II$ and by \eqref{Ddu_trivial} we also have that level sets of $u$ intersected with $\Omega_0$ are totally geodesic hypersurfaces in $M$. 
	
	Fix $b\in u(\Omega_0)\backslash u(\partial\Omega)$, which exists since $u(\partial\Omega)$ is at most countable and $u$ is non-constant. Let $N \subseteq \Omega_0$ be a connected component of $\Sigma_b \doteq \{x\in\Omega_0 : u(x) = b\}$ and let $c>0$ be the constant value of $|\nabla u|$ on $N$. Since $\Sigma_b$ is closed in $\Omega_0$ and $N$ is a connected component of $\Sigma_b$, we have that $N$ is closed in $\Sigma_b$, hence in $\Omega_0$. We claim that $N$ is also closed in $M$. This is equivalent to saying that the closure $\overline{N}$ of $N$ in $M$ is contained in $\Omega_0$, that is, that it does not intersect $\partial\Omega_0\equiv\overline{\Omega}_0\backslash\Omega_0$. Indeed, by continuity of $u$ and $|\nabla u|$ in $\overline{\Omega}$ we have $u\equiv b$ and $|\nabla u|\equiv c$ on $\overline{N}$; since $\partial\Omega_0$ is contained in $\critu \cup \partial\Omega$, $u\neq b$ on $\partial\Omega$ and $|\nabla u| = 0 \neq c$ on $\critu$, we have $\overline{N}\cap\partial\Omega_0 = \emptyset$ as claimed. Furthermore, $N$ is also embedded and totally geodesic in $\Omega$, therefore in $M$. In particular, $N$ is also complete. Let
	\[
		\Psi : \R \times N \to M : (t,x) \mapsto \exp_x(t\nu_x) \, .
	\]
	Note that $\Psi$ is well defined and surjective by completeness of $M$ and closedness of $N$, and that $\Psi^{-1}(\Omega_0)$ is open in $\R\times N$. For each $x\in N$ we denote by $I_x \subseteq \R$ the maximal open interval containing $0$ such that $\Psi(t,x) \in \Omega_0$ for all $t\in I_x$ and we denote by $\gamma_x : I_x \to \Omega_0$ the geodesic curve given by
	\[
		\gamma_x(t) = \Psi(t,x) \qquad \text{for all } \, t \in I_x \, .
	\]
	For each $x\in N$ we also set
	\[
		t_\ast(x) = \inf I_x \, , \qquad t^\ast(x) = \sup I_x \, .
	\]
	Note that if $t_\ast(x) > -\infty$ then $\Psi(t_\ast(x),x) \in \partial\Omega_0$. Similarly, $\Psi(t^\ast(x),x) \in \partial\Omega_0$ if $t^\ast(x) < \infty$.
	
	We define
	\[
		\mathscr D = \{ (t,x) : x \in N, \, t \in I_x \} \subseteq \Psi^{-1}(\Omega_0) \, .
	\]
	Our ultimate goal is to show that $\mathscr D = I \times N$ for a fixed open interval $I\subseteq\R$, that $\Psi : \mathscr D \to \Omega_0$ is an isometry when $\mathscr D$ is equipped with the product metric $\di t^2 + \metric_{|N}$ and that $u$ is of the form $u(\Psi(t,x)) = u_0(t)$ for some function $u_0 : I \to \R$.
	
	First, we show that $\mathscr D$ is open in $\Psi^{-1}(\Omega_0)$, and therefore also in $\R\times N$. Suppose, by contradiction, that this is false, that is, suppose that $\Psi^{-1}(\Omega_0)\backslash\mathscr D$ is not closed in $\Psi^{-1}(\Omega_0)$. Then there exist $(t,x) \in \mathscr D$ and a sequence $\{(t_n,x_n)\} \in \Psi^{-1}(\Omega_0)\backslash\mathscr D$ such that $(t_n,x_n) \to (t,x)$. Suppose, without loss of generality, that $t>0$. Then we also have $t_n > 0$ for all sufficiently large $n\in\N$. Since $(t_n,x_n)\not\in\mathscr D$, that is, $t_n \not\in I_{x_n}$, this implies that $t^\ast(x_n) \leq t_n$ for all sufficiently large $n$. Since $0 \leq t^\ast(x_n) \leq t_n$ and $t_n \to t < \infty$, the sequence $t^\ast(x_n)$ is bounded, so there exists a subsequence $x_{n_k}$ such that $t^\ast(x_{n_k})$ converges to some limit $\bar t \in [0,t]$. By continuity, we have $\Psi(t^\ast(x_{n_k}),x_{n_k}) \to \Psi(\bar t,x)$ as $k\to\infty$. Since $\bar t \in [0,t] \subseteq I_x$, we have $\Psi(\bar t,x) \in \Omega$. But then $\{\Psi(t^\ast(x_{n_k}),x_{n_k})\}$ is a sequence of points of $\partial\Omega$ converging to a point of $\Omega$, absurd.
	
	From \eqref{eq_Ddu} we deduce that the vector field $\nu$ is parallel in $\Omega_0$. We let
	\[
		\Phi : E \subseteq \R \times \Omega_0 \to \Omega_0
	\]
	be the maximally defined flow of $\nu$. By the very definition of $\Psi$ and $\mathscr D$ and by parallelism of $\nu$, we have $\mathscr D \subseteq E$ and
	\[
		\Psi(t,x) = \Phi(t,\Psi(0,x)) \equiv \Phi(t,x) \qquad \text{for all } \, (t,x) \in \mathscr D \, .
	\]
	Moreover, since $\nu$ is a parallel vector field and since $\Psi \equiv \Phi_\nu$ on $\mathscr D$, the pullback $\Psi_{|\mathscr D}^\ast \metric$ of the metric $\metric$ of $\Omega_0$ induced by $\Psi_{|\mathscr D}$ on $\mathscr D$ is the product metric
	\[
		\di t^2 + \metric_{|N} \, .
	\]
	Indeed, for each $(t,x)\in\mathscr D$ we have
	\begin{align*}
		(\Psi_{|\mathscr D}^\ast \metric)(\partial_t,\partial_t) & = \langle\Psi_\ast\partial_t,\Psi_\ast\partial_t\rangle \underset{(i)}{=} \langle\nu,\nu\rangle = 1 \\
		(\Psi_{|\mathscr D}^\ast \metric)(\partial_t,X) & = \langle\Psi_\ast\partial_t,\Psi_\ast X\rangle = \langle(\Phi_t)_\ast\nu,(\Phi_t)_\ast X\rangle \underset{(ii)}{=} \langle\nu,X\rangle = 0 \\
		(\Psi_{|\mathscr D}^\ast \metric)(X,Y) & = \langle\Psi_\ast X,\Psi_\ast Y\rangle = \langle(\Phi_t)_\ast X,(\Phi_t)_\ast Y\rangle \underset{(iii)}{=} \langle X,Y\rangle = \langle X,Y\rangle_N
	\end{align*}
	for any $X,Y\in TN$, where $(i)$ holds because $\Psi$ is a restriction of the exponential map and $(ii)$-$(iii)$ hold because the map $y \mapsto \Phi_t(y) \doteq \Phi(t,y)$ is a local isometry by parallelism of $\nu$. In particular, $\Psi_{|\mathscr D}$ is an open map and a local diffeomorphism onto its image (in fact, we will show that it is also injective, hence a global diffeomorphism) and $\Psi(\mathscr D)$ is open in $\Omega_0$.
	
	Let $c>0$ be the constant value of $|\nabla u|$ on $N$ and let $y_0 : J \to \R$ be the maximal solution to the Cauchy problem
	\[
	\begin{cases}
		[a(y')+y'a'(y')] y'' + f(y) = 0 \\
		y(0) = b \\
		y'(0) = c
	\end{cases}
	\]
	subject to the constraint
	\[
		y_0'(t) \neq 0 \qquad \text{for all } \, t \in J \, .
	\]
	We claim that $I_x \subseteq J$ for all $x\in N$ and that $u(\Psi(t,x)) = y_0(t)$ for all $x\in N$, $t\in I_x$. For a given $x\in N$, consider the function $u_x = u\circ\gamma_x : I_x \to \R$. We have
	\[
		u_x' = |\nabla u|\circ\gamma_x > 0 \, , \qquad u_x'' = \nabla\di u(\nu,\nu) \circ \gamma_x
	\]
	so $u_x$ satisfies
	\[
		[a(u_x') + u_x' a'(u_x')] u_x'' + f(u_x) = 0 \qquad \text{on } \, I_x
	\]
	and $u_x(0) = b$, $u_x'(0) = c$. Hence, the claim follows by uniqueness of the local solution to the Cauchy problem for the ODE
	\[
		[a(y')+y'a'(y')]y'' + f(y) = 0
	\]
	for any initial datum $(y(t_0),y'(t_0)) \in \R\times\R^+$.
	
	We now show that the intervals $I_x$, $x\in N$, are all equal to some given interval $I\subseteq\R$ independent from $x$. This is equivalent to saying that for each $t\in\R$ the set
	\[
		N_t = \{ x \in N : t \in I_x \}
	\]
	is either empty or equal to $N$. We prove this latter claim. Let $t\in\R$ be such that $N_t\neq\emptyset$, then we prove that $N_t = N$ by showing that $N_t$ is both open and closed in the connected $N$. The first property is an immediate consequence of openness of $\mathscr D$: since $\mathscr D$ is open in $\R\times N$, the intersection $\mathscr D \cap (\{t\}\times N)$ is also open in $\{t\}\times N$; since $\mathscr D \cap (\{t\}\times N) \simeq N_t$ and $\{t\}\times N \simeq N$ via the trivial homeomorphism $\psi : \{t\}\times N \to N : (t,x) \mapsto x$, it follows that $N_t$ is open in $N$. We now prove that $N_t$ is closed in $N$. Without loss of generality, assume that $t>0$. Let $\{x_n\} \subseteq N_t$ be a sequence converging to some point $x\in N$. Suppose, by contradiction, that $x\not\in N_t$. Then we have $t^\ast(x) \leq t < t^\ast(x_n)$. In particular, $t^\ast(x) < \infty$ and so $\Psi(t^\ast(x),x) \in \partial\Omega_0 \subseteq \critu\cup\partial\Omega$. By continuity,
	\[
		|\nabla u|(\Psi(t^\ast(x),x)) = \lim_{n\to\infty} |\nabla u|(\Psi(t^\ast(x),x_n)) = \lim_{n\to\infty} u_{x_n}'(t^\ast(x)) = y_0'(t^\ast(x)) > 0.
	\]
	Therefore, $\nabla u\neq 0$ at $\Psi(t^\ast(x),x)$ and we have $\Psi(t^\ast(x),x) \in \partial\Omega$. Since $u$ is locally constant on $\partial\Omega$, the gradient $\nabla u = |\nabla u|\dot\gamma_x$ must be orthogonal to $\partial\Omega$ at $\Psi(t^\ast(x),x)$, that is, the curve $s \mapsto \Psi(s,x) \equiv \gamma_x(s)$ meets $\partial\Omega$ orthogonally at time $s = t^\ast(x)$. Since $\partial\Omega$ is an embedded $C^1$ hypersurface of $M$, for each large enough $n\in\N$ the curve $s \mapsto \Psi(s,x_n)$ must intersect $\partial\Omega$ for some value $s_n\in\R$ of $s$ such that $s_n \to t^\ast(x)$ as $n\to\infty$. Since $t^\ast(x) \leq t < t^\ast(x_n) \leq s_n$, by comparison it must be $t^\ast(x_n) \to t^\ast(x)$ as $n\to\infty$, hence $\Psi(t^\ast(x_n),x_n) \to \Psi(t^\ast(x),x)$. Since $\partial\Omega$ is $C^1$ and embedded, for all sufficiently large $n\in\N$ the point $\Psi(t^\ast(x_n),x_n)$ belongs to the same connected component of $\partial\Omega$ containing $\Psi(t^\ast(x),x)$. Since $u$ is constant along that component of $\partial\Omega$, it must be
	\[
		u(\Psi(t^\ast(x),x)) = u(\Psi(t^\ast(x_n),x_n))
	\]
	for all sufficiently large $n\in\N$. By continuity, we have
	\[
		u(\Psi(t^\ast(x),x)) = \lim_{s\to t^\ast(x)} u(\Psi(s,x)) = \lim_{s\to t^\ast(x)} y_0(s) = y_0(t^\ast(x)) \, .
	\]
	On the other hand, since $y_0$ is strictly increasing and $t^\ast(x) \leq t < t^\ast(x_n)$, we also have
	\[
		y_0(t^\ast(x)) < \lim_{s\to t^\ast(x_n)} y_0(s) = \lim_{s\to t^\ast(x_n)} u(\Psi(s,x_n)) = u(\Psi(t^\ast(x_n),x_n))
	\]
	for each $n\in\N$, contradiction. This shows that $N_t$ is also closed. By connectedness of $N$ and since $N_t\neq\emptyset$ we deduce $N_t=N$.
	
	So far, we have proved that
	\[
		\mathscr D = I \times N
	\]
	for some open interval $I\subseteq J$, that $\Psi_{|\mathscr D} : \mathscr D \to \Omega_0$ is a local isometry (when $\mathscr D$ is equipped with the product metric), that $\Psi(\mathscr D)$ is open in $\Omega_0$ and that $u$ satisfies $u(\Psi(t,x)) = y_0(t)$ for all $(t,x) \in \mathscr D$ for a fixed function $y_0 : J \to \R$. Since $y_0$ is strictly increasing on $J$ and integral curves of $\nu$ are disjoint, we further deduce that $\Psi_{|\mathscr D}$ is injective, hence a global isometry onto its image. We now prove that $\Psi(\mathscr D)$ is also closed in $\Omega_0$, thus concluding $\Psi(\mathscr D) = \Omega_0$ by connectedness of $\Omega_0$. Let $\{p_n\}\subseteq \Psi(\mathscr{D})$ be a given sequence converging to some point $p \in \Omega_0$. We have to show that $p \in \Psi(\mathscr{D})$. We consider the (unique) sequence $\{(t_n,x_n)\} \subseteq \mathscr D$ such that $p_n = \Psi(t_n,x_n)$ for each $n\in\N$. Let $U\subseteq\Omega_0$ be a geodesically convex open neighbourhood of $p$, which exists by openness of $\Omega_0$. For all $n,m\in\N$ large enough we have $p_n,p_m\in U$ and there exists a minimizing geodesic $\gamma = \gamma_{n,m} : [0,1] \to U$ such that $\gamma(0) = p_n$ and $\gamma(1) = p_m$. Let
	\[
		v = \dot\gamma(0) - \langle\dot\gamma(0),\nu\rangle \nu
	\]
	be the component of $\dot\gamma(0)$ orthogonal to $\nu$ and let $\sigma : [0,1] \to N$ be the geodesic with initial point $x_n$ and initial velocity $(\Psi(t_n,\,\cdot\,)^{-1})_\ast v \in T_{x_n} N$, which exists by completeness of $N$. Consider the curve $c : [0,1] \to M$ defined by
	\[
		c(s) = \Psi(t_n+s\langle\dot\gamma(0),\nu\rangle,\sigma(s)) \, .
	\]
	We claim that $c(s) = \gamma(s)$ for all $s\in[0,1]$. Indeed, by construction we have $c(0) = \gamma(0)$ and $\dot c(0) = \dot\gamma(0)$. Let
	\[
		\bar s = \sup\{ s_0 \in [0,1] : t_n + s\langle\dot\gamma(0),\nu\rangle \in I \text{ for all } s \in [0,s_0] \} \, .
	\]
	The restriction $c_{|[0,\bar s]} : [0,\bar s] \to \Omega_0$ is a geodesic, and by uniqueness of geodesics we have $c(s) = \gamma(s)$ for all $s\in[0,\bar s]$. Clearly $\bar s > 0$, since $t_n \in I$ and $I$ is open. In fact, we have $\bar s = 1$, since otherwise $t_n + \bar s\langle\dot\gamma(0),\nu\rangle$ would be a boundary point for $I$, yielding $\gamma(\bar s) = c(\bar s) \in \partial\Omega_0$, contradiction. So, we have $c(s) = \gamma(s)$ for all $s\in[0,1]$, as claimed. In particular,
	\[
		\Psi(t_n + \langle\dot\gamma(0),\nu\rangle,\sigma(1)) = c(1) = \gamma(1) = \Psi(t_m,x_m)
	\]
	and by injectivity of $\Psi_{|\mathscr D}$ this implies
	\[
		t_n + \langle\dot\gamma(0),\nu\rangle = t_m \qquad \text{and} \qquad \sigma(1) = x_m \, .
	\]
	Since $\gamma([0,1])$ lies entirely in $\Psi(\mathscr D)$, we have
	\[
		\dist_\Omega(p_n,p_m)^2 = |\dot\gamma(s)|^2 = \langle\dot\gamma(0),\nu\rangle^2 + |\dot\sigma(s)|^2 \geq |t_n-t_m|^2 + \dist_N(x_n,x_m)^2 .
	\]
	Hence, $\{t_n\}$ and $\{x_n\}$ are Cauchy sequences in $\R$ and $N$, respectively, so we have the existence of $(t,x) \in \R\times N$ such that $(t_n,x_n) \to (t,x)$. By continuity of $\Psi$ we also have $p_n = \Psi(t_n,x_n) \to \Psi(t,x)$, so $p = \Psi(t,x)$. If $t\not\in I$ then we would have either $t = \sup I$ or $t = \inf I$; in both cases, $\Psi(t,x_n)$ would be a sequence of points of $\partial\Omega_0$ converging to the point $p\in\Omega_0$, absurd. Hence, it must be $t\in I$ and so $p\in\Psi(\mathscr D)$, as desired.
\end{proof}

We are ready to prove Theorem \ref{teo_splitting_intro} in the Introduction, in the following more general form:

\begin{theorem} \label{teo_splitting}
	Let $(M^n, \metric)$ be a complete Riemannian manifold. Let $\Omega$ be a $C^2$ domain such that $\Ric\geq 0$ in $\Omega$. Let $a$ satisfy \eqref{assu_A_superstrong}, $f \in C^1(\R)$ and let $u \in C^3(\Omega) \cap C^2(\overline\Omega)$ be a non-constant, stable solution to
	\[
	\left\{
	\begin{array}{l@{\qquad}l}
		\Delta_a u + f(u) = 0 & \text{in } \Omega \\[0.3cm]
		u, \partial_\eta u \text{ locally constant} & \text{on } \partial\Omega
	\end{array}
	\right.
	\]
	with moderate energy growth. Let $\Omega_0$ be a connected component of $\Omega\backslash\critu$ with locally finite perimeter in $\Omega$. Assume that $X$ is a bounded Killing vector field in  $\overline\Omega$ such that
	\begin{equation}\label{eq_monotonoc}
		\langle \nabla u, X \rangle > 0 \qquad \text{in } \, \Omega_0.
	\end{equation}
	Then,
	\begin{itemize}
		\item [i)] there exists an isometry $\Psi : I \times N \to \Omega_0$, where $I\subseteq\R$ is an open interval and $N$ is a complete, totally geodesic hypersurface of $M$ with $\Ric_N \geq 0$;
		\item [ii)] $u(\Psi(t,x)) = u_0(t)$ for all $(t,x)\in I\times N$, for some $u_0 : I \to \R$ with $u_0'>0$;
		\item [iii)] $\langle \nu,X \rangle$ is a positive constant in $\Omega_0$, where $\nu = \frac{\nabla u}{|\nabla u|}$.
	\end{itemize}
\end{theorem}

\begin{proof}[Proof of Theorem \ref{teo_splitting}]
Setting $w \doteq \langle \nabla u, X \rangle$, by Theorem \ref{teo_Poincare_loc} we have
	\begin{align*}
		\int_{\Omega_0} |\nabla u|^2 \langle A(\nabla u) \nabla \varphi, \nabla \varphi \rangle & \geq \int_{\Omega_0} \varphi^2 \Big\{\lambda_1 \big|\nabla^\top  |\nabla u|\big|^2 +\lambda_2\Big[|\nabla u|^2|\II|^2 +\Ric(\nabla u, \nabla u)\Big]\Big\} \\[0.2cm]
		& \phantom{=\;} + \int_{\Omega_0} w^2 \langle A(\nabla u) \nabla \left(\frac{\varphi|\nabla u|}{w}\right), \nabla \left(\frac{\varphi|\nabla u|}{w}\right) \rangle
	\end{align*}
	for all $\varphi\in\lip_c(M)$. 
	Testing the inequality with $\varphi = \varphi_j$ in the definition of moderate energy growth, and letting $j \to \infty$, we obtain
	\[
		|\nabla^\top  |\nabla u|\big|^2 \equiv 0 \, , \qquad |\II|^2 \equiv 0 \, , \qquad \Ric(\nabla u, \nabla u) \equiv 0 \, , \qquad \nabla \frac{|\nabla u|}{w} \equiv 0 \qquad \text{in } \Omega_0 \, .
	\]
	Proposition \ref{prop_splitting_alt} ensures the validity of clauses $i)$ and $ii)$ of the statement, 
	while $iii)$ follows by observing that
	\[
		\frac{|\nabla u|}{w} = \frac{|\nabla u|}{\langle\nabla u,X\rangle} = \frac{1}{\langle \nu,X\rangle} \, .
	\]
\end{proof}

\begin{proof}[Proof of Theorem \ref{teo_splitting_intro}]
	Condition \eqref{condi_Hcj_con_a} guarantees that 
	\[
	\langle \nabla u, X \rangle \ge 0, \qquad \langle \nabla u, X \rangle \not\equiv 0 \qquad \text{on } \, \partial \Omega,
	\]
	and by Corollary \ref{cor_lowenergy} $u$ has moderate energy growth. Hence, using Theorem \ref{teo_monotonicity} we deduce 
	\[
	\langle \nabla u, X \rangle > 0 \qquad \text{in } \, \Omega,
	\]
	and we conclude by Theorem \ref{teo_splitting} applied to $\Omega_0 = \Omega$.
\end{proof}

\section{Euclidean space and the critical set of $u$}\label{sec_critical}

Let $u \in C^3(\Omega)$ be a non-constant solution to $\Delta_a u + f(u) = 0$. In this section, we examine more closely the case of Euclidean space and provide some results that allow to ``glue" different solutions along boundaries on which $\nabla u$ vanishes. Towards this aim, we assume as usual $f \in C^1(\R)$ and \eqref{assu_A_superstrong}, that is, 
\[
a\in C^{1,1}(\R_0^+), \qquad a(t) > 0 \qquad \text{and} \qquad ta'(t) + a(t) > 0 \qquad \text{for all } \, t \in \R^+_0
\]
to guarantee a certain amount of regularity for the set of critical points $\critu$. In fact, each partial derivative $w_j = \partial_j u$ belongs to $C^2(\Omega)$ and satisfies the linearized equation
\begin{equation}\label{eq_partialder}
\diver(A(\nabla u)\nabla w_j) + f'(u) w_j = 0 \qquad \text{in } \, \Omega \, ,
\end{equation}
where in our assumptions $A(\nabla u)$ has $\lip_\loc$ coefficients and $f'(u) \in C(\Omega)$. Therefore, the unique continuation principle holds and so either $w_j \equiv 0$ or $w_j$ has finite order of vanishing at every point 
\[
x \in Z_j = \{ x \in \Omega : w_j(x) = 0 \} \supseteq \critu.
\]
In particular, if $u$ is constant in an open subset $\Omega_0$ of $\Omega$, then $u$ is constant in the entire $\Omega$.

\begin{lemma} \label{Om_finper}
	Let $\Omega\subseteq\R^n$ be a $C^1$ domain, Let $a$ satisfy \eqref{assu_A_superstrong} and let $f \in C^1(\R)$. Consider a non-constant solution $u \in C^3(\Omega)$ to
	\[
	\Delta_a u + f(u) = 0 \qquad \text{in } \, \Omega \, .
	\]
	If $\Omega_0\subseteq\Omega$ is an open set such that $\partial\Omega_0 \subseteq \partial\Omega \cup \critu$
	then $\mathscr{H}^{n-1}(K\cap\partial\Omega_0) < \infty$ for all compact sets $K\subseteq\Omega$. In particular, $\Omega_0$ is a set of locally finite perimeter in $\Omega$.
\end{lemma}


\begin{proof}
	Let $K\subseteq\Omega$ be a compact set. Since $u$ is non-constant, there exists some $j\in\{1,\dots,n\}$ such that $w_j = \partial_j u \not\equiv 0$ in $\Omega$. As $\critu \subseteq Z_j = \{ w_j = 0 \}$, it suffices to show that $\mathscr H^{n-1}(K\cap Z_j) < \infty$.
	As observed at the beginning of this section, $w_j$ satisfies the linear elliptic equation \eqref{eq_partialder} and has finite order of vanishing at every point of $Z_j$. By \cite[Theorem 1.7]{HS89}, for each point $x\in Z_j$ there exist $\rho_x>0$ and $C=C(n,x)>0$ such that $B_{\rho_x}(x) \subseteq \Omega$ and
	\[
	\mathscr{H}^{n-1}(B_\rho(x) \cap Z_j) \leq C \rho^{n-1} \qquad \text{for all } \, 0 < \rho \leq \rho_x \, .
	\]
	By covering the compact set $K\cap Z_j$ with a finite number of balls $B_{\rho_{x_1}}(x_1),\dots,B_{\rho_{x_k}}(x_k)$ centered at points $x_1,\dots,x_k$ of $K\cap Z_j$ we deduce $\mathscr{H}^{n-1}(K \cap Z_j) < \infty$, as desired.
	
	By \cite[Theorem 4.5.11]{federer} (also \cite[Proposition 3.62]{ambrosiofuscopallara} together with the introductory remark), since $\mathscr H^{n-1}(K\cap\partial\Omega_0) < \infty$ for all compact $K\subseteq\Omega$ it follows that $\Omega_0$ has finite perimeter in every compact set $K\subseteq\Omega$, as claimed.
\end{proof}

We are ready to prove our main results of this section.

\begin{lemma} \label{lem_Du0}
	Let $\Omega\subseteq\R^n$ be a $C^1$ domain, let $a$ satisfy \eqref{assu_A_superstrong} and let $f \in C^1(\R)$. Assume that $u \in C^3(\Omega)\cap C^1(\overline{\Omega})$ is a non-constant, stable solution to
	\[
		\Delta_a u + f(u) = 0 \qquad \text{in } \Omega.
	\]
	%
	%
%
	Let $\Omega_0\subseteq\Omega$ be a connected open subset such that $\nabla u = 0$ everywhere on $\partial\Omega_0$. If $u$ has moderate energy growth in $\Omega_0$, then $\Omega_0 \equiv \Omega$ and there exists $X\in\mathbb{S}^{n-1}$ such that $\langle\nabla u,X\rangle > 0$ in $\Omega$. In particular, $\critu = \emptyset$.
\end{lemma}

\begin{proof}
	By unique continuation, see the beginning of this section, since $u$ is non-constant in $\Omega$ then it must be non-constant in $\Omega_0$. Each partial derivative $w_j = \partial_j u$ satisfies \eqref{eq_partialder} and $w_j \equiv 0$ on $\partial\Omega_0$. Since $w_j^2 \leq |\nabla u|^2$ and $u$ has moderate energy growth in $\Omega_0$, by Lemma \ref{Lemma_w-} we deduce that $w_j$ either vanishes identically in $\Omega_0$ or it has a constant strict sign in $\Omega_0$. Since $u$ is non-constant in $\Omega_0$, for some index $j\in\{1,\dots,n\}$ we have $w_j \not \equiv 0$ and thus either $w_j>0$ or $w_j<0$ everywhere in $\Omega_0$. Then, taking either $X = e_j$ or $X = -e_j$ we have $\langle\nabla u,X\rangle > 0$ everywhere in $\Omega_0$. We now show that $\Omega_0 = \Omega$. Suppose, by contradiction, that $\Omega_0\subsetneq\Omega$. Then there exists $x\in\Omega\cap\partial\Omega_0\neq\emptyset$. The stability of $u$ guarantees the existence of $0<v\in C^{1,\alpha}_\loc(\Omega)$ satisfying
	\[
		\diver(A(\nabla u)\nabla v) + f'(u)v = 0 \qquad \text{in } \, \Omega \, .
	\]
	By Remark \ref{rem-Lemma_w-} we have $w_j = \lambda v$ in $\Omega_0$ for some $\lambda\in\R\backslash\{0\}$. By continuity, $w_j(x) = \lambda v(x) \neq 0$, contradiction.
\end{proof}

\begin{theorem}\label{teo_twopieces}
	Let $\Omega\subseteq\R^n$ be a $C^2$ domain with disconnected boundary. Let $a$ satisfy \eqref{assu_A_superstrong}, let $f \in C^1(\R)$ and let $u \in C^3(\Omega) \cap C^2(\overline\Omega)$ be a non-constant, stable solution of moderate energy growth to
	\[
	\left\{
	\begin{array}{l@{\qquad}l}
		\Delta_a u + f(u) = 0 & \text{in } \Omega \\[0.3cm]
		u, \partial_\eta u \text{ locally constant} & \text{on } \partial\Omega.
	\end{array}
	\right.
	\]
	Suppose that, for $j=1,2$, there exists $X_j \in \mathbb{S}^{n-1}$ such that
	\[
	\langle \eta, X_j \rangle \ge 0, \quad \langle \eta,X_j \rangle \not\equiv 0 \qquad \text{on } \, \partial_j \Omega, 
	\]
	and also assume that $\partial_\eta u = 0$ on the (possibly empty) set $\partial\Omega \backslash (\partial_1\Omega \cup \partial_2\Omega)$. If $\critu\neq\emptyset$, then $\nabla u \neq 0$ everywhere on $\partial_1\Omega \cup \partial_2\Omega$. 
	Moreover, if there exists $\Gamma\subseteq\critu$ such that $\Omega\backslash\Gamma$ is disconnected then $\partial_1\Omega$ and $\partial_2\Omega$ are parallel hyperplanes, $\partial\Omega = \partial_1\Omega \cup \partial_2\Omega$, $u$ is 1D and $\critu$ is a hyperplane parallel to each $\partial_j\Omega$.
\end{theorem}

\begin{remark}\label{rem_keyR2}
	In dimension $n=2$ the condition $\langle \eta, X_j \rangle \not \equiv 0$ in $\partial_j\Omega$ can be omitted. Indeed, if $\langle \eta, X_j \rangle \equiv 0$ then $X_j$ is tangent to $\partial_j\Omega$, hence $\partial_j\Omega$ is a line. One can therefore rotate $X_j$ to find $\widetilde X_j \in \mathbb{S}^{1}$ such that $\langle \eta, \widetilde X_j \rangle >0$ on  $\partial_j\Omega$.
\end{remark}

\begin{proof}
	We first show that if $\critu\neq\emptyset$ then $\nabla u\neq 0$ everywhere on $\partial_1\Omega \cup \partial_2\Omega$. Suppose by contradiction that this is false. Then, denoting by $c_j$ the constant value of $\partial_\eta u$ on $\partial_j\Omega$, we either have $c_j = 0$ for each $j$ or, up to switching index 1 and 2, $c_1 \neq 0$ and $c_j = 0$ for $j \ge 2$.
	
	\textbf{Case 1.} Suppose that $c_j = 0$ for each $j$. Then $\nabla u = 0$ everywhere on $\partial\Omega$. Since $u$ is non-constant in $\Omega$, by Lemma \ref{lem_Du0} applied with $\Omega_0 = \Omega$ we get $\critu=\emptyset$, contradiction.
	
	
	\textbf{Case 2.} Suppose that $c_1\neq 0$ and $c_j=0$ for $j \ge 2$. By our assumptions on $\partial_1\Omega$, the inner product
	\[
		\langle\nabla u,X_1\rangle = \langle\nabla u,\eta\rangle \langle\eta,X_1\rangle = c_1 \langle\eta, X_1\rangle
	\]
	does not change sign and does not vanish identically on $\partial_1\Omega$, while $\langle\nabla u,X_1\rangle = 0$ everywhere on $\partial\Omega\backslash\partial_1\Omega$. Up to replacing $X_1$ with $-X_1$, we can ensure that
	\[
	\langle \nabla u, X_1 \rangle \ge 0, \quad \langle \nabla u,X_1 \rangle \not\equiv 0 \qquad \text{on } \, \partial \Omega. 
	\]
	We can therefore apply Theorem \ref{teo_monotonicity} to deduce that $\langle\nabla u,X_1\rangle > 0$ everywhere in $\Omega$, yielding again $\critu=\emptyset$ and therefore the desired contradiction.
	 
	Now, suppose that there exists $\Gamma \subset \critu$ disconnecting $\Omega$. Then $\Omega\setminus\overline{\Gamma}$ is also disconnected. Indeed, since $\Omega\setminus\Gamma$ is disconnected there exist open sets $U_1,U_2\subseteq\R^n$ such that
	\[
		\Omega\backslash\Gamma \subseteq U_1\cup U_2 \, , \qquad U_1\cap U_2\cap(\Omega\backslash\Gamma) = \emptyset \, , \qquad U_i \cap (\Omega\backslash\Gamma) \neq \emptyset \quad \text{for } i = 1,2.
	\]
	Since $\Omega\backslash\overline{\Gamma} \subseteq \Omega\backslash\Gamma$ we also have
	\[
		\Omega\backslash\overline{\Gamma} \subseteq U_1\cup U_2 \, , \qquad U_1\cap U_2\cap(\Omega\backslash\overline{\Gamma}) = \emptyset
	\]
	Moreover, since $(\Omega\backslash\Gamma)\backslash(\Omega\backslash\overline{\Gamma}) = \Omega\cap(\overline{\Gamma}\backslash\Gamma) \subseteq \critu$ and $\critu$ has empty interior by non-constancy of $u$ and unique continuation, 
	we also have
	\[
		\qquad U_i \cap (\Omega\backslash\overline{\Gamma}) \neq \emptyset \quad \text{for } i = 1,2.
	\]
	This shows that $\Omega\backslash\overline{\Gamma}$ is disconnected, as claimed. Since $\nabla u=0$ everywhere on $\overline{\Gamma}$ by continuity of $|\nabla u|$ in $\overline{\Omega}$ and, on the other hand, $\nabla u\neq0$ everywhere on $\partial_1\Omega\cup\partial_2\Omega$, we have for $j=1,2$ that each $\partial_j\Omega$ is disjoint from the closed set $\overline{\Gamma}$ and therefore entirely contained in the boundary of some connected component of $\Omega\backslash\overline{\Gamma}$. On the other hand, each component of $\Omega\backslash\overline{\Gamma}$ has boundary contained in $\partial\Omega \cup \overline{\Gamma}$, and $\partial_1\Omega \cup \partial_2\Omega$ is precisely the subset of $\partial\Omega \cup \overline{\Gamma}$ where $\nabla u\neq 0$. By Lemma \ref{lem_Du0}, $\nabla u$ cannot vanish identically on the boundary of a component of $\Omega\backslash\overline{\Gamma}$, otherwise we would deduce $\critu=\emptyset$, contradiction. Summarizing,  we conclude that $\Omega\backslash\overline{\Gamma}$ has exactly two connected components, call them $\Omega_1$ and $\Omega_2$, and without loss of generality we can assume $\partial_j\Omega\subseteq\partial\Omega_j$ for $j=1,2$.

	For $j=1,2$, the inner product $\langle\nabla u,X_j\rangle = c_j\langle\eta,X_j\rangle$ does not change sign and does not vanish identically on $\partial_j\Omega$, while $\langle\nabla u,X_j\rangle = 0$ everywhere else on $\partial\Omega_j$. Up to replacing $X_j$ with $-X_j$ we can assume that
	\[
	\langle \nabla u, X_j \rangle \ge 0, \quad \langle \nabla u,X_j \rangle \not\equiv 0 \qquad \text{on } \, \partial \Omega_j, 
	\]
	so indeed $\langle\nabla u,X_j\rangle >0$ in $\Omega_j$ because of  Theorem \ref{teo_monotonicity}. By Lemma \ref{Om_finper}, each $\Omega_j$ has locally finite perimeter in $\Omega$. We then apply Theorem \ref{teo_splitting} to conclude that each $\Omega_j$ is a slab in $\R^n$ bounded by two parallel hyperplanes, and in each of them $u$ is strictly monotone in the direction perpendicular to the boundary. Since $\Omega_1 \cap \Omega_2 = \emptyset$, the hyperplanes bounding $\Omega_1$ and $\Omega_2$ must all be parallel. In particular, $\partial_1\Omega$ and $\partial_2\Omega$ are parallel hyperplanes and the set $\Omega\cap\overline{\Gamma} = \Omega \backslash (\Omega_1\cup\Omega_2)$, having empty interior, must be a hyperplane parallel to both $\partial_1\Omega$ and $\partial_2\Omega$. In particular, $\Omega$ is the open slab bounded between $\partial_1\Omega$ and $\partial_2\Omega$, so $\partial\Omega$ has no other component besides them. Lastly, since $\Omega\cap\overline{\Gamma}\subseteq\critu$ and $\nabla u\neq 0$ everywhere in $\Omega_1\cup\Omega_2$ we have $\critu = \Omega\cap\overline{\Gamma}$, so $\critu$ is a hyperplane parallel to both $\partial_j\Omega$ as claimed.
\end{proof}

\section{Proof of the main theorems for capillary graphs}\label{sec_proof}

\begin{proof}[Proof of Theorem \ref{teo_main_MC_manifolds}]
	Under assumptions (A), (B) or (C) we apply items $(ii)$, $(i)$ or $(v)$, respectively, in Theorem \ref{teo_goodcutoff_MC} for the choice $a(t) = 1/\sqrt{1+t^2}$ to infer, in view of \eqref{eq_quadraticgrowth}, that $u$ has moderate energy growth. On the other hand, under assumption (D) we apply Theorem \ref{prop_grad} to deduce that $|\nabla u|\in L^\infty(\Omega)$ and then Proposition \ref{prop_moderate_boundgradient} to conclude, again, that $u$ has moderate energy growth. Condition \eqref{condi_Hcj} can be rewritten as $\langle \nabla u, X \rangle \ge 0$ and $\langle \nabla u, X \rangle \not \equiv 0$ on $\partial \Omega$, whence $\langle\nabla u,X\rangle>0$ in $\Omega$ by Theorem \ref{teo_monotonicity}. The conclusion follows by Theorem \ref{teo_splitting}.
\end{proof}

\begin{proof}[Proof of Theorem \ref{teo_CMC_intro_R2}]
	First, observe that by regularity theory $u$ is analytic in $\Omega$. Under any of the assumptions (i)-(iv) we have that $|\nabla u| \in L^\infty(\partial\Omega)$, since $|\nabla u| = |\partial_\eta u|$ is locally constant on $\partial\Omega$ and the subset $\partial_\star \Omega \subseteq \partial\Omega$ where $\partial_\eta u\neq 0$ has finitely many connected components. Hence, $|\nabla u| \in L^\infty(\Omega)$ by Theorem \ref{prop_grad}. Clearly, any subset $\Omega \subseteq \R^2$ satisfies $|\Omega \cap B_R| \le CR^2$, so $u$ has moderate energy growth by Proposition \ref{prop_moderate_boundgradient}.
	
	$(i)$. Since $\partial\Omega$ is connected we have that $\partial_\eta u = c_1$ is constant on $\partial\Omega$. Moreover, by Remark \ref{rem_keyR2} we can assume that $\langle \eta, v \rangle \ge 0$ and $\not \equiv 0$ on $\partial \Omega$. We claim that $c_1\neq 0$. In this case, by Theorem \ref{teo_monotonicity} we obtain $\langle \nabla u, X \rangle > 0$ in $\Omega$ with $X = v$. A direct application of Theorem \ref{teo_splitting} yields that $\Omega = (0,\infty) \times \R$ and that $u$ is 1D. Hence, necessarily $H=0$ and $\Omega$, $\GG_u$ are half-planes. To show the claim, if it were $c_1 = 0$ then $\nabla u$ would vanish identically on $\partial\Omega$ and we could apply Lemma \ref{lem_Du0} with $\Omega_0 \equiv \Omega$ to deduce the existence of $X \in \mathbb{S}^1$ such that $\langle \nabla u, X \rangle >0$ in $\Omega$. Then, again, by Theorem \ref{teo_splitting} we would have that $\Omega = (0,\infty) \times \R$ and that $u$ is a non-constant affine function, contradicting the vanishing of $\nabla u$ on $\partial\Omega$.
	
	$(ii)$. By assumption $c_1 = \partial_\eta u \neq 0$, $c_j = 0$ for each $j\neq 1$. By Remark \ref{rem_keyR2}, up to changing the sign of $v$ and rotating it we can assume $\langle\eta,v\rangle > 0$ on $\partial_1\Omega$. Then, assumption \eqref{condi_Hcj} in Theorem \ref{teo_main_MC_manifolds} is satisfied for $X = v$ and the conclusion follows by a direct application of that theorem (note that, in particular, the a priori disconnected set $\partial\Omega$ has exactly two connected components and $\partial_\eta u = c_2 = 0$ on $\partial_2\Omega$; this excludes the possibility that $u$ be affine, forcing $H\neq 0$ and $\GG_u$ to be a piece of cylinder). 
	
	$(iii)$. If $\langle\eta,v\rangle\not\equiv 0$ on $\partial_\star\Omega$ then we apply Theorem \ref{teo_main_MC_manifolds} for $X$ equal to either $v$ or $-v$ to deduce that $\Omega =(0,T)\times\R$ and $u$ is 1D and strictly monotone, hence $\GG_u$ is a piece of a cylinder or of a half-plane. On the other hand, if $\langle\eta,v\rangle\equiv 0$ on $\partial_\star\Omega$ then $\partial_1\Omega$ and $\partial_2\Omega$ are parallel lines in $\R^2$, both tangent to $v$, and then the connected set $\Omega$ must be contained in the strip bounded by them. In particular $\eta$ points in opposite directions on $\partial_1\Omega$ and $\partial_2\Omega$, so there exists $w\in\mathbb{S}^1$ such that $\langle\eta,w\rangle>0$ on $\partial_1\Omega$ and $\langle\eta,w\rangle<0$ on $\partial_2\Omega$ and again we reach the desired conclusion by applying Theorem \ref{teo_main_MC_manifolds} with $X$ equal to either $w$ or $-w$.
	
	$(iv)$. 
	Since $u$ has moderate energy growth, we apply Theorem \ref{teo_twopieces} (together with Remark \ref{rem_keyR2}) to deduce that $\Omega$ is a slab $(0,T) \times \R$, $\Gamma$ is a straight line and $u$ is 1D. Since the graph $\GG_u$ is a surface of constant mean curvature and $\partial_\eta u\neq0$ on $\partial \Omega$, $\GG_u$ must be a piece of cylinder.
	
	We now prove that if $\inf_\Omega u < b_j$ for each $j\geq 3$ and $\Gamma$ has an accumulation point in $\overline{\Omega}$, then indeed $\Gamma$ disconnects $\Omega$. 
	By the \L ojasiewicz structure theorem (see \cite{KP}), the analytic set $\Gamma$ 
	can be decomposed as
	\[
		\Gamma = \Gamma_0\sqcup\Gamma_1,
	\]
	where $\Gamma_0$ is a finite set and $\Gamma_1$ is a finite union of real analytic curves embedded in $\Omega$. The assumption that $\Gamma$ has an accumulation point grants that $\Gamma_1\neq\emptyset$. Since $u$ attains an interior minimum, by the strong maximum principle it must be $H < 0$ and so we have $|\nabla \di u(x)|\neq 0$ for each $x \in \Omega$. We can thus apply \cite[Corollary 3.4]{BCM} to deduce that the closure of each component $\gamma$ of $\Gamma_1$ still belongs to $\Gamma_1$, and therefore $\gamma$ is a real analytic curve properly embedded in $\Omega$. We claim that $\gamma$ is properly embedded in the entire $M$. Otherwise, there would exist a sequence $\{p_j\} \subset \gamma$ with $p_j \to p \in \partial \Omega$, and since $u \in C^1(\overline\Omega)$ we would get $u(p) = \inf_\Omega u$ and $\nabla u(p) = 0$. The strict inequality $\inf_\Omega u < b_j$ for $j \ge 3$ guarantees that this is not possible. To conclude, by the smooth Jordan-Brouwer separation theorem (see \cite{Lima}) $\gamma$ would disconnect $\R^2$, hence also $\Omega$.
%
\end{proof}

\begin{proof}[Proof of Theorem \ref{teo_CMC_intro_R3}]
	Again, by regularity theory $u$ is analytic in $\Omega$. As in the proof of Theorem \ref{teo_CMC_intro_R2}, under any of the assumptions $(i)-(iv)$ we have that $|\nabla u| \in L^\infty(\Omega)$ and then $u$ has moderate energy growth by Proposition \ref{prop_moderate_boundgradient} and assumption \eqref{eq_quadraticgrowth}.
	
	$(i)$. Assume by contradiction that $\partial \Omega$ is connected. Then, proceeding verbatim as in $(i)$ of Theorem \ref{teo_CMC_intro_R2} above, we deduce that $\Omega = (0,\infty) \times \R^2$. This in turn would imply $|\Omega\cap B_R|\sim\frac23\pi R^3$ as $R\to \infty$, a contradiction.	
%
%
	
	$(ii)$-$(iii)$-$(iv)$. We proceed as in cases $(ii)$, $(iii)$ and $(iv)$ of Theorem \ref{teo_CMC_intro_R2}.	
	Concerning case $(iv)$, we only have to prove that $\Gamma$ disconnects $\Omega$ in case $\inf_\Omega u < b_j$ for any $j\geq 3$ and $\haus^2(\Gamma) > 0$. As observed in the proof of $(iv)$ of Theorem \ref{teo_CMC_intro_R2}, by the strong maximum principle it must be $H < 0$ and thus $\nabla \di u \neq0$ everywhere in $\Omega$. Applying \L ojasiewicz's structure theorem together with \cite[Corollary 3.4]{BCM}, from $\haus^{2}(\Gamma)>0$ we deduce that the top stratum $\Gamma_2$ of $\Gamma$ is non-empty, and that any of its components is a complete real analytic surface without boundary, properly embedded in $\Omega$. The conclusion then follows again by the smooth Jordan-Brower theorem.
\end{proof}

\begin{proof}[Proof of Theorem \ref{teo_CMC_intro_R2_gen}]
	In cases $(i)$, $(ii)$ and $(iii)$ one can always deduce that $u$ has moderate energy growth by suitable application of either Theorem \ref{teo_goodcutoff_MC} or Proposition \ref{prop_moderate_boundgradient} together with Theorem \ref{prop_grad}, according to which condition is assumed among $(A)$, $(B)$, $(C)$ or $(D)$. Then, the proof proceeds as that of cases $(i)$, $(ii)$ and $(iii)$ of Theorem \ref{teo_CMC_intro_R2}. Case $(iv)$ requires a bit of care.
	
	$(iv)$. First, note that each connected component of $\Omega\setminus\Gamma$ has locally finite perimeter in $\Omega$ by Lemma \ref{Om_finper}. If $(D)$ holds then $|\nabla u|\in L^\infty(\Omega)$ by Theorem \eqref{prop_grad} and $u$ has moderate energy growth by Proposition \ref{prop_moderate_boundgradient}, so the desired conclusion follows from Theorem \ref{teo_twopieces} together with Remark \ref{rem_keyR2}. Suppose, instead, that $(A)$ or $(C)$ holds. Let $\partial_\star \Omega = \partial\Omega \cap \{\partial_\eta u \neq 0\}$. If $u$ is constant on $\partial_\star\Omega$, then $u$ has moderate energy growth in $\Omega$ by either item $(ii)$ or $(v)$ in Theorem \ref{teo_goodcutoff_MC} and the conclusion follows from Theorem \ref{teo_twopieces}. On the other hand, suppose that $u$ is not constant on $\partial_\star \Omega$. Each of the connected sets $\partial_1\Omega$ and $\partial_2\Omega$ is entirely contained in the boundary of some connected component of $\Omega\backslash\Gamma$; conversely, no connected component of $\Omega\backslash\Gamma$ can be such that $\nabla u = 0$ everywhere on its boundary, otherwise $u$ would be of moderate energy growth on that component by Theorem \ref{teo_goodcutoff_MC} (note that in this scenario the connected component $\Omega_0$ under discussion would satisfy $\overline{\Omega}_0 \cap \partial_\star\Omega = \emptyset$, so any item in Theorem \ref{teo_goodcutoff_MC} would be applicable) and then we would reach a contradiction by Lemma \ref{lem_Du0}. Therefore, $\Omega\backslash\Gamma$ consists of exactly two connected components $\Omega_1$ and $\Omega_2$ satisfying $\partial_j\Omega \subseteq \partial\Omega_j$ for $j=1,2$. Now, for each $j\in\{1,2\}$ we have that $u$ is constant on $\overline{\Omega}_j \cap \partial_\star\Omega \equiv \partial_j\Omega$, so by item $(ii)$ of Theorem \ref{teo_goodcutoff_MC} we deduce that $u$ has moderate energy growth in each of them. From this point on, we can conclude again as in the proof of Theorem \ref{teo_twopieces}.
	
	The only claim left to prove is that $\Gamma$ disconnects $\Omega$ whenever $u_\ast \doteq \inf_\Omega u > b_j$ for all $j\geq 3$, $f$ is analytic in a neighbourhood of $u_\ast$, $f(u_\ast) \neq 0$ and $\Gamma$ has an accumulation point in $\overline{\Omega}$. To this aim, we need only show that the top stratum $\Gamma_1$ of $\Gamma$ is a finite union of real analytic curves properly embedded in $\Omega$, and then the conclusion will follow by Jordan-Brouwer theorem as in the proof of Theorem \ref{teo_CMC_intro_R2}. The assumption $f(u_\ast) \neq 0$, the analyticity of $f$ on $I_\eps \doteq (u_\ast-\eps, u_\ast+\eps)$ and the fact that $\Gamma$ has an accumulation point enable us to use \cite[Corollary 3.4]{BCM} to deduce that the top stratum $\Gamma_1$ of $\Gamma$ is non-empty and given by a finite union of real analytic curves properly embedded in $\Omega \cap u^{-1}(I_\eps)$. Each of these curves is also properly embedded in $\Omega$. Indeed, if at least one of them were not properly embedded in $\Omega$, parametrizing it by $\gamma : \R \to \Omega$ there would exist a sequence of points $t_j \to \infty$ such that $x_j \doteq \gamma(t_j) \to x \in \Omega$. Since $\gamma(\R)$ is properly embedded in $\Omega \cap u^{-1}(I_\eps)$, we would have $u(x) = u_\ast-\eps$, contradicting the fact that $u(x_j)= u_\ast$ and the continuity of $u$. This concludes the proof of the claim.
\end{proof}

\section*{Appendix}

We here prove a form of the divergence theorem which is suited for our purposes. The main point requiring some care is that the subset $\Omega_0$ considered below just has locally finite perimeter in $\Omega$ and not in $\overline\Omega$, while the vector field $W$ does not have compact support in $\Omega$.

\begin{lemma} \label{lem_divergence}
	Let $M$ be a complete Riemannian manifold, $\Omega\subseteq M$ a $C^1$ domain with interior normal vector $\eta$ and let $\Omega_0\subseteq\Omega$ be an open set with locally finite perimeter in $\Omega$. Let $W$ be a compactly supported vector field in $\overline\Omega$, continuous in $\overline{\Omega}$ and locally Lipschitz in $\Omega$, satisfying $W=0$ on $\partial\Omega_0\cap\Omega$. If $h\in L^1(\Omega)$ is such that $\diver\,W \geq h$ weakly in $\Omega$, then
	\begin{equation} \label{GGloc}
		\int_{\overline{\Omega}_0\cap\partial\Omega} \langle W,\eta\rangle + \int_{\Omega_0} h \leq 0 .
	\end{equation}
\end{lemma}

\begin{proof}
	Let $K = \{x\in M : \dist(x, {\rm spt}\, W) < 1\}$. Let
	\[
		\Psi \, : \, \R \times \partial\Omega \to M, \qquad \Psi(t,x) = \exp_x(t\eta_x) \, .
	\]
	Since $\partial\Omega$ is $C^1$ regular and $K$ is relatively compact and open in $M$, there exists $\tau_0\in(0,1]$ such that the restriction of $\Psi$ to the open set
	\[
	E := (-\tau_0,\tau_0) \times (K\cap\partial\Omega)
	\]
	is injective and realizes a $C^1$ diffeomorphism onto $\Psi(E)$. From our choice of $\tau_0$, we have that the distance $r$ from $\partial \Omega$ is $C^1$ regular in the set
	\[
		U := K \cap \{x\in\Omega : 0 < r(x) < \tau_0\}
	\]
	and
	\begin{equation} \label{rPsiX}
		U\subseteq\Psi(E) \qquad \text{and} \qquad r = \pi_\R \circ (\Psi_{|E})^{-1} \quad \text{in } \, U
	\end{equation}
	where $\pi_\R : \R \times \partial\Omega \to \R$ denotes the projection onto the first factor. For $\tau \in (0, \tau_0/2)$, consider the cut-off function $\psi\in\lip(M)$ given by 
	\begin{equation} \label{def_psiX}
		\psi(x) = \left\{ \begin{array}{ll}
			0 & \quad \text{if } \, r(x) < \tau \\[0.1cm]
			\dfrac{r(x)-\tau}{\tau} & \quad \text{if } \, r(x) \in [\tau, 2\tau] \\[0.3cm]
			1 & \quad \text{if } \, r(x) > 2\tau.
		\end{array}\right.
	\end{equation}
	Applying the divergence theorem to the vector field $\psi W$, which is compactly supported in $\Omega$ and globally Lipschitz, we get
	\begin{equation} \label{div_X}
		0 = \int_{\Omega_0} \diver (\psi W) = \int_{\Omega_0} \langle W, \nabla \psi \rangle  + \int_{\Omega_0} \psi \, \diver W \geq \int_{\Omega_0} \langle W, \nabla \psi \rangle  + \int_{\Omega_0} \psi h \, .
	\end{equation}
	Using the definition of $\psi$ and the coarea formula,
	\begin{equation} \label{GGloc1}
		\int_{\Omega_0} \langle W, \nabla\psi\rangle = \frac{1}{\tau} \int_\tau^{2\tau} \int_{\Omega_0\cap\{r=t\}}  \langle W, \nabla r\rangle \, \di t \, .
	\end{equation}
	We claim that when passing to limits as $\tau\to0$ it holds
	\begin{align}
		\label{GGloc2a}
		\int_{\Omega_0} \psi h & \to \int_{\Omega_0} h \\
		\label{GGloc2b}
		\frac{1}{\tau} \int_\tau^{2\tau} \int_{\Omega_0\cap\{r=t\}} 
		\langle W, \nabla r\rangle \, \di t & \to \int_{\overline{\Omega}_0\cap\partial\Omega} \langle W,\eta\rangle	
	\end{align}
	so that \eqref{GGloc} follows from \eqref{div_X} and \eqref{GGloc1}. Clearly \eqref{GGloc2a} holds by the Lebesgue convergence theorem since $h\in L^1(\Omega)$, so we focus on justifying \eqref{GGloc2b}, which requires a bit of care. Fix $b>0$ and consider
	\[
	F_b = K \cap \partial \Omega \cap \{|W| > b\}, \qquad F^b = K \cap \partial \Omega \cap \{|W| \leq b\} \, .
	\] 
	For $t \in (0, \tau_0)$, define also $F_b(t) = \Psi(t, F_b)$ and $F^b(t) = \Psi(t, F^b)$, which are disjoint sets being $\Psi(t,\cdot)$ a diffeomorphism. Then, for each $t\in(0,\tau_0)$
	\begin{equation} \label{eq_int}
		\disp \int_{\Omega_0\cap\{r=t\}} \langle W, \nabla r\rangle = \int_{\Omega_0 \cap F_b(t)} \langle W, \nabla r\rangle + \int_{\Omega_0\cap F^b(t)} \langle W, \nabla r\rangle .
	\end{equation}
	Since the function $|W|$ is continuous and compactly supported in $\overline{\Omega}$, it is also uniformly continuous. Denoting by $\sigma : \R^+ \to \R^+$ its modulus of continuity, we have then
	\begin{equation} \label{in_sigma}
		|W| \geq b - \sigma(t) \quad \text{on } \, F_b(t) \qquad \text{and} \qquad |W| \leq b + \sigma(t) \quad \text{on } \, F^b(t)
	\end{equation}
	for each $t\in(0,\tau_0)$. Fix $t_b\in(0,\tau_0)$ small enough so that $\sigma(t_b) < b$. Then, for each $t\in(0,t_b)$ we have $\sigma(t) < b$, hence $|W|>0$ on $F_b(t)$ by \eqref{in_sigma} and from this together with the assumption that $W=0$ on $\partial\Omega_0 \cap \Omega$ we infer that $F_b(t) \cap \partial\Omega_0 = \emptyset$. In particular, the flow $\Psi$ starting from points of $\overline{\Omega}_0 \cap F_b$ does not hit $\partial\Omega_0 \cap \Omega$ before time $t_b$ and we have
	\[
	\Omega_0 \cap F_b(t) \equiv \Psi(t,\overline{\Omega}_0 \cap F_b) \qquad \text{for all } \, t \in (0,t_b) \, .
	\] 
	The family of hypersurfaces $t \mapsto \Omega_0 \cap F_b(t)$ is thus continuous in the $C^1$ topology and converges to $\overline\Omega_0 \cap F_b$ as $t\to0$. Therefore, we have
	\[
	\lim_{t\to0} \int_{\Omega_0 \cap F_b(t)} \langle W, \nabla r \rangle = \int_{\overline{\Omega}_0\cap F_b} \langle W, \eta\rangle
	\]
	and from the integral mean value theorem we obtain
	\begin{equation}\label{eq_int2}
		\lim_{\tau\to0} \frac{1}{\tau} \int_\tau^{2\tau} \int_{\Omega_0\cap F_b(t)} \langle W, \nabla r \rangle \, \di t = \int_{\overline\Omega_0\cap F_b} \langle W, \eta\rangle \, .
	\end{equation}
	On the other hand, for each $t\in(0,\tau_0)$ we have
	\[
	\left| \int_{\Omega_0\cap F^b(t)} \langle W, \nabla r\rangle \right| \le (b+\sigma(t)) \haus^{n-1}(F^b(t))
	\]
	by \eqref{in_sigma}, and since $\haus^{n-1}(F^b(t)) \to \haus^{n-1}(F^b) \leq \haus^{n-1}(K\cap\partial\Omega) < \infty$ as $t \to 0$ we get
	\begin{equation}\label{eq_int1}
		\limsup_{\tau\to0} \left| \frac{1}{\tau} \int_\tau^{2\tau} \int_{\Omega_0\cap F^b(t)} \langle W, \nabla r \rangle \, \di t \right| \le C b
	\end{equation}
	with $C>0$ independent of $b$. Then, from \eqref{eq_int}, \eqref{eq_int2} and \eqref{eq_int1} we obtain
	\[
		\limsup_{\tau\to0} \left| \frac{1}{\tau} \int_\tau^{2\tau} \int_{\Omega_0\cap\{r=t\}} \langle W,\nabla\psi\rangle \, \di t - \int_{\overline{\Omega}_0\cap F_b} \langle W,\eta\rangle \right| \leq Cb
	\]
	for each $b>0$. Then \eqref{GGloc2b} follows by letting $b\to0$ and noting that
	\[
		\lim_{b\to0} \int_{\overline{\Omega}_0\cap F_b} \langle W,\eta\rangle = \int_{\overline{\Omega}_0\cap \partial\Omega} \langle W,\eta\rangle
	\]
	by Lebesgue's convergence theorem.
\end{proof}

\begin{theorem} \label{thm_prop_grad}
	Let $(M^n,\metric)$ be a complete Riemannian manifold and $\Omega\subseteq M$ a domain. Let $f\in C^1(\R)$ and let $u\in C^3(\Omega) \cap C^1(\overline{\Omega})$ be a solution to
	\[
		\MM[u] + f(u) = 0 \qquad \text{in } \, \Omega \, .
	\]
	Suppose that
	\[
		\sup_\Omega |f(u)| < \infty \, , \qquad f'(u) \leq 0 \, , \qquad \inf_\Omega u > -\infty
	\]
	and that
	\[
		\Ric \geq 0 \quad \text{in } \, \Omega , \qquad \Sec \geq - \kappa^2 \quad \text{in } \, M
	\]
	for some $\kappa\in\R^+$. Then
	\[
		\sup_\Omega |\nabla u| \leq \max\left\{ \sqrt{n/2}, \, \sup_{\partial\Omega} |\nabla u|\right\} .
	\]
\end{theorem}

\begin{proof}
	
	Without loss of generality, we can assume $u\geq 0$ in $\Omega$ and choose $C_0>0$ so that
	\[
		|f(u)| \leq C_0 \qquad \text{in } \, \Omega \, .
	\]
	Let $F : \Omega \to \R\times M$ be the graph map given by $F(x) = (u(x),x)$ for each $x\in\Omega$ and denote by $g$ the Riemannian metric on $\Omega$ obtained by pulling back via $F$ the ambient product metric on $\R\times M$. We have
	\[
		g = \metric + \di u \otimes \di u \, .
	\]
	In components, if we write
	\[
		\metric = \sigma_{ij} \, \di x^i \otimes \di x^j, \qquad (\sigma^{ij}) = (\sigma_{ij})^{-1}, \qquad \di u = u_i \, \di x^i, \qquad u^i = \sigma^{ij} u_j
	\]
	then the components $g_{ij}$ of the metric $g$ and of the inverse matrix $(g^{ij}) = (g_{ij})^{-1}$ are
	\begin{equation} \label{gij}
		g_{ij} = \sigma_{ij} + u_i u_j \qquad \text{and} \qquad g^{ij} = \sigma^{ij} - \frac{u^i u^j}{W^2}
	\end{equation}
	where $W = \sqrt{1+|\nabla u|^2}$. We denote by $\nabla^g$, $\|\cdot\|$, $\diver_g$ and $\Delta_g$ the gradient, tensor norm, divergence and Laplace-Beltrami operator associated to $g$, respectively. For any $\varphi\in C^2(\Omega)$ we have (see, for instance, formula (25) in \cite{cmr})
	\begin{equation} \label{Lapg}
		\Delta_g \varphi = g^{ij} \varphi_{ij} - \varphi_i u^i \frac{H}{W}
	\end{equation}
	where $\varphi_i$ and $\varphi_{ij}$ are the components of $\di\varphi$ and $\nabla\di\varphi$, respectively, and $H$ is the (non-normalized) mean curvature of the graph of $u$ in $\R\times M$ in the direction of the upward normal vector $\mathbf{n}$. From the Jacobi equation we have
	\[
		\Delta_g \frac{1}{W} + (\|\II\|^2 + \overline{\Ric}(\mathbf{n},\mathbf{n})) \frac{1}{W} + g(\nabla^g H,\nabla^g u) = 0, \qquad \Delta_g u = \frac{H}{W}
	\]
	where $\II$ is the second fundamental form of the graph of $u$ in $\R\times M$ and $\overline{\Ric}$ is the Ricci tensor of $\R\times M$. 
	Since
	\[
		H = \MM[u] = -f(u), \qquad g(\nabla^g H,\nabla^g u) = -f'(u)\|\nabla^g u\|^2 \geq 0 \qquad \text{and} \qquad \overline{\Ric} \geq 0,
	\]
	we have
	\begin{equation} \label{DeltaWu}
		\Delta_g \frac{1}{W} + \frac{\|\II\|^2}{W} \leq 0, \qquad \Delta_g u = - \frac{f(u)}{W} \, .
	\end{equation}
	Note that the first inequality implies
	\begin{equation} \label{DeltaW2}
		\Delta_g W \geq \|\II\|^2 W + 2 \frac{\|\nabla^g W\|^2}{W} \, .
	\end{equation}
	
	Let $x_0\in\Omega$ be fixed. Let $C,\eps,\delta>0$ be given, with $C$ and $\delta$ satisfying
	\begin{equation} \label{deltaC}
		\delta < e^{-Cu(x_0)} \, ,
	\end{equation}
	and define
	\[
		\eta = e^{-Cu-\eps r^2} - \delta, \qquad z = W \eta, \qquad \Omega_0 = \{ x \in \Omega : \eta(x) > 0 \}
	\]
	where $r(x) = \dist(x_0,x)$ denotes the Riemannian distance function in $M$ from $x_0$. Since $\eta(x_0) = e^{-Cu(x_0)} - \delta > 0$, we clearly have $x_0 \in \Omega_0 \neq \emptyset$. Since $u\geq 0$, we also have
	\begin{equation} \label{Omega0}
		\Omega_0 \subseteq \{ e^{-\eps r^2} > \delta \} = \{ \eps r^2 < \log(1/\delta) \}
	\end{equation}
	(note that $\log(1/\delta) > 0$ since $\delta<1$ by the constraint \eqref{deltaC} and the assumption $u\geq 0$), hence $\Omega_0$ is relatively compact. By continuity of $z$ in $\overline{\Omega}$, there exists $x_1\in\overline{\Omega}_0$ such that
	\[
		z(x_1) = \max_{\overline{\Omega}_0} z > 0 \, .
	\]
	
	If $x_1\in\partial\Omega$ then we have
	\begin{equation} \label{West1}
		W(x_0) e^{-Cu(x_0)} = z(x_0) \leq z(x_1) \leq W(x_1) \leq \sup_{\partial\Omega} W
	\end{equation}
	where we used that $z \leq W$ since $u\geq 0$. 
	Let us suppose that $x_1\in\Omega$ instead. A direct computation yields
	\[
		W^2 \diver_g(W^{-2} \nabla^g z) = \eta \Delta_g W + W\Delta_g\eta - 2 \eta \frac{\|\nabla^g W\|^2}{W} \qquad \text{in } \, \Omega
	\]
	from which, using also \eqref{DeltaW2}, we infer
	\begin{equation} \label{divWz}
		W^2 \diver_g(W^{-2} \nabla^g z) \geq \left(\|\II\|^2 + \frac{\Delta_g \eta}{\eta} \right) z \qquad \text{in } \, \Omega_0 \, .
	\end{equation}
	By direct computation,
	\begin{equation} \label{Lapeta}
		\Delta_g \eta = (\eta+\delta)(-C\Delta_g u - \eps\Delta_g r^2 + \|C\nabla^g u + \eps\nabla^g r^2\|^2) \qquad \text{in } \, \Omega \, .
	\end{equation}
	By \eqref{Lapg}, \eqref{gij}, the Hessian comparison theorem and $|H| = |f(u)| \leq C_0$ we estimate
	\begin{equation} \label{Lapr2}
		\Delta_g r^2 = g^{ij} (r^2)_{ij} - \frac{H}{W} \langle \nabla u,\nabla r^2\rangle \leq 2n \kappa r \coth(\kappa r) + 2C_0 r \leq C_2(1+r)
	\end{equation}
	where $C_2 = C_2(n,\kappa) > 0$. By \eqref{gij} we have
	\[
		\|\nabla^g u\|^2 = g^{ij} u_i u_j = \frac{|\nabla u|^2}{W^2} = 1 - \frac{1}{W^2}, \qquad \|\nabla^g r\|^2 = g^{ij} r_i r_j \leq \sigma^{ij} r_i r_j = |\nabla r|^2 \leq 1
	\]
	hence, by Young's inequality,
	\begin{equation} \label{CeYoung}
		\|C\nabla^g u + \eps\nabla^g r^2\|^2 \geq \frac{1}{2} C^2\|\nabla^g u\|^2 - \eps^2 \|\nabla^g r^2\|^2 \geq \frac{C^2}{2} (1-W^{-2}) - 4\eps^2 r^2 \, .
	\end{equation}
	From \eqref{Lapeta}, \eqref{Lapr2}, \eqref{CeYoung} and the second in \eqref{DeltaWu} we get
	\[
		\frac{\Delta_g\eta}{\eta} \geq \left(1+\frac{\delta}{\eta}\right)\left(\frac{Cf(u)}{W} + \frac{C^2}{2} \left(1-\frac{1}{W^2}\right) - C_2\eps(1+r) - 4 \eps^2 r^2 \right) \qquad \text{in } \, \Omega_0
	\]
	and therefore, by \eqref{divWz},
	\begin{align*}
		W^2 \diver_g(W^{-2} \nabla^g z) & \geq \left(\|\II\|^2 + \left(1+\frac{\delta}{\eta}\right) \frac{Cf(u)}{W}\right) z \\
		& \phantom{=\;} + \left(1+\frac{\delta}{\eta}\right) \left( \frac{C^2}{2} \left(1-\frac{1}{W^2}\right) - C_2 \eps (1+r) - 4 \eps^2 r^2 \right) z \qquad \text{in } \, \Omega_0 \, .
	\end{align*}
	Since we assumed $x_1\in\Omega_0$ to be an interior maximum point for $z$, we have
	\[
		W^2 \diver_g(W^{-2}\nabla^g z) \leq 0 \qquad \text{at } \, x_1
	\]
	by ellipticity, hence
	\begin{equation} \label{max_bon}
		0 \geq \|\II\|^2 + \left(1+\frac{\delta}{\eta}\right) \frac{Cf(u)}{W} + \left(1+\frac{\delta}{\eta}\right) \left( \frac{C^2}{2} \left(1-\frac{1}{W^2}\right) - C_2 \eps (1+r) - 4 \eps^2 r^2 \right)
	\end{equation}
	at $x_1$. We proceed to suitably bound the various terms on the right hand side of \eqref{max_bon} from below. We first consider the quantity
	\[
		(\ast) \doteq \|\II\|^2 + \left(1+\frac{\delta}{\eta}\right) \frac{Cf(u)}{W} = \|\II\|^2 + \frac{Cf(u)}{W} + \frac{\delta C f(u)}{z} \, .
	\]
	Clearly $(\ast) \geq 0$ if $f(u) \geq 0$. On the other hand, if $f(u) < 0$ we can estimate
	\[
		\|\II\|^2 \geq \frac{(\mathrm{tr}_g \II)^2}{n} = \frac{H^2}{n} = \frac{f(u)^2}{n} \qquad \text{and} \qquad \frac{Cf(u)}{W} \geq \frac{Cf(u)}{z} \, ,
	\]
	where in the second inequality we used the fact that $Cf(u) < 0$ and $z \leq W$, whence
	\[
		(\ast) \geq \frac{f(u)^2}{n} + (1+\delta) \frac{Cf(u)}{z} \geq - \frac{(1+\delta)^2 n C^2}{4z^2} \qquad \text{if } \, f(u) < 0
	\]
	where we used Young's inequality to estimate
	\[
		\frac{f(u)^2}{n} + \frac{(1+\delta)^2 n C^2}{4z^2} \geq - (1+\delta) \frac{Cf(u)}{z} \, .
	\]
	Therefore, since $\delta/\eta>0$ in $\Omega_0$, we have
	\[
		(\ast) \geq \left\{
			\begin{array}{l@{\quad}l}
				0 & \text{if } \, f(u) \geq 0 \\
				- \left(1+\dfrac{\delta}{\eta}\right)\dfrac{(1+\delta)^2 n C^2}{4 z^2} & \text{if } \, f(u) < 0 \, .
			\end{array}
		\right.
	\]
	Since $z\leq W$ we also have
	\[
		\frac{C^2}{2} \left(1-\frac{1}{W^2}\right) \geq \frac{C^2}{2} \left(1-\frac{1}{z^2}\right) .
	\]
	As observed in \eqref{Omega0}, we have $\eps r^2 < \log(1/\delta)$ in $\Omega_0$ and therefore
	\[
		C_2\eps(1+r) + 4 \eps^2 r^2 \leq C_2\left(\eps + \sqrt{\eps \log(1/\delta)}\right) + 4\eps\log(1/\delta) \qquad \text{in } \, \Omega_0 \, .
	\]
	Then, from \eqref{max_bon} we deduce
	\begin{equation} \label{max_bon2}
		\frac{C^2}{2} \left(1-\frac{1}{z^2}\right) - C_2\left(\eps + \sqrt{\eps \log(1/\delta)}\right) - 4\eps\log(1/\delta) \leq \frac{(1+\delta)^2 n C^2}{4 z^2} \qquad \text{at } \, x_1,
	\end{equation}
	that is,
	\[
		1 - \frac{C_\delta}{z(x_1)^2} \leq \frac{2\sqrt{\eps}}{C^2} \left(C_2\sqrt{\log(1/\delta)} + \sqrt{\eps}(C_2+4\log(1/\delta))\right)
	\]
	where $C_\delta = 1 + (1+\delta)^2n/2$. Since $z(x_0) \leq z(x_1)$, we further have
	\begin{equation} \label{max_bon3}
		1 - \frac{C_\delta}{z(x_0)^2} \leq \frac{2\sqrt{\eps}}{C^2} \left(C_2\sqrt{\log(1/\delta)} + \sqrt{\eps}(C_2+4\log(1/\delta))\right) .
	\end{equation}
	(Note that if $f(u) \geq 0$ then \eqref{max_bon2} holds true, more strongly, with $(1+\delta)^2 n C^2 / 4z^2$ replaced by $0$, so \eqref{max_bon3} still holds with $C_\delta$ replaced by $1$.) We emphasize that \eqref{max_bon3} holds for \emph{any} choice of parameters $C,\eps,\delta>0$ provided that \eqref{deltaC} is satisfied and $z$ attains its global maximum at an interior point of $\Omega_0$, so in view of \eqref{West1} we have
	\[
		z(x_0) \leq \max\left\{ A(C,\eps,\delta), \sup_{\partial\Omega} W \right\} \, .
	\]
	where
	\[
		A(C,\eps,\delta) = \sup\left\{ t > 0 : 1 - \frac{C_\delta}{t^2} \leq \frac{2\sqrt{\eps}}{C^2} \left(C_2\sqrt{\log(1/\delta)} + \sqrt{\eps}(C_2+4\log(1/\delta))\right) \right\} .
	\]
	For a fixed $C>0$, letting first $\eps\to0$ and then $\delta\to0$ we get $A(C,\eps,\delta) \to A = \sqrt{1+n/2}$, hence 
	\[
		W(x_0) e^{-Cu(x_0)} = z(x_0) \leq \max\left\{A,\sup_{\partial\Omega} W\right\} .
	\]
	Passing to the limit as $C\to0$ we obtain
	\[
		W(x_0) \leq \max\left\{A,\sup_{\partial\Omega} W\right\},
	\]
	that is,
	\[
		|\nabla u|(x_0) \leq \max\left\{ \sqrt{A^2-1},\,\sup_{\partial\Omega} |\nabla u| \right\} = \max\left\{ \sqrt{n/2},\,\sup_{\partial\Omega} |\nabla u| \right\}
	\]
	as desired.
\end{proof}

%
%
%
%
%
%
%

\bibliographystyle{plain}


%
%

\end{document}